\documentclass{amsart}
\usepackage{amsmath}
\usepackage{paralist}
\usepackage{amsfonts}
\usepackage{amssymb}
\usepackage{amsthm}
\usepackage{amscd}
\usepackage{mathtools}
\usepackage{float}
\usepackage{tikz}
\usepackage{graphicx}
\usepackage[colorlinks=true]{hyperref}
\hypersetup{urlcolor=blue, citecolor=red}
\usepackage{hyperref}
\usepackage{extarrows}

\newtheorem{theorem}{Theorem}[section]

\newtheorem{lemma}[theorem]{Lemma}
\newtheorem{proposition}[theorem]{Proposition}

\theoremstyle{definition}
\newtheorem{definition}[theorem]{Definition}
\newtheorem{remark}[theorem]{Remark}
\newtheorem{assumption}{Assumption}
\newcommand{\ep}{\varepsilon}
\newcommand{\T}{\mathbb{T}}
\newcommand{\R}{\mathbb{R}}
\newcommand{\Z}{\mathbb{Z}}
\newcommand{\N}{\mathbb{N}}
\newcommand{\C}{\mathbb{C}}

\newcommand{\supp}{\operatorname{supp}}
\newcommand{\LR}[1]{{\langle {#1} \rangle }}

\begin{document}

\title[Periodic Zakharov-Kuznetsov equation] 
      {Loomis-Whitney-type inequalities and low regularity well-posedness of the periodic Zakharov-Kuznetsov  equation}

\author{Shinya Kinoshita \and Robert Schippa}
\address{Fakult\"at f\"ur Mathematik, Universit\"at Bielefeld, Postfach 10 01 31, 33501 Bielefeld, Germany}
\email{kinoshita@math.uni-bielefeld.de}
\email{robert.schippa@uni-bielefeld.de}

\begin{abstract}
Local well-posedness for the two-dimensional Zakharov-Kuznetsov equation in the fully periodic case with initial data in Sobolev spaces $H^s$, $s>1$, is proved. Frequency dependent time localization is utilized to control the derivative nonlinearity. The new ingredient to improve on previous results is a nonlinear Loomis-Whitney-type inequality.
\end{abstract}

\maketitle
\normalsize
\section{Introduction}
The purpose of this article is to improve local well-posedness of the Zak\-ha\-rov\-Kuz\-net\-sov equation with periodic boundary conditions in two dimensions
\begin{equation}
\label{eq:PeriodicZakharovKuznetsovEquation}
\left\{\begin{array}{cl}
\partial_t u + (\partial_{x_1}^3 + \partial_{x_1} \partial_{x_2}^2) u &= u \partial_{x_1} u, \quad (t,x) \in \R \times \mathbb{T}^2, \\
u(0) &= u_0 \in H_{\mathcal{R}}^s(\mathbb{T}^2), \end{array} \right.
\end{equation}
where $\mathbb{T} = \R / (2 \pi \Z)$, and $H^s_{\mathcal{R}}$ denotes the Sobolev space with regularity index $s$ comprised of real-valued functions.

By local well-posedness we mean that the data-to-solution mapping $S_T^\infty: H_{\mathcal{R}}^\infty \rightarrow C([0,T],H_{\mathcal{R}}^\infty)$ for $T=T(\Vert u_0 \Vert_{H_{\mathcal{R}}^s})$ assigning smooth, real-valued initial data to smooth, real-valued solutions admits an extension to a continuous mapping $S_T^s : H_{\mathcal{R}}^s \rightarrow C([0,T],H_{\mathcal{R}}^s)$.

The Zakharov-Kuznetsov equation in three dimensions was derived in \cite{ZakharovKuznetsov1974} to describe unidirectional ionic-sonic wave propagation in a magnetized plasma. Laedke and Spatschek derived also the two-dimensional model from the equations of motions for hydrodynamics in \cite{LaedkeSpatschek1982}, which was further justified in \cite{LannesLinaresSaut2013} by Lannes-Linares-Saut.

As a higher-dimensional analog of the Kor\-te\-weg-\-de Vries equation
\begin{equation*}
\partial_t u + \partial_x^3 u = u \partial_x u,
\end{equation*}
\eqref{eq:PeriodicZakharovKuznetsovEquation} has also been extensively studied, and the body of literature is huge. In the following we aim to deliver an overview of the well-posedness theory for \eqref{eq:PeriodicZakharovKuznetsovEquation} in two dimensions.

Conserved quantities for real-valued solutions are the mass 
\begin{equation*}
M(u) = \int u_0^2 dx
\end{equation*}
and energy
\begin{equation*}
E(u) = \int \frac{|\nabla u|^2}{2} + \frac{u^3}{3} dx.
\end{equation*}
In Euclidean space the Zakharov-Kuznetsov equation is invariant under the scaling
\begin{equation*}
u(t,x) \rightarrow \lambda^2 u(\lambda^3 t, \lambda x),
\end{equation*}
which distinguishes $s_c = -1$ as scaling critical regularity.

The classical energy method (cf. \cite{BonaSmith1975}) gives local well-posedness in $H^s_{\mathcal{R}}$, $s > 2$ as well on $\R^2$ as $\T^2$. In Euclidean space this was subsequently improved making use of dispersive effects. In \cite{Faminskii1995} global well-posedness was proved in $H^1_{\mathcal{R}}(\R^2)$. In this work smoothing and maximal function estimates were used to solve the Zakharov-Kuznetsov equation via the contraction mapping principle (cp. \cite{KenigPonceVega1993} for the earlier application in context of the Korteweg-de Vries equation). Linares and Pastor improved local well-posedness to $s>3/4$ in \cite{LinaresPastor2011} by refining the proof in \cite{Faminskii1995}.
In the works \cite{MolinetPilod2015} and \cite{GruenrockHerr2014} due to Molinet-Pilod and Gr\"unrock-Herr, bilinear Strichartz estimates were used to prove local well-posedness for $s>1/2$.

Recently, the first author proved local well-posedness in $H^s(\R^2)$ for $s>-1/4$ in \cite{Kinoshita2019}. The improvement stems from the use of the nonlinear Loomis-Whitney inequality to derive refined multilinear estimates for fully transverse interactions. The result from \cite{Kinoshita2019} is sharp up to endpoints in the sense that the data-to-solution mapping fails to be $C^2$ for $s<-1/4$. The literature for Loomis-Whitney inequalities is vast (see e.g. \cite{LoomisWhitney1949,Carbery2004,BennettCarberyWright2005,
BejenaruHerrTataru2010,BennettBez2010,BejenaruHerr2011,KochSteinerberger2015}); however, for many results on abstract Loomis-Whitney inequalities the application to nonlinear dispersive equations is not clear, as transversality or size of the involved hypersurfaces is not quantified precisely. The nonlinear Loomis-Whitney inequality in $\R^3$ with scalable assumptions on the hypersurfaces was investigated in \cite{BejenaruHerrTataru2010}; see also \cite{BejenaruHerrHolmerTataru2009} for an application to the Zakharov system and \cite{BejenaruHerr2011,BennettBez2010} for subsequent higher-dimensional progress. A strengthened form of the nonlinear Loomis-Whitney inequality is given by multilinear restriction inequalities; see \cite{BennettCarberyTao2006} and the references therein. In \cite{BennettCarberyTao2006}, the dependence on the transversality was not quantified. This was only recently accomplished in three dimensions in \cite{Ramos2018}.

Due to decreased dispersion, the periodic case is worse behaved: in the work \cite{LinaresPantheeRobertTzvetkov2019} by Linares et al. was shown that \eqref{eq:PeriodicZakharovKuznetsovEquation} is not amenable to Picard iteration for $s>1/2$, provided that \eqref{eq:PeriodicZakharovKuznetsovEquation} is locally well-posed at all for $s>1/2$. In fact, local well-posedness was proved for $s>5/3$ by short-time linear Strichartz estimates in \cite{LinaresPantheeRobertTzvetkov2019}.

This was modestly improved by the second author to $s>3/2$ via short-time bilinear Strichartz estimates adapting the bilinear arguments from \cite{GruenrockHerr2014,MolinetPilod2015} to the periodic case in \cite{Schippa2019HDBO}.
Thus, the natural question is to what extend the refined approach from \cite{Kinoshita2019} leads to improved local well-posedness on $\T^2$. We prove the following theorem:
\begin{theorem}
\label{thm:LocalWellposednessZK}
Let $s>1$. Then, we find \eqref{eq:PeriodicZakharovKuznetsovEquation} to be locally well-posed.
\end{theorem}
\begin{remark}
Our result is sensitive with respect to the periods. The proof does not extend to the torus $\sqrt{3} \lambda \T \times \lambda \T$ with $\lambda >0$, but to all ratinal tori. 
We refer to Section \ref{section:ShorttimeEnergyEstimates} for further details. In Section \ref{section:ShorttimeEnergyEstimates} we shall also see that $s =1$ is the limit of our method of frequency dependent time localization and transversality considerations.

Furthermore, the local well-posedness results on $\R^2$ proved via the contraction mapping principle hold for complex initial data. In Section \ref{section:NormInflation} we prove norm inflation for complex initial data on $\T^2$. In the following we consider local well-posedness for \eqref{eq:PeriodicZakharovKuznetsovEquation} implicitly only for real-valued initial data.
\end{remark}
Short-time analysis was used in the periodic case in \cite{LinaresPantheeRobertTzvetkov2019} and \cite{Schippa2019HDBO} as it was pointed out in \cite{LinaresPantheeRobertTzvetkov2019} that \eqref{eq:PeriodicZakharovKuznetsovEquation} is not amenable to Picard iteration. The function spaces used in the present work were introduced for the Euclidean space in \cite{IonescuKenigTataru2008}. The construction in the periodic case will be revisited in Section \ref{section:Notation}. By now there are many works related with frequency dependent time localization. We refer to the expositions in \cite{LinaresPantheeRobertTzvetkov2019,Schippa2019HDBO,Schippa2019PhDThesis} and the references therein for a more complete depiction.
To deal with large initial data, we rescale the torus to handle small initial data on large tori. Thus, we will also consider estimates on tori with arbitrary periods. In the context of short-time analysis this was previously done in \cite{Molinet2012}; see also Section \ref{section:Notation}.

For the proof of Theorem \ref{thm:LocalWellposednessZK} we will show the following sets of estimates. Let $\lambda \geq 1$ denote the period length and $1<s \leq s^\prime$ the regularity and $T \in (0,1]$. Firstly, for smooth solutions $u \in C([0,T],H_\lambda^s)$ emanating from $\lambda$-periodic smooth initial data $u_0 \in H_\lambda^{\infty}$ we find the following estimates to hold:
\begin{equation}
\label{eq:PropagationSolutions}
\left\{ \begin{array}{cl}
\Vert u \Vert_{F^{s^\prime}_{\lambda}(T)} &\lesssim \Vert u \Vert_{E^{s^\prime}_{\lambda}(T)} + \Vert u \partial_{x_1} u \Vert_{N^{s^\prime}_{\lambda}(T)}, \\
\Vert u \partial_{x_1} u \Vert_{N^{s^\prime}_{\lambda}(T)} &\lesssim \Vert u \Vert_{F^{s^\prime}_{\lambda}(T)} \Vert u \Vert_{F^{s}_{\lambda}(T)}, \\
\Vert u \Vert^2_{E^{s^\prime}_{\lambda}(T)} &\lesssim \Vert u_0 \Vert_{H_{\lambda}^{s^\prime}}^2 + \Vert u \Vert_{F^{s}_{\lambda}(T)} \Vert u \Vert^2_{F^{s^\prime}_{\lambda}(T)}.
\end{array} \right.
\end{equation}
By standard bootstrap arguments this proves a priori estimates and persistence of regularity on $[0,T]$ for small initial data in $H_\lambda^s$.

For differences of solutions $v = u_1 - u_2$, with smooth initial data $u_i(0) \in H_\lambda^{\infty}$ and $1<s$, we show
\begin{equation}
\label{eq:PropagationDifferencesL2}
\left\{ \begin{array}{cl}
\Vert v \Vert_{F_{\lambda}^{0}(T)} &\lesssim \Vert v \Vert_{E_{\lambda}^{0}(T)} + \Vert \partial_{x_1} ( v (u_1+u_2)) \Vert_{N_{\lambda}^{0}(T)}, \\
\Vert \partial_{x_1} (v(u_1+u_2)) \Vert_{N_{\lambda}^{0}(T)} &\lesssim \Vert v \Vert_{F_{\lambda}^{0}(T)} \Vert u_1 + u_2 \Vert_{F_{\lambda}^{s}(T)}, \\
\Vert v \Vert^2_{E_{\lambda}^{0}(T)} &\lesssim \Vert v(0) \Vert_{L_{\lambda}^2}^2 + \Vert v \Vert^2_{F_{\lambda}^{0}(T)} ( \Vert u_1 \Vert_{F_{\lambda}^{s}(T)} + \Vert u_2 \Vert_{F_{\lambda}^{s}(T)}).
\end{array} \right.
\end{equation}
This proves Lipschitz-continuous dependence in $L_\lambda^2$ for small initial data in $H^s_\lambda$.

By virtue of the following set of estimates,
\begin{equation}
\label{eq:PropagationDifferencesHs}
\left\{ \begin{array}{cl}
\Vert v \Vert_{F^{s^\prime}_{\lambda}(T)} &\lesssim \Vert v \Vert_{E^{s^\prime}_{\lambda}(T)} + \Vert \partial_{x_1} (v(u_1+u_2)) \Vert_{N^{s^\prime}_{\lambda}(T)} \\
\Vert \partial_{x_1} (v(u_1+u_2)) \Vert_{N^{s^\prime}_{\lambda}(T)} &\lesssim \Vert v \Vert_{F^{s^\prime}_{\lambda}(T)} (\Vert u_1 \Vert_{F^{s}_{\lambda}(T)} + \Vert u_2 \Vert_{F_{\lambda}^{s}(T)}) \\
\Vert v \Vert^2_{E^{s^\prime}_{\lambda}(T)} &\lesssim \Vert v(0) \Vert_{H_{\lambda}^{s^\prime}}^2 + (\Vert v \Vert_{F_{\lambda}^{0}(T)} \Vert v \Vert_{F^{s^\prime}_{\lambda}(T)} \Vert u_2 \Vert_{F^{2s^\prime}_{\lambda}(T)} \\
&\quad \quad + \Vert v \Vert^2_{F^{s^\prime}_{\lambda}(T)} \Vert u_2 \Vert_{F^{s}_{\lambda}(T)} + \Vert v \Vert^2_{F^{s^\prime}_{\lambda}(T)} \Vert v \Vert_{F^{s}_{\lambda}(T)})
\end{array} \right.
\end{equation}
continuous dependence for small initial data in $H^s_\lambda$ follows via the classical Bona-Smith approximation (cf. \cite{BonaSmith1975}). The reduction from arbitrary initial data in $H^s(\T^2)$ to initial data with small Sobolev norm on $\lambda \T^2$ is carried out via scaling. For previous applications of scaling in the context of frequency dependent time localization applied to periodic solutions; see e.g. \cite{Molinet2012,Schippa2019PhDThesis}.

The linear estimate, propagating $u$, $v$, respectively, in $F_{\lambda}^{s}(T)$ is recalled in Section \ref{section:Notation}. The short-time nonlinear estimate propagating the nonlinearity in $N_\lambda^s(T)$ was carried out in \cite{Schippa2019HDBO} and is recalled in Section \ref{section:ShorttimeBilinearEstimates}. 
The first part of Section \ref{section:NLW} is devoted to the global nonlinear Loomis-Whitney inequality on $\R^3$. After that 
Loomis-Whitney-type inequalities on $\R \times$lattices which play a crucial role in the proof of energy estimates are discussed. 
For the energy estimate in Section \ref{section:ShorttimeEnergyEstimates}, the analysis from \cite{Schippa2019HDBO} is refined with the aid of the transversality considerations from \cite{Kinoshita2019}. In Section \ref{section:NormInflation}, we prove norm inflation for periodic complex initial data with arbitrary Sobolev regularity, which is not the case in $\R^2$.

With the above sets of estimates at disposal, the proof of Theorem \ref{thm:LocalWellposednessZK} is concluded by standard bootstrap arguments, which are omitted. For details, we refer to \cite{Schippa2019HDBO}.

\section{Notation}
\label{section:Notation}
Dyadic numbers will be denoted by capital letters $N \in 2^{\mathbb{N}_0}$, where $\mathbb{N}_0 = \mathbb{N} \cup \{ 0 \}$. For $\xi \in \R^n$ let $|\xi| = \sqrt{\xi_1^2+\ldots+\xi_n^2}$ denote the Euclidean norm and $\langle \xi \rangle ^2 = 1+|\xi|^2$. Set $\T = \R/(2\pi \Z)$ and for $\lambda \geq 1$ set $\lambda \T^n = \lambda \T \times \ldots \times \lambda \T$ and $\Z^n / \lambda = \Z / \lambda \times \ldots \times \Z / \lambda$. Varying $\lambda$ we have to be aware of possible dependencies of constants on the spatial scale.
Let $(d\xi)_\lambda$ be the normalized counting measure on $\Z^n/\lambda$:
\begin{equation*}
\int a(\xi) (d\xi)_\lambda := \lambda^{-n} \sum_{\xi \in \Z^n/\lambda} a(\xi).
\end{equation*}
The Fourier transform on $\lambda \T^n$ is defined for $f \in L^1(\lambda \T^n; \C)$ by
\begin{equation*}
\hat{f}(k) = \int_{\lambda \T^n} e^{-i k.x} f(x) dx, \quad k \in \mathbb{Z}^n/\lambda.
\end{equation*}
The inverse Fourier transform is given by
\begin{equation*}
\check{g}(x) = \frac{1}{(2 \pi)^n} \int g(\xi) e^{i x.\xi} (d\xi)_\lambda.
\end{equation*}
The usual properties like Plancherel's theorem or Parseval's identity of the Fourier transform hold. We refer to \cite[p.~727]{CollianderKeelStaffilaniTakaokaTao2003} for further properties.

Let $\chi: \R \rightarrow \R_{\geq 0}$ denote a smooth symmetric function, supported in $[-7/8,7/8]$ with $\chi \equiv 1$ on $[-5/4,5/4]$ and set $\chi_k(\xi) = \chi(2^{-k} |\xi|) - \chi(2^{1-k} |\xi|)$ for $k \in \mathbb{N}$.
Note that
\begin{equation*}
\sum_{k=1}^\infty \chi_k(\xi) + \chi(\xi) \equiv 1.
\end{equation*}

For $N = 2^n$, $n \in \mathbb{N}_0$ we denote by $P_N$ the Littlewood-Paley projector associated with $\chi_n$, i.e., 
\begin{equation*}
(P_N f) \widehat (\xi) = \chi_n(|\xi|) \hat{f}(\xi).
\end{equation*}.
We define Sobolev spaces for $s \geq 0$ as 
\begin{equation*}
H^s(\lambda \T^n) = \{ f \in L^2(\lambda \T^n) \; | \; \Vert f \Vert_{H^s_\lambda}^2 = \int \langle \xi \rangle^{2s} |\hat{f}(\xi)|^2 (d\xi)_\lambda < \infty \}
\end{equation*}
and $H^\infty(\lambda \T^n) = \bigcap_{s \geq 0} H^s(\lambda \T^n)$.

We turn to the definition of the short-time $X^{s,b}$-spaces. Let $\eta_0: \R \rightarrow [0,1]$ denote an even, smooth function with $\eta_0 \equiv 1$ on $[-5/4,5/4]$ on $\text{supp} \, (\eta_0) \subseteq [-7/8,7/8]$. For $k \in \mathbb{N}$ we set
\begin{equation*}
\eta_k(\tau) = \eta_0(\tau/2^k) - \eta_0(\tau/2^{k-1}).
\end{equation*}
We write $\eta_{\leq m} = \sum_{j=0}^m \eta_j$ for $m \in \mathbb{N}$.

Set 
\begin{equation*}
A_k = 
\begin{cases}
\{ \xi \in \R^n \; | \; |\xi| \lesssim 1 \}, \quad k=0, \\
\{ \xi \in \R^n \; | \; |\xi| \sim 2^k \}, \quad k \geq 1,
\end{cases}
\end{equation*}
and for the dispersion relation $\varphi(\xi,\eta) = \xi^3 + \xi \eta^2$, $N,L \in 2^{\mathbb{N}_0}$ 
\begin{equation}
\label{eq:NotationFrequencyModulationLocalization}
\begin{split}
G_{N,L} &= \{ (\tau,\xi) \in \mathbb{R} \times \R^2 \, | \, |\xi| \sim N, \,   |\tau - \varphi(\xi)| \sim L \}, \\
G_{N,\leq L} &= \{ (\tau,\xi) \in \mathbb{R} \times \R^2 \, | \, |\xi| \sim N, \, |\tau - \varphi(\xi)| \leq L \}.
\end{split}
\end{equation}

Next, we define an $X^{s,b}$-type space for the Fourier transform of frequency-localized space-periodic functions:
\begin{equation*}
\begin{split}
\label{eq:XkDefinition}
&X_{k,\lambda} = \{ f:  \mathbb{R} \times \mathbb{Z}^n/\lambda \rightarrow \mathbb{C} \; | \\
 &\mathrm{supp}(f) \subseteq \R \times {A_k}, \Vert f \Vert_{X_{k,\lambda}} = \sum_{j=0}^\infty 2^{j/2} \Vert \eta_j(\tau - \varphi(\xi)) f(\tau,\xi) \Vert_{L^2_{(d\xi)_\lambda} L^2_{\tau}} < \infty \} .
\end{split}
\end{equation*}

Partitioning the modulation variable through a sum over $\eta_j$ yields the estimate
\begin{equation}
\label{eq:XkEstimateI}
\Vert \int_{\mathbb{R}} | f_k(\tau^\prime,\xi) | d\tau^\prime \Vert_{L^2_{(d\xi)_\lambda}} \lesssim \Vert f_k \Vert_{X_{k,\lambda}}. 
\end{equation}

Also, we record the estimate
\begin{equation}
\label{eq:XkEstimateII}
\begin{split}
&\quad \sum_{j=l+1}^{\infty} 2^{j/2} \Vert \eta_j(\tau - \varphi(\xi)) \cdot \int_{\mathbb{R}} | f_k(\tau^\prime,\xi) | \cdot 2^{-l} (1+ 2^{-l}|\tau - \tau^\prime |)^{-4} d\tau^\prime \Vert_{L^2_{(d\xi)_\lambda} L^2_\tau} \\
&\quad + 2^{l/2} \Vert \eta_{\leq l} (\tau - \varphi(\xi)) \cdot \int_{\mathbb{R}} | f_k(\tau^\prime,\xi) | \cdot 2^{-l} (1+ 2^{-l}|\tau - \tau^\prime |)^{-4} d\tau^\prime \Vert_{L^2_{(d\xi)_\lambda} L^2_\tau }\\
&\lesssim \Vert f_k \Vert_{X_{k,\lambda}},
\end{split}
\end{equation}
which is a rescaled version of \cite[Equation~(3.5)]{GuoOh2018}.

In particular, we find for a Schwartz-function $\gamma$ for $k, l \in \mathbb{N}, t_0 \in \mathbb{R}, f_k \in X_{k,\lambda}$ the estimate
\begin{equation}
\label{eq:XkEstimateIII}
\Vert \mathcal{F}[\gamma(2^l(t-t_0)) \cdot \mathcal{F}^{-1}(f_k)] \Vert_{X_{k,\lambda}} \lesssim_{\gamma} \Vert f_k \Vert_{X_{k,\lambda}}.
\end{equation}

We define the spaces
\begin{equation*}
E_{k,\lambda} = \{ u_0 : \lambda \mathbb{T}^n \rightarrow \mathbb{C} \, | \, \text{supp} (\hat{u}_0) \subseteq A_k, \; \Vert u_0 \Vert_{E_{k,\lambda}} = \Vert u_0 \Vert_{L_\lambda^2} < \infty \},
\end{equation*}
which are the spaces for the dyadically localized energy.\\
Next, we set
\begin{equation*}
C_0 (\mathbb{R}, E_{k,\lambda}) = \{ u_k \in C(\mathbb{R},E_{k,\lambda}) \; | \; \mathrm{supp}(u_k) \subseteq [-4,4] \}
\end{equation*}
and define for a frequency $2^k$ the following short-time $X^{s,b}$-space:
\begin{equation*}
F_{k,\lambda} = \{ u_k \in C_0(\mathbb{R}, E_{k,\lambda}) \; | \Vert u_k \Vert_{F_{k,\lambda}} = \sup_{t_k \in \mathbb{R}} \Vert \mathcal{F}[u_k \eta_0(2^{ k}(t-t_k))] \Vert_{X_{k,\lambda}} < \infty \}.
\end{equation*}
The frequency dependent time localization for frequencies $N \in 2^{\mathbb{N}_0}$ is $T(N) = N^{-1}$. This allows us to overcome the derivative loss in the nonlinear estimate (cf. \cite{Schippa2019HDBO}).
Similarly, we define the spaces to capture the nonlinearity:
\begin{equation*}
\begin{split}
&N_{k,\lambda} = \{ u_k \in C_0(\mathbb{R}, E_{k,\lambda}) \; | \\ 
&\Vert u_k \Vert_{N_{k,\lambda}} = \sup_{t_k \in \mathbb{R}} \Vert (\tau - \varphi(\xi) + i2^{ k})^{-1} \mathcal{F}[u_k \eta_0(2^{ k}(t-t_k))] \Vert_{X_{k,\lambda}} < \infty \}.
\end{split}
\end{equation*}

We localize the spaces in time in the usual way. For $T \in (0,1]$ we set
\begin{equation*}
F_{k,\lambda}(T) = \{ u_k \in C([-T,T], E_{k,\lambda}) \; | \Vert u_k \Vert_{F_{k,\lambda}(T)} = \inf_{\tilde{u}_k = u_k \mathrm{in} [-T,T] } \Vert \tilde{u}_k \Vert_{F_{k,\lambda}} < \infty \}
\end{equation*}
and
\begin{equation*}
N_{k,\lambda}(T) = \{ u_k \in C([-T,T], E_{k,\lambda}) \; | \Vert u_k \Vert_{N_{k,\lambda}(T)} = \inf_{\tilde{u}_k = u_k \mathrm{in} [-T,T]} \Vert \tilde{u}_k \Vert_{N_{k,\lambda}} < \infty \}.
\end{equation*}

We assemble the spaces for dyadically localized frequencies in a straight-forward manner using Littlewood-Paley theory: as an energy space for solutions we consider
\begin{equation*}
\begin{split}
E_\lambda^{s}(T) &= \{ u \in C([-T,T],H_\lambda^{\infty}) \; | \; \\
\Vert u \Vert_{E_\lambda^{s}(T)}^2 &= \Vert P_{\leq 0} u(0) \Vert_{L_\lambda^2}^2 + \sum_{N \in 2^{\mathbb{N}}} \sup_{t_k \in [-T,T]} N^{2s} \Vert P_N u(t_k) \Vert_{L_\lambda^2}^2 < \infty \}.
\end{split}
\end{equation*}
We define the short-time $X^{s,b}$-space for the solution
\begin{equation*}
F_\lambda^{s}(T) = \{ u \in C([-T,T],H_\lambda^{\infty}) \; | \Vert u \Vert_{F_\lambda^{s}(T)}^2 = \sum_{N = 2^n, n \in \mathbb{N}_0} N^{2s} \Vert P_N u \Vert_{F_{n,\lambda}(T)}^2 < \infty \} ,
\end{equation*}
and for the nonlinearity we consider
\begin{equation*}
N^{s}_\lambda(T) = \{ u \in C([-T,T],H_\lambda^{\infty}) \; |  \Vert u \Vert_{N_\lambda^{s}(T)}^2 = \sum_{N = 2^n, n \in \mathbb{N}_0} N^{2s} \Vert P_n u \Vert_{N_{n,\lambda}(T)}^2 < \infty \}.
\end{equation*}

We also make use of $k$-acceptable time multiplication factors (cf. \cite{IonescuKenigTataru2008}): for $k \in \mathbb{N}_0$ we set
\begin{equation*}
S_k = \{ m_k \in C^{\infty}(\mathbb{R},\mathbb{R}) : \; \Vert m_k \Vert_{S_k} = \sum_{j=0}^{10} 2^{-j k} \Vert \partial^j m_k \Vert_{L^{\infty}} < \infty \}.
\end{equation*}
The generic example is given by time localization on a scale of $2^{-k}$, i.e., $\eta_0(2^{ k} \cdot)$.\\ 
The estimates (cf. \cite[Eq.~(2.21),~p.~273]{IonescuKenigTataru2008})
\begin{equation}
\label{eq:timeLocalizationShorttimeNorms}
\begin{split}
\left\{\begin{array}{cl}
\Vert \sum_{k \geq 0} m_k(t) P_{2^k}(u) \Vert_{F_\lambda^{s}(T)} \lesssim (\sup_{k \geq 0} \Vert m_k \Vert_{S_k}) \cdot \Vert u \Vert_{F_\lambda^{s}(T)}, \\
\Vert \sum_{k \geq 0} m_k(t) P_{2^k}(u) \Vert_{N_\lambda^{s}(T)} \lesssim (\sup_{k \geq 0} \Vert m_k \Vert_{S_k}) \cdot \Vert u \Vert_{N_\lambda^{s}(T)}, \\
\Vert \sum_{k \geq 0} m_k(t) P_{2^k}(u) \Vert_{E_\lambda^{s}(T)} \lesssim (\sup_{k \geq 0} \Vert m_k \Vert_{S_k}) \cdot \Vert u \Vert_{E_\lambda^{s}(T)}
\end{array} \right.
\end{split}
\end{equation}
follow from integration by parts.\\
From \eqref{eq:timeLocalizationShorttimeNorms} follows that we can assume $F_{k,\lambda}(T)$ functions to be supported in time on an interval $[-T-2^{-k-10},T+2^{-k-10}]$.

We record basic properties of the short-time $X^{s,b}$-spaces introduced above. The next lemma establishes the embedding $F_\lambda^{s}(T) \hookrightarrow C([0,T],H_\lambda^s)$.
\begin{lemma}
\label{lem:FsEmbedding}
\begin{enumerate}
\item[(i)]
We find the estimate
\begin{equation*}
\Vert u \Vert_{L_t^\infty L_\lambda^2} \lesssim \Vert u \Vert_{F_{k,\lambda}}
\end{equation*}
to hold for any $u \in F_{k,\lambda}$ with implicit constant independent of $k$ and $\lambda$.
\item[(ii)]
Suppose that $s \in \R$, $T>0$ and $u \in F_\lambda^{s}(T)$. Then, we find the estimate
\begin{equation*}
\Vert u \Vert_{C([0,T],H_\lambda^s)} \lesssim \Vert u \Vert_{F_\lambda^{s}(T)}
\end{equation*}
to hold.
\end{enumerate}
\end{lemma}
\begin{proof}
For a proof, see \cite[Lemma~3.1.,~p.~274]{IonescuKenigTataru2008} in Euclidean space and \\
\cite[Lemma~3.2,~3.3]{GuoOh2018} in the periodic case. Independence of the period length follows from rescaling.
\end{proof}
We state the energy estimate for the above short-time $X^{s,b}$-spaces. The proof, which is carried out on the real line in \cite[Proposition~3.2.,~p.~274]{IonescuKenigTataru2008} and in the periodic case in \cite[Proposition~4.1.]{GuoOh2018}, is omitted.
\begin{proposition}
\label{prop:linearShorttimeEnergyEstimate}
Let $T \in (0,1]$, $\lambda \geq 1$ and $u, v \in C([-T,T],H_\lambda^{\infty})$ satisfy the equation
\begin{equation*}
\partial_t u + (\partial_{x_1}^3 + \partial_{x_1} \partial_{x_2}^2) u =v \; \mathrm{ in } \; (-T,T) \times \lambda \mathbb{T}^2.
\end{equation*}
Then, we find the following estimate to hold for any $s \in \mathbb{R}$ with implicit constant independent of $s,T$ and $\lambda$:
\begin{equation*}
\Vert u \Vert_{F_\lambda^{s}(T)} \lesssim \Vert u \Vert_{E_\lambda^{s}(T)} + \Vert v \Vert_{N_\lambda^{s}(T)}.
\end{equation*}
\end{proposition}
Below we have to consider the action of sharp time cutoffs in the $X_k$-spaces. Recall from the usual $X^{s,b}$-space-theory that multiplication with a sharp cutoff in time is not bounded. However, we find the following estimate to hold:
\begin{lemma}{\cite[Lemma~3.5]{GuoOh2018}}
\label{lem:sharpTimeCutoffAlmostBounded}
Let $N = 2^n, \; n \in \mathbb{N}_0$ and $\lambda \geq 1$. Then, for any interval $I=[t_1,t_2] \subseteq \R$, we find the following estimate to hold:
\begin{equation*}
\sup_{j \geq 0} 2^{j/2} \Vert \eta_j(\tau-\varphi(\xi)) \mathcal{F}_{t,x}[1_{I}(t) P_N u] \Vert_{L_\tau^2 L^2_{(d\xi)_\lambda}} \lesssim \Vert \mathcal{F}_{t,x} (P_N u) \Vert_{X_{n,\lambda}}
\end{equation*}
with implicit constant independent of $n, \lambda$ and $I$.
\end{lemma}
\section{Short-time nonlinear estimates}
\label{section:ShorttimeBilinearEstimates}
We recall short-time nonlinear estimates on $\lambda \T^2$ from \cite{Schippa2019HDBO} for $s>1/2$.

\begin{proposition}
\label{prop:ShorttimeNonlinearEstimates}
Let $\lambda \geq 1$, $ T \in (0,1]$, $1/2<s \leq s^\prime$. Then, we find the following estimates to hold for $u_1$, $u_2 \in F_\lambda^s(T)$:
\begin{align}
\label{eq:ShorttimeNonlinearEstimateI}
\Vert \partial_{x_1} (u_1 u_2) \Vert_{N_\lambda^{s}(T)} &\lesssim \Vert u_1 \Vert_{F_\lambda^{s}(T)} \Vert u_2 \Vert_{F_\lambda^{s}(T)}, \\
\label{eq:ShorttimeNonlinearEstimateII}
\Vert \partial_{x_1} (u_1 u_2) \Vert_{N_\lambda^{0}(T)} &\lesssim \Vert u_1 \Vert_{F_\lambda^{0}(T)} \Vert u_2 \Vert_{F_\lambda^{s}(T)}.
\end{align}
\end{proposition}
\begin{proof}
The proof for $\lambda = 1$ is given in \cite[Prop.~7.5]{Schippa2019HDBO}; the general case follows from rescaling.
\end{proof}

\section{Global nonlinear Loomis-Whitney inequalities}
\label{section:NLW}
In this section, global nonlinear Loomis-Whitney inequalities are discussed. After globalizing local results in $\R^3$, we turn to Loomis-Whitney-type inequalities on $\R \times \Z^2/N$. The arguments from considering Euclidean space will be useful on $\R \times \Z^2/N$.

\subsection{Loomis-Whitney inequalities on $\R^3$}

For $i=1,2,3$, letting \\
$S_i:=\{(x_1,x_2,x_3) \in \R^3 \, | \, x_i=0\}$, the classical Loomis-Whitney inequality in $\R^3$ is described as
\[
\| f_1 * f_2 \|_{L^2 (S_3)} \leq \|f_1 \|_{L^2(S_1)} \| f_2 \|_{L^2(S_2)}.
\]
 Note that the case of fully transverse hyperplanes, quantified by $A$ in Assumption \ref{AssumptionSurfaces}.(iii), is recovered by a change of variables, and we find the above estimate to hold with constant $A^{1/2}$.
If $S_1$, $S_2$, $S_3$ are oriented hypersurfaces in $\R^3$, then the above is called the nonlinear Loomis-Whitney inequalities in $\R^3$. 

The estimate for fully transverse hyperplanes was extended to $C^3$-hypersurfaces in
\cite{BennettCarberyWright2005} by Bennett-Carbery-Wright. 
Then, Bejenaru-Herr-Tataru relaxed the regularity conditions of the 
hypersurfaces in \cite{BejenaruHerrTataru2010} by employing induction on scales. In these results, the transversality of the oriented hypersurfaces determines the constant for which the estimate from the above display holds, which matches the case of hyperplanes. The constant also depends on regularity properties of the surfaces; see Assumption \ref{AssumptionSurfaces} below. Furthermore, the results from \cite{BennettCarberyWright2005,BejenaruHerrTataru2010} are local, i.e., these are only stated for bounded hypersurfaces.

Nonlinear Loomis-Whitney inequalities yield smoothing effects in Euclidean space related to bilinear Strichartz estimates.
This cannot hold on $\R \times \Z^2/N$, due to the discrete nature of the counting measure. One can well have a fully transverse interaction of three frequencies on a lattice, which cannot yield any smoothing effect. See the end of this section for an example.

Before turning to Loomis-Whitney-type inequalities on $\R \times \Z^2/N$, we shall see how to remove the locality assumption in Euclidean space. The underlying argument will be crucial to handle the discrete case. Our argument is related to a recent work by Koch-Steinerberger \cite{KochSteinerberger2015}. In \cite[Theorem~2.1,~p.~1226]{KochSteinerberger2015} a global result for hypersurfaces described as Lipschitz graphs is given. However, the stated dependence on the transversality constant is worse than in the case of hyperplanes in \cite{KochSteinerberger2015}.\\
The argument in \cite{KochSteinerberger2015} does not make use of induction on scales, contrary to \cite{BejenaruHerrTataru2010}, but relies entirely on suitable decompositions and almost orthogonality. Our proof is also based on decompositions of the hypersurfaces and almost orthogonality. We improve the dependence on the transversality given in \cite{KochSteinerberger2015} for hypersurfaces slightly more regular than Lipschitz, which we do not cover.

In the following we consider $C^{1,\beta}$-hypersurfaces given as rotated graphs of $C^{1,\beta}$-functions following \cite[Assumption~1.1]{BejenaruHerr2011}.
\begin{assumption}\label{AssumptionSurfaces}
For $i=1,2,3$ there exist $0 < \beta \leq 1$, $b > 0$, $A \geq 1$, $F_i \in C^{1,\beta}(\mathcal{U}_i)$, where the $\mathcal{U}_i$ denote open and convex sets in $\R^2$ and $G_i \in O(3)$ such that
\begin{enumerate}
\item[(i)] the oriented surfaces $S_i$ are given by
\begin{equation*}
S_i = G_i gr(F_i), \quad 
gr(F_i) = \{ (x,y,z) \in \R^3 \; | \; z = F_i(x,y), \; (x,y) \in \mathcal{U}_i \}.
\end{equation*}
\item[(ii)] the unit normal vector field $\mathfrak{n}_i$ on $S_i$ satisfies the H\"older condition
\begin{equation}
\label{eq:HoelderConditionUnitNormals}
\sup_{\sigma, \tilde{\sigma} \in S_i} \frac{|\mathfrak{n}_i(\sigma) - \mathfrak{n}_i(\tilde{\sigma})|}{|\sigma - \tilde{\sigma}|^\beta} + \frac{|\mathfrak{n}_i(\sigma)(\sigma - \tilde{\sigma})|}{|\sigma - \tilde{\sigma}|^{1+\beta
}} \leq b;
\end{equation}
\item[(iii)] the matrix $N(\sigma_1,\sigma_2,\sigma_3) = (\mathfrak{n}_1(\sigma_1),\mathfrak{n}_2(\sigma_2),\mathfrak{n}_3(\sigma_3))$ satisfies the transversality condition
\begin{equation}
\label{eq:TransversalityAssumption}
A^{-1} \leq \det N(\sigma_1,\sigma_2,\sigma_3) \leq 1
\end{equation}
for all $(\sigma_1,\sigma_2,\sigma_3) \in S_1 \times S_2 \times S_3$.
\end{enumerate}
\end{assumption}
Under Assumption \ref{AssumptionSurfaces}, we establish the nonlinear Loomis-Whitney inequalities without locality assumptions on $S_i$. 
\begin{theorem}
\label{thm:GlobalNonlinearLoomisWhitneyInequalityR3}
Suppose that $(S_i)_{i=1}^3$ satisfies Assumption \ref{AssumptionSurfaces}. Then, for each $f \in L^2(S_1)$ and $g \in L^2(S_2)$, we have
\begin{equation*}
\Vert f * g \Vert_{L^2(S_3)} \leq C A^{1/2} \Vert f \Vert_{L^2(S_1)} \Vert g \Vert_{L^2(S_2)},
\end{equation*}
where the constant $C>0$ is independent of $\beta$ and $b$.
\end{theorem}

Note that in \cite{BejenaruHerrTataru2010}, though the sharp dependence on $A$ is obtained, the constant $C$ in the above display depends on $\beta$ and $b$.
To begin with the proof, we see how we can quantify the overlap of thickened hypersurfaces. 
We write $S_i(\varepsilon) = G_i \{ (x,y,z) \in \mathcal{U}_i \times \R \;| \; |z-F_i(x,y)| < \varepsilon \}$ with notations from above and define $\chi_M$ as the characteristic function of a set $M$.
\begin{proposition}
\label{prop:IntersectionThickenedHypersurface}
Suppose that $(S_i)_{i=1}^3$ satisfies Assumption \ref{AssumptionSurfaces}. Then, for $\varepsilon >0$, 
the following estimate holds true:
\begin{equation}
\label{eq:IntersectionThickenedHypersurfaces}
\int_{\R^3} \chi_{S_1(\varepsilon)}(x) \chi_{S_2(\varepsilon)}(x) \chi_{S_3(\varepsilon)}(x) dx \lesssim A \varepsilon^3,
\end{equation}
where the implicit constant is independent of $\beta$ and $b$.
\end{proposition}

\begin{proof}
Clearly, by the definitions of $S_i(\varepsilon)$, we may assume that $\varepsilon= \varepsilon (A,\beta,b)$ is sufficiently small. 
We start with the elementary case that $S_i$ are three transverse hyperplanes $(H_i)_{i=1}^3$. The estimate
\begin{equation*}
\int_{\R^3} \chi_{H_1(\varepsilon)}(x) \chi_{H_2(\varepsilon)}(x) \chi_{H_3(\varepsilon)}(x) dx \lesssim A \varepsilon^3
\end{equation*}
follows from a linear change of variables, mapping the normals of the hyperplanes to the unit matrix.

We turn to the nonlinear case. Let $p \in S_1(\varepsilon) \cap S_2(\varepsilon) \cap S_3(\varepsilon)$, as for an empty intersection there is nothing to show. We observe that
\begin{equation*}
\int_{\R^3} \chi_{S_1(\varepsilon)}(x) \chi_{S_2(\varepsilon)}(x) \chi_{S_3(\varepsilon)}(x) dx = \int_{B(100 A \varepsilon,p)}  \chi_{S_1( \varepsilon)}(x) \chi_{S_2(\varepsilon)}(x) \chi_{S_3(\varepsilon)}(x) dx.
\end{equation*}
To confine the range of integration to $B(100 A \varepsilon,p)$, suppose that there is $q \in S_1(\varepsilon) \cap S_2(\varepsilon) \cap S_3(\varepsilon)$ with $d(p,q) \geq 100 A \varepsilon$. For $i=1,2,3$ we can find $q_i \in S_i$ with $d(q_i,q) \leq \varepsilon$ and $p_i \in S_i$ with $d(p_i,p) \leq \varepsilon$. We have $d(q_i,p_i) \geq 98 A\varepsilon$. By the mean-value theorem, we find a normal vector $\mathfrak{n}_i$ of $S_i$ with $\mathfrak{n}_i \perp p_i - q_i$. It is straight-forward to check that $|\det(\mathfrak{n}_1,\mathfrak{n}_2,\mathfrak{n}_3)| \ll A^{-1}$. This contradiction allows us to bound the domain of integration like above.

To reduce the nonlinear case to the case of hyperplanes, we shall approximate $S_i(\varepsilon) \cap B(100 \varepsilon,p)$ with $T_{p_i} S_i(C \varepsilon)$, $p_i \in S_i$, $d(p_i,p) \leq \varepsilon$. Here, $T_p S$ denotes the tangent space at $S$, as a subset of $\R^3$.
Observe that by the $C^{1,\beta}$-property Assumption \ref{AssumptionSurfaces}.(ii), $\lambda_i \in S_i$ satisfy the estimate
\begin{equation*}
 |\mathfrak{n}_i(p_i) \cdot (p_i- \lambda_i)| \leq b |p_i - \lambda_i|^{1+\beta}.
\end{equation*}
 For $\varepsilon \ll b^{-1/\beta} A^{-(1+1/\beta)}$, we find $T_{p_i} S_i(C \varepsilon) \supseteq B(100 A \varepsilon,p) \cap S_i(\varepsilon)$.

To finish the proof, we estimate by our considerations in the case of hyperplanes
\begin{equation*}
\int_{B(100 A \varepsilon,p)} \chi_{T_{p_1} S_1(C \varepsilon)}(x) \chi_{T_{p_2} S_2(C \varepsilon)}(x) \chi_{T_{p_3} S_3(C \varepsilon)}(x) dx \lesssim A \varepsilon^3.
\end{equation*}
This completes the proof.
\end{proof}
By the above proposition, we show the following global nonlinear Loomis-Whitney inequality for thickened hypersurfaces. This will allow us to remove the locality assumption from \cite{BejenaruHerrTataru2010} by taking the thickness to zero in the next subsection.
\begin{theorem}
\label{thm:GlobalNonlinearLoomisWhitneyInequality}
Let $A$ be dyadic and $f_i \in L^2(S_i(\varepsilon))$, $i=1,2$. Suppose that $(S_i)_{i=1}^3$ satisfies Assumption \ref{AssumptionSurfaces}. Then, for $\varepsilon>0$ we find the following estimate to hold
\begin{equation}\label{eq:Theorem4.3}
\Vert f_1 * f_2 \Vert_{L^2(S_3(\varepsilon))} \lesssim \varepsilon^{3/2} A^{1/2} \Vert f_1 \Vert_{L^2(S_1(\varepsilon))} \Vert f_2 \Vert_{L^2(S_2(\varepsilon))},
\end{equation}
where the implicit constant is independent of $\beta$ and $b$.
\end{theorem}
\begin{proof}
Let $\{ B_{\varepsilon,j} \}_{j \in \mathbb{N}}$ denote a finitely overlapping family of balls with radius $\varepsilon$ covering $\R^3$. 
Set $J_{3,\varepsilon} = \{ j \in \mathbb{N} \; | \; B_{\varepsilon,j} \cap S_3(\varepsilon) \neq \emptyset \}$ and $S_{3,j}(\varepsilon) = B_{\varepsilon,j} \cap S_3(\varepsilon)$. We break up the support of $f_3$ into $\{ S_{3,j}(\varepsilon) \}_{j \in J_{3,\varepsilon}}$ to get
\begin{equation*}
\left| \int_{\R^3} (f_1 * f_2)(\lambda) f_3(\lambda) d\lambda \right| \leq \sum_{j \in J_{3,\varepsilon}} \left| \int_{\R^3} (f_1 * f_2)(\lambda) f_3 \vert_{S_{3,j}(\varepsilon)}(\lambda) d\lambda \right|.
\end{equation*}
We use the Cauchy-Schwarz inequality to estimate the single contributions
\begin{equation*}
\left| \int_{\R^3} (f_1 * f_2)(\lambda) f_3 \vert_{S_{3,j}(\varepsilon)}(\lambda) d\lambda \right| \lesssim \varepsilon^{3/2} \Vert f_1 \Vert_{L^2(S_{1,j,\varepsilon})} \Vert f_2 \Vert_{L^2(S_{2,j,\varepsilon})} \Vert f_3 \Vert_{L^2(S_{3,j}(\varepsilon))},
\end{equation*}
where
\begin{equation*}
\begin{split}
S_{1,j,\varepsilon} &= \{ \lambda_1 \in S_1(\varepsilon) \; | \; \exists \lambda^\prime \in B_{\varepsilon,j}: \; \lambda^\prime - \lambda_1 \in S_2(\varepsilon) \}, \\
 S_{2,j,\varepsilon} &= \{ \lambda_2 \in S_2(\varepsilon) \; | \; \exists \lambda^\prime \in B_{\varepsilon,j}: \; \lambda^\prime - \lambda_2 \in S_1(\varepsilon) \}.
\end{split}
\end{equation*}
By the Cauchy-Schwarz inequality, for all $(\lambda_1,\lambda_2) \in S_1(\varepsilon) \times S_2(\varepsilon)$, it suffices to show
\begin{equation}
\label{eq:RelevantJ}
\sum_{j \in J_{3,\varepsilon}} \chi_{S_{1,j,\varepsilon} \times S_{2,j,\varepsilon}}(\lambda_1,\lambda_2) \lesssim A.
\end{equation}
Indeed, assuming \eqref{eq:RelevantJ}, we conclude
\begin{equation*}
\begin{split}
&\quad \sum_j \left| \int_{\R^3} (f_1 * f_2)(\lambda) f_3 \vert_{S_{3,j}(\varepsilon)}(\lambda) d\lambda \right| \\
 &\lesssim \varepsilon^{3/2} \sum_j \Vert f_1 \Vert_{L^2(S_{1,j,\varepsilon})} \Vert f_2 \Vert_{L^2(S_{2,j,\varepsilon})} \Vert f_3 \Vert_{L^2(S_{3,j}(\varepsilon))} \\
&\lesssim \varepsilon^{3/2} \left( \sum_j \Vert f_3 \Vert_{L^2(S_{3,j}(\varepsilon))}^2 \right)^{1/2} \left( \sum_j \Vert f_1 \Vert^2_{L^2(S_{1,j,\varepsilon})} 
{\Vert f_2 \Vert^2_{L^2(S_{2,j,\varepsilon})}}^2 \right)^{1/2} \\
&\lesssim \varepsilon^{3/2} \Vert f_3 \Vert_{L^2(S_3(\varepsilon))} \left( \sum_j \Vert f_1 \chi_{S_{1,j,\varepsilon}} \Vert_{L^2}^2 \Vert f_2 \chi_{S_{2,j,\varepsilon}} \Vert_{L^2}^2 \right)^{1/2} \\
&\lesssim \varepsilon^{3/2} \Vert f_3 \Vert_{L^2(S_3(\varepsilon))} \left( \int_{\R^3 \times \R^3} \sum_j \chi_{S_{1,j,\varepsilon} \times S_{2,j,\varepsilon}}(\lambda_1,\lambda_2) |f_1(\lambda_1)|^2 |f_2(\lambda_2)|^2 d\lambda_1 d\lambda_2 \right)^{1/2} \\
&\lesssim \varepsilon^{3/2} A^{1/2} \prod_{i=1}^3 \Vert f_i \Vert_{L^2(S_i(\varepsilon))}.
\end{split}
\end{equation*}
Thus, the remainder of the proof is devoted to the proof of \eqref{eq:RelevantJ}.

Without loss of generality, we may assume that there exists $j_0 \in J_{3,\varepsilon}$ such that $(\lambda_1,\lambda_2) \in S_{1,j_0,\varepsilon} \times S_{2,j_0,\varepsilon}$. Suppose that $j \in J_{3,\varepsilon}$ satisfies $(\lambda_1,\lambda_2) \in S_{1,j,\varepsilon} \times S_{2,j,\varepsilon}$. We define $\lambda_{j_0} \in B_{\varepsilon,j_0}$ as the center of $B_{\varepsilon,j_0}$ and choose $\lambda_j \in B_{\varepsilon,j}$ arbitrarily. The assumption $\lambda_1 \in S_{1,j_0,\varepsilon} \cap S_{1,j,\varepsilon}$ implies that there exist $\lambda_{j_0}^\prime \in B_{\varepsilon,j_0}$ and $\lambda_j^\prime \in B_{\varepsilon,j}$ such that $(\lambda_{j_0}^\prime - \lambda_1) \in S_2(\varepsilon)$ and $(\lambda_j^\prime - \lambda_1) \in S_2(\varepsilon)$. Similarly, the assumption $\lambda_2 \in S_{2,j_0,\varepsilon} \cap S_{2,j,\varepsilon}$ yields $\tilde{\lambda}_{j_0} \in B_{\varepsilon,j_0}$ and $\tilde{\lambda_j} \in B_{\varepsilon,j}$ such that $(\tilde{\lambda}_{j_0} - \lambda_2) \in S_1(\varepsilon)$ and $(\tilde{\lambda}_j - \lambda_2) \in S_1(\varepsilon)$. We note that
\begin{align}
\label{eq:RelationLambdaI}
|(\lambda_{j_0}^\prime - \lambda_1) - (\lambda_j^\prime - \lambda_1) - (\lambda_{j_0} - \lambda_j)| &\leq 4 \varepsilon, \\
\label{eq:RelationLambdaII}
|(\tilde{\lambda}_{j_0} - \lambda_2) -(\tilde{\lambda}_j - \lambda_2) - (\lambda_{j_0} - \lambda_j)| &\leq 4 \varepsilon.
\end{align}
Now we define the new hypersurfaces $S_1^\prime = S_1^\prime(j_0,\lambda_2)$ and $S_2^\prime = S_2^\prime(j_0,\lambda_1)$ as
\begin{equation*}
S_1^\prime = S_1 - (\tilde{\lambda}_{j_0} - \lambda_2) + \lambda_{j_0}, \quad \quad S_2^\prime = S_2 - (\lambda_{j_0}^\prime - \lambda_1) + \lambda_{j_0}.
\end{equation*}
Since $(\tilde{\lambda}_{j_0} - \lambda_2) \in S_1(\varepsilon)$ and $(\lambda_{j_0}^\prime - \lambda_1) \in S_2(\varepsilon)$, it follows that $\lambda_{j_0} \in S_{1}^\prime (\varepsilon) \cap S_2^{\prime}(\varepsilon)$. In addition, we deduce from $(\tilde{\lambda}_j - \lambda_2) \in S_1(\varepsilon)$, \eqref{eq:RelationLambdaII} and $(\lambda_j^\prime - \lambda_1) \in S_2(\varepsilon)$, \eqref{eq:RelationLambdaI} that
\begin{equation*}
\text{dist}(\lambda_j,S_1^\prime) \leq 6 \varepsilon, \quad \quad \text{dist}(\lambda_j,S_2^\prime) \leq 6 \varepsilon.
\end{equation*}
Since $\lambda_j \in B_{\varepsilon,j}$ was chosen arbitrarily, the above display implies that if $j \in J_{3,\varepsilon}$ satisfies $(\lambda_1,\lambda_2) \in S_{1,j,\varepsilon} \times S_{2,j,\varepsilon}$, then it holds that
\begin{equation*}
B_{\varepsilon,j} \subseteq S_1^\prime(6\varepsilon) \cap S_2^\prime(6 \varepsilon) \cap S_3(6 \varepsilon).
\end{equation*}
Consequently, \eqref{eq:RelevantJ} follows from Proposition \ref{prop:IntersectionThickenedHypersurface} as
\begin{equation*}
\int_{\R^3} \chi_{S_1^\prime(6 \varepsilon)} \chi_{S_2^\prime(6 \varepsilon)} \chi_{S_3(6 \varepsilon)} \lesssim \varepsilon^3 A.
\end{equation*}
\end{proof}

\subsection{Functions on thickened hypersurfaces}

With the notations from above and $(f_i)_{i=1}^3 \subseteq C_c(\R^3)$ compactly supported functions that
\begin{equation}
\label{eq:ConvergenceThickenedHypersurfaces}
\frac{1}{(2 \varepsilon)^3} \int_{\R^3} (f_1 \vert_{S_1(\varepsilon)} * f_2 \vert_{S_2(\varepsilon)})(x) f_3 \vert_{S_3(\varepsilon)}(x) dx \rightarrow \int_{S_3} (f_1 \vert_{S_1} * f_2 \vert_{S_2})(x) f_3 \vert_{S_3}(x) d\sigma_3(x),
\end{equation}
where $\sigma_3$ denotes the surface measure on $S_3$. 
Since $\varepsilon^{-1/2} \| f_i \|_{L^2(S_i(\varepsilon))} \to \| f_i \|_{L^2(S_i)}$, the estimate \eqref{eq:ConvergenceThickenedHypersurfaces}, together with 
Theorem \ref{thm:GlobalNonlinearLoomisWhitneyInequality} immediately yields Theorem \ref{thm:GlobalNonlinearLoomisWhitneyInequalityR3}.

At several points, we make use of the coarea formula:
\begin{theorem}[Coarea formula]
Let $\Omega \subseteq \R^n$ be an open set and $u: \Omega \rightarrow \R^k$ a Lipschitz-continuous mapping, where $k \leq n$. Then, the following equality holds:
\begin{equation}
\label{eq:CoareaFormula}
\int_{\Omega} g(x) dx = \int_{\R^k} \int_{u^{-1}(t)} g(x) |J_k(x)| d\mathcal{H}^k(x) dt,
\end{equation}
where $d \mathcal{H}^k$ denotes $k$-dimensional Hausdorff measure and $J_k(x) = |det ((J u)^t Ju)|^{1/2}$ the $k$-Jacobian of $u$.
\end{theorem}

We have the following lemma on the convolution on hypersurfaces:
\begin{lemma}
Let $0<\beta \leq 1$ and $S_1,S_2$ denote oriented transverse $C^{1,\beta}$-\-hy\-per\-sur\-faces and let $f_i \in C_c(\R^3)$, $i=1,2$. Then, the following holds true:
\begin{equation*}
\sup_{x \in \R^3} \left| \frac{1}{(2 \varepsilon)^2} f_1 \vert_{S_1(\varepsilon)} * f_2 \vert_{S_2(\varepsilon)} (x) - f_1 \vert_{S_1} * f_2 \vert_{S_2} (x) \right| \rightarrow 0 \text{ as } \varepsilon \to 0.
\end{equation*}
\end{lemma}

\begin{proof}
Let $F_1$ parametrize $S_1$ and $F_2^x$ parametrize $x-S_2$. In the first step, we use the coarea formula with $u= (F_1, F_2^x)$ to decompose $S_1(\varepsilon)$ and $S_2(\varepsilon)$ into hypersurfaces:
\begin{equation*}
\begin{split}
&\quad f_1 \vert_{S_1(\varepsilon)} * f_2 \vert_{S_2(\varepsilon)}(x) \\
 &= \int_{\substack{ y \in S_1(\varepsilon), \\ x-y \in S_2(\varepsilon)}} f_1(x-y) f_2(y) dy \\
&= \int_{-\varepsilon}^\varepsilon \int_{-\varepsilon}^\varepsilon \int_{\substack{x- y \in S_2^{\varepsilon_2},\\ y \in S_1^{\varepsilon_1}}} f_1(x-y) f_2(y) \sin^{-1}(\alpha(y,x-y)) d\mathcal{H}^1(y) d\varepsilon_1 d\varepsilon_2.
\end{split}
\end{equation*}
Here, $\alpha(y,x-y)$ denotes the angle between $\mathfrak{n}_1(y)$ and $\mathfrak{n}_2(x-y)$ and $\mathcal{H}^1$ the one-dimensional Hausdorff measure. We parametrize $S_1^{\varepsilon_2} \cap (x- S_2^{\varepsilon_1})$ by $\gamma_x^{\varepsilon_1,\varepsilon_2}: (0,1) = I \rightarrow S_2^{\varepsilon_2} \cap (x- S_1^{\varepsilon_1})$ by virtue of the implicit function theorem.\\
Note that $|S_1^{\varepsilon_2} \cap (x- S_2^{\varepsilon_1})|$ depends continuously on $\varepsilon_1, \varepsilon_2$ and $x$. Moreover, it is enough to consider $S_1 \cap (x-S_2) \neq \emptyset$. For these points, the implicit function theorem gives that $\gamma_x^{\varepsilon_1,\varepsilon_2}$ depends jointly continuously differentiable on $x$, $\varepsilon_1$ and $\varepsilon_2$.

This gives by the mean value theorem
\begin{equation*}
\begin{split}
&\quad f_1 \vert_{S_1(\varepsilon)} * f_2 \vert_{S_2(\varepsilon)}(x) \\
 &= \int_{-\varepsilon}^\varepsilon \int_{-\varepsilon}^\varepsilon \int_{\substack{ y \in S_2^{\varepsilon_2}, \\ x-y \in S_1^{\varepsilon_1}}} f_1(x-y) f_2(y) \sin^{-1}(\alpha(x-y,y)) d\mathcal{H}^1(y) d\varepsilon_1 d\varepsilon_2 \\
&= \int_{-\varepsilon}^\varepsilon \int_{-\varepsilon}^\varepsilon \int_I f_1(x-\gamma_x^{\varepsilon_1,\varepsilon_2}(t)) f_2(\gamma_x^{\varepsilon_1,\varepsilon_2}(t)) \sin^{-1}(\alpha) | \dot{\gamma}_x^{\varepsilon_1,\varepsilon_2}(t)| dt d\varepsilon_1 d\varepsilon_2 \\
&= (2\varepsilon)^2 \int_I f_1(x-\gamma_x^{\varepsilon_1^\prime,\varepsilon_2^\prime}(t)) f_2(\gamma_x^{\varepsilon_1,\varepsilon_2}(t)) \sin^{-1}(\alpha) |\dot{\gamma}_x^{\varepsilon_1^\prime, \varepsilon_2^\prime}(t)| dt
\end{split}
\end{equation*}
for $\varepsilon_1^\prime, \varepsilon_2^\prime \in [-\varepsilon,\varepsilon]$. The proof is complete.
\end{proof}

Furthermore, we have the following lemma:
\begin{lemma}
Let $0< \beta \leq 1$, $S_3$ be a $C^{1,\beta}$-hypersurface and $f_3 \in C_c(\R^3)$. Let $(g_\varepsilon)_{\varepsilon \in (-\varepsilon^\prime,\varepsilon^\prime)} \subseteq C_c(\R^3)$ denote a family of continuous functions with $g_\varepsilon \rightarrow g \in C_c(\R^3)$ as $\varepsilon \to 0$. Then, we find the following estimate to hold:
\begin{equation*}
\left| \frac{1}{2 \varepsilon} \int_{\R^3} g_\varepsilon(x) f_3 \vert_{S_3(\varepsilon)}(x) dx - \int_{S_3} g(x) f_3 \vert_{S_3}(x) d \sigma_3(x) \right| \to 0.
\end{equation*}
\end{lemma}

\begin{proof}
We use the coarea formula as in the proof of the previous lemma to write
\begin{equation*}
\frac{1}{2 \varepsilon} \int_{S_3(\varepsilon)} g_{\varepsilon}(x) f_3(x) dx = \frac{1}{2 \varepsilon} \int_{-\varepsilon}^\varepsilon \int_{S_3^{\varepsilon^\prime}} g_\varepsilon(x) f_3(x) d\mathcal{H}^{2}(x) d\varepsilon^\prime,
\end{equation*}
where $S_3(\varepsilon) = \bigcup_{\varepsilon^\prime \in [-\varepsilon,\varepsilon]} S_3^{\varepsilon^\prime}$.

By continuity of the integral in $\varepsilon^\prime$, we can write by the mean value theorem
\begin{equation*}
\frac{1}{2 \varepsilon} \int_{S_3(\varepsilon)} g_{\varepsilon}(x) f_3(x) dx = \int_{S_3^{\varepsilon_3}} g_\varepsilon(x) f_3(x) d\sigma_3^{\varepsilon_3}(x) \text{ for } \varepsilon_3 \in [-\varepsilon,\varepsilon].
\end{equation*}
Next, we choose parametrizations of $S_3^{\varepsilon_3}$ to write
\begin{align*}
\int_{S_3^{\varepsilon_3}} g_\varepsilon(x) f_3(x) d\sigma_3^{\varepsilon_3}(x) &= \int_{\R^2} g_\varepsilon(\psi_{\varepsilon_3}(x)) f_3(\psi_{\varepsilon_3}(x)) \sqrt{\det((J \psi_{\varepsilon_3})^{\tau} (J \psi_{\varepsilon_3})} dx
\end{align*}
with $J \psi_{\varepsilon} $ independent of $\varepsilon$, which is possible as varying $\varepsilon$ in $S^\varepsilon_3$ only amounts to a linear shift. The proof is complete.
\end{proof}

Taking the above two lemmas together finishes the proof of \eqref{eq:ConvergenceThickenedHypersurfaces}.
\vspace{0.5cm}

We highlight that versions for thickened hypersurfaces like provided by Theorem \ref{thm:GlobalNonlinearLoomisWhitneyInequality} are more natural for applications in the context of dispersive equations, see Section \ref{section:ShorttimeEnergyEstimates}, than the counterparts for actual hypersurfaces. We give another example of relevance for applications to dispersive equations, which shows that it is not enough to require the transversality at vectors respecting the convolution structure. This partially answers \textit{Question 2.2 (1)} from the work \cite{KochSteinerberger2015} by Koch-Steinerberger negatively:

\begin{proposition}
There exist $C^2$-hypersurfaces $S_i \subseteq \R^3$, $i=1,2,3$, which satisfy
\begin{equation}
\label{eq:RestrictedTransversality}
\inf_{\substack{ \lambda_i \in S_i, \\ \lambda_1 + \lambda_2 = \lambda_3}} | \det(\mathfrak{n}_1(\lambda_1),\mathfrak{n}_2(\lambda_2),\mathfrak{n}_3(\lambda_3))| \geq 1/2,
\end{equation}
and for any $C>0$, there exist $f_i \in L^2(S_i)$, $i=1,2$, such that
\begin{equation*}
\Vert f_1 * f_2 \Vert_{L^2(S_3)} \geq C \Vert f_1 \Vert_{L^2(S_1)} \Vert f_2 \Vert_{L^2(S_2)}.
\end{equation*}
\end{proposition}
\begin{proof}
For the sake of contradiction, suppose that for all $C^2$-hypersurfaces satisfying \eqref{eq:RestrictedTransversality}, there exists $C= C(S_1,S_2,S_3)>0$ such that
\begin{equation}
\label{eq:RestrictedLoomisWhitneyInequality}
\Vert f_1 * f_2 \Vert_{L^2(S_3)} \leq C \Vert f_1 \Vert_{L^2(S_1)} \Vert f_2 \Vert_{L^2(S_2)}.
\end{equation}
Let $-2^{-5} < c_i < 2^{-5}$ and define three families of hypersurfaces $S_i^{c_i} \subseteq \R^3$ as follows:
\begin{align*}
S_1^{c_1} &= \{ (x,y,z) \in \R^3 \, | \, | y| < 2^{-5}, \, z=c_1 \}, \\
S_2^{c_2} &= \{ (x,y,z) \in \R^3 \, | \, y=c_2, \, |z| < 2^{-5} \}, \\
S_3^{c_3} &= \{ (x,y,z) \in \R^3 \, | \, z=\sin(\pi x) + c_3 \}.
\end{align*}
Since $|c_i| < 2^{-5}$, it is straight-forward to check that
\begin{equation*}
\inf_{\substack{0 \leq |c_i| < 2^{-5}, \\ i = 1,2,3}} \inf_{\substack{\lambda_i \in S_i^{c_i},\\ \lambda_1 + \lambda_2 = \lambda_3 }} |\det(\mathfrak{n}_1(\lambda_1),\mathfrak{n}_2(\lambda_2),\mathfrak{n}_3(\lambda_3))| \geq 1/2 .
\end{equation*}
By \eqref{eq:RestrictedLoomisWhitneyInequality}, for any $0 \leq |c_i| < 2^{-5}$, $i=1,2,3$, we get
\begin{equation*}
\Vert f_1 * f_2 \Vert_{L^2(S_3^{c_3})} \lesssim \Vert f_1 \Vert_{L^2(S_1^{c_1})} \Vert f_2 \Vert_{L^2(S_2^{c_2})}.
\end{equation*}
Setting $S_i(2^{-5}) = \bigcup_{0 \leq |c_i| < 2^{-5}} S_i^{c_i} \subseteq \R^3$, this gives
\begin{equation}
\label{eq:ThickenedRestrictedLoomisWhitneyInequality}
\left| \int_{\R^3} ( f_1 \vert_{S_1(2^{-5})} * f_2 \vert_{S_2(2^{-5})} )(x) f_3 \vert_{S_3(2^{-5})} (x) dx \right| \lesssim \prod_{i=1}^3 \Vert f_i \Vert_{L^2(S_i(2^{-5}))}.
\end{equation}
For $R \gg 1$, consider
\begin{equation*}
\mathcal{T}_R = \bigcup_{\substack{|k| \leq R,\\ k \in \Z}} B((k,0,0),2^{-10}) \subseteq S_1(2^{-5}) \cap S_2(2^{-5}) \cap S_3(2^{-5}).
\end{equation*}
Set $f_1 = f_2 = f_3 = \chi_{\mathcal{T}_R}$. Then, $\Vert f_i \Vert_{L^2} \sim R^{1/2}$, and
\begin{equation*}
\left| \int_{\R^3} (f_1 \vert_{S_1(2^{-5})} * f_2 \vert_{S_2(2^{-5})} )(x) f_3 \vert_{S_3(2^{-5})}(x) dx \right| \sim R^2,
\end{equation*}
which contradicts \eqref{eq:ThickenedRestrictedLoomisWhitneyInequality}. The proof is complete.
\end{proof}

\subsection{Loomis-Whitney-type inequalities on $\R \times \Z^2/N$}
The previous considerations allow us to prove a version of the nonlinear Loomis-Whitney inequality on $\R \times$lattices under scalable assumptions:
\begin{proposition}
\label{prop:PeriodicNLW}
Let $1 \leq A \ll N$ be dyadic and $f_i : \R \times \Z^2/N \rightarrow \R$. 
For $i=1,2,3$, let $S_i = \{ (\psi_i(\xi),\xi) \; | \; \xi \in \R^2 \}$ be hypersurfaces with $C^{1,1}$-functions $\psi_i$ on $\R^2$. Suppose that $\| \nabla \psi_i\|_{L^\infty} \lesssim 1$, and that the $S_i,\;i=1,2,3$ obey \eqref{eq:HoelderConditionUnitNormals} with $\beta = 1$ and \eqref{eq:TransversalityAssumption}.\\
Suppose that
\begin{equation*}
\text{supp}(f_i) \subseteq S_i(L_i), \ S_i(L_i) = \{ (\tau,k) = (\tau,k_1,k_2) \in \R \times \Z^2/N \; | \; |\tau - \psi_i(k)| \leq L_i \}.
\end{equation*}
Then, we find the following estimate to hold:
\begin{equation}
\label{eq:InequalityOnLattices}
N^4 \left| \int_{\R \times (\Z^2 / N)} (f_1 * f_2) f_3 d\tau (dk)_N \right| \lesssim C(A,N,L_1,L_2,L_3) \prod_{i=1}^3 (N \Vert f_i \Vert_{L_\tau^2 L^2_{(dk)_N}} ),
\end{equation}
where\footnote{Note that the convolution on $\Z^2/N$ also carries the renormalized counting measure.}
\begin{equation*}
C(A,N,L_1,L_2,L_3) = L_{\min}^{1/2} \langle N L_{\mathrm{med}} \rangle^{1/2} \langle AN L_{\max} \rangle^{1/2}.
\end{equation*}
\end{proposition}
We point out how in the limiting cases $N \to \infty$ or $L_{\mathrm{med}} \to \infty$ Proposition \ref{prop:PeriodicNLW} recovers \eqref{eq:Theorem4.3} in Theorem \ref{thm:GlobalNonlinearLoomisWhitneyInequality}.

\begin{proof}
The claim is that \eqref{eq:InequalityOnLattices} holds with
\begin{equation}
\label{ConstantC}
C(A,N,L_1,L_2,L_3) = 
\begin{cases}
L_{\min}^{1/2} \langle A N L_{\max} \rangle^{1/2}, \quad & L_{\mathrm{med}} \leq N^{-1}, \\
(A L_1 L_2 L_3)^{1/2} N, \quad & N^{-1} \leq L_{\mathrm{med}}.
\end{cases}
\end{equation}
Without loss of generality, we can assume $L_1 \leq L_2 \leq L_3$. If $L_{\max} \geq \frac{1}{AN}$, by decomposing $S_3(L_3)$ into $L_3/L_2$ translated $L_2$-thickened $S_3$, we can also assume that $L_2 = L_3$. Furthermore, if $L_2 \geq N^{-1}$, we decompose $S_2(L_2)$ and $S_3(L_2)$ into $NL_2$ translated $N^{-1}$-thickened hypersurfaces $S_2$ and $S_3$, respectively. If $L_{\max} \leq \frac{1}{AN}$, we do not decompose.

It suffices to show
\begin{equation*}
N^4 \left| \int_{\R \times (\mathbb{Z}/N)^2} (f_1 * f_2) f_3 d\tau (dk)_N \right| \lesssim L_1^{1/2} \langle AN L_2 \rangle^{1/2} \prod_{i=1}^3 (N \Vert f_i \Vert_{L_\tau^2 L^2_{(dk)_N}} ), 
\end{equation*}
for $L_2 = L_3$ and $L_2 \leq N^{-1}$. The support of spatial frequencies for $f$ will be denoted by $\text{supp}_k(f)$. Suppose that $k_3 \in \text{supp}_k(f_3)$ is fixed and define
\begin{align*}
\Phi_1(k_1,\tau_1,k_3,\tau_3) &= |\tau_1 - \psi_1(k_1)| + |\tau_3 - \tau_1 - \psi_2(k_3-k_1)| + |\tau_3 - \psi_3(k_3)|, \\
\Phi_2(k_2,\tau_2,k_3,\tau_3) &= |\tau_2-\psi_2(k_2)| + |\tau_3-\tau_2 - \psi_1(k_3-k_2)| + |\tau_3 - \psi_3(k_3)|, \\
S^1_{k_3,L_2} &= \{ k_1 \in \text{supp}_k f_1 \; | \; k_3 - k_1 \in \text{supp}_k f_2, \\
&\quad \quad \exists \tau_1, \tau_3 \in \R: \Phi_1(k_1,\tau_1,k_3,\tau_3) \lesssim L_2 \}, \\
S^2_{k_3,L_2} &= \{ k_2 \in \text{supp}_k f_2 \; | \; k_3 - k_2 \in \text{supp}_k f_1, \\
&\quad \quad \exists \tau_2, \tau_3 \in \R: \Phi_2(k_2,\tau_2,k_3,\tau_3) \lesssim L_2 \}.
\end{align*}
Note that $k_3 - S^1_{k_3,L_2} = S^2_{k_3,L_2}$. For all fixed $(k_1,k_2) \in \text{supp}_k f_1 \times \text{supp}_k f_2$, we show the following:
\begin{equation}
\label{orthogonality}
\sum_{k_3} \chi_{S_{k_3,L_2}^1 \times S^2_{k_3,L_2}}(k_1,k_2) \lesssim \langle ANL_2 \rangle.
\end{equation}
Firstly, we consider the easy case of large $N$. Observe that for $N^{-1} \ll A^{-2}$, $N^{-1} \leq L_{\mathrm{med}}$ this is a consequence of the considerations from the proof of Theorem \ref{thm:GlobalNonlinearLoomisWhitneyInequality} as the $1/N$-lattice points can be related with the $\varepsilon$-balls from above.

In this case, like in \eqref{eq:RelevantJ},
\begin{equation*}
\sum_{k_3} \chi_{S_{k_3,L_2}^1 \times S^2_{k_3,L_2}} \lesssim A,
\end{equation*}
and we infer the bound with $C(A,N,L_1,L_2,L_3) = (A L_1 L_2 L_3)^{1/2} N$.


The case of smaller $N$ requires more sophisticated arguments. We prove \eqref{orthogonality} by contradiction. 
First we consider the simple case $L_2 \lesssim A^{-1} N^{-1}$. 
Assume that there exist $k_3'$, $\tilde{k_3} \in \supp_k {f_3}$ such that 
$|k_3'-\tilde{k_3}| \gg N^{-1}$ and there exists 
$(k_1', k_2') \in \supp_{k}{f_1} \times \supp_{k}{f_2}$ 
such that
\begin{equation}
(k_1', k_2') \in (S_{k_3',L_2}^1 \times S_{k_3',L_2}^2) \cap (S_{\tilde{k_3},L_2}^1 \times S_{\tilde{k_3},L_2}^2).\label{relation1}
\end{equation}
For $k_i \in \supp_k f_i$, let us write $\lambda_i (k_i) = (k_i, \psi_i (k_i)) \in S_i$ for $k_i \in \Z^2 / N$. 
Then $k_1' \in S_{k_3',L_2}^1$ implies that 
\begin{equation}
|\lambda_1(k_1') + \lambda_2'(k_3'-k_1') - \lambda_3(k_3')| \lesssim A^{-1} N^{-1}.\label{relation2}
\end{equation}
Similarly, it follows from \eqref{relation1} that
\begin{align}
&|\lambda_1(k_3'-k_2') + \lambda_2(k_2') - \lambda_3(k_3')| \lesssim A^{-1} N^{-1},\label{relation3}\\
&|\lambda_1(k_1') + \lambda_2(\tilde{k_3}-k_1') - \lambda_3(\tilde{k_3})| \lesssim A^{-1} N^{-1},\label{relation4}\\
&|\lambda_1 (\tilde{k_3}-k_2') + \lambda_2(k_2') - \lambda_3(\tilde{k_3})| \lesssim A^{-1} N^{-1}.\label{relation5}
\end{align}
\eqref{relation2}-\eqref{relation5} yield
\begin{align}
& |\bigl( \lambda_1(k_3'-k_2') - \lambda_1 (\tilde{k_3}-k_2')\bigr) - \bigl(\lambda_3(k_3') - 
\lambda_3(\tilde{k_3})\bigr)| \lesssim A^{-1} N^{-1},\label{relation6}\\
&  |\bigl( \lambda_2(k_3'-k_1') - \lambda_2 (\tilde{k_3}-k_1') \bigr) - \bigl(\lambda_3(k_3') - 
\lambda_3(\tilde{k_3})\bigr)| \lesssim A^{-1} N^{-1}.\label{relation7}
\end{align}
Define the vectors as 
\begin{equation*}
\vec{v}_1 = \lambda_1(k_3'-k_2') - \lambda_1 (\tilde{k_3}-k_2'), \quad 
\vec{v}_2 = \lambda_2(k_3'-k_1') - \lambda_2 (\tilde{k_3}-k_1'), \quad 
\vec{v}_3 =\lambda_3(k_3') - \lambda_3(\tilde{k_3}).
\end{equation*}
By the mean value theorem, there exist 
$\widehat{\lambda}_i \in S_i$ such that $\mathfrak{n}_1 
({\widehat{\lambda}_i}) \perp \vec{v}_i$. 
This, \eqref{relation6}, \eqref{relation7} and $|k_3' - \tilde{k_3}| \gg 1$ which means 
$|\vec{v}_3| \gg N^{-1}$ provide
\begin{equation*}
|\textnormal{det} ({\mathfrak{n}_1} 
({\widehat{\lambda}_1}), {\mathfrak{n}_2}({\widehat{\lambda}_2}), {\mathfrak{n}_3}({\widehat{\lambda}_3})) | \ll A^{-1},
\end{equation*}
which contradicts \eqref{eq:TransversalityAssumption}. 

Next we consider the case $A^{-1} N^{-1} \ll L_2 \leq N^{-1}$. 
By following the above argument, if $|k_3' - \tilde{k_3}| \gg AL_2$ we can show 
$(S_{k_3',L_2}^1 \times S_{k_3',L_2}^2) \cap (S_{\tilde{k_3},L_2}^1 \times S_{\tilde{k_3},L_2}^2)= \emptyset$. 
Thus, after a harmless decomposition, it suffices to show that for any $k_3' \in \supp_k f_3$ it holds
\begin{equation*}
\sum_{|k_3-k_3'| \ll A L_2} \chi_{S_{k_3,L_2}^1 \times S_{k_3,L_2}^2} (k_1,k_2) \lesssim \LR{A N L_2}.
\end{equation*}
Without loss of generality, we may assume $k_3'=(0,0)$ and 
$(k_1,k_2) \in S_{0,L_2}^1 \times S_{0,L_2}^2$. Define
\begin{equation*}
K_{k_1,k_2,L_2} = \{k_3 \in \Z^2/N \, | \, |k_3| \ll A L_2, \ (k_1,k_2) \in S_{k_3,L_2}^1 \times S_{k_3,L_2}^2\}.
\end{equation*}
Our goal is to show $\# K_{k_1,k_2,L_2} \lesssim A N L_2$. Let $k_3 \in K_{k_1,k_2,L_2}$. 
By following the same observation as in the former case $L_2 \lesssim A^{-1} N^{-1}$, it follows from $0 \in K_{k_1,k_2,L_2}$, $k_3 \in K_{k_1,k_2,L_2}$ that
\begin{align}
&|\lambda_1(k_1) + \lambda_2(-k_1) - \lambda_3(0)| \lesssim L_2,\label{relation2-1}\\
&|\lambda_1(-k_2) + \lambda_2(k_2) - \lambda_3(0)| \lesssim L_2,\label{relation2-2}\\
&|\lambda_1 (k_1) + \lambda_2(k_3-k_2) - \lambda_3(k_3)| \lesssim L_2,\label{relation2-3}\\
&|\lambda_1 (k_3-k_2) + \lambda_2(k_2) - \lambda_3(k_3)| \lesssim L_2.\label{relation2-4}
\end{align}
These yield
\begin{align}
& |\bigl( \lambda_1(-k_2) - \lambda_1 (k_3-k_2)\bigr) - \bigl(\lambda_3(0) - 
\lambda_3(k_3)\bigr)| \lesssim L_2,\label{relation2-5}\\
&  |\bigl( \lambda_2(-k_1) - \lambda_2 (k_3-k_1) \bigr) - \bigl(\lambda_3(0) - 
\lambda_3(k_3)\bigr)| \lesssim L_2.\label{relation2-6}
\end{align}
Now we define the hypersurfaces $S_1'$ and $S_2'$ as
\begin{equation*}
S_1' = S_1 - \lambda_1(-k_2) + \lambda_3(0), \qquad S_2' = S_2 - \lambda_2(-k_1) + \lambda_3(0).
\end{equation*}
Clearly, $\lambda_3(0) \in S_1' \cap S_2'$ and we deduce from \eqref{relation2-5} and 
\eqref{relation2-6} that
\begin{equation*}
\textnormal{dist}(\lambda_3(k_3), S_1') \lesssim L_2, \quad 
\textnormal{dist}(\lambda_3(k_3), S_2') \lesssim L_2.
\end{equation*}
Consequently, $k_3 \in K_{k_1,k_2,L_2}$ implies
\begin{equation*}
\begin{split}
k_3 \in \tilde{K}_{k_1,k_2,L_2} &:= \{ k_3 \in \Z^2/N \; | \; |k_3| \ll A L_2, \\
&\quad \quad \textnormal{dist}(\lambda_3(k_3), S_1')+ \textnormal{dist}(\lambda_3(k_3), S_2') \lesssim L_2\},
\end{split}
\end{equation*}
and it suffices to show $\#  \tilde{K}_{k_1,k_2,L_2} \lesssim A  N L_2$. 
To see this, we choose $\tilde{k}_3 \in \tilde{K}_{k_1,k_2,L_2}$ which satisfies 
$\displaystyle{ |\tilde{k}_3| \sim \sup_{k_3 \in \tilde{K}_{k_1,k_2,L_2}}|k_3|}$. 
Clearly, $|\tilde{k}_3| \lesssim A^{1/2}  N^{-1/2} L_2^{1/2}$ gives the desired estimate. 
Thus we assume $|\tilde{k}_3| \gg A^{1/2}  N^{-1/2} L_2^{1/2}$. 
Further, for simplicity, we here assume that $\tilde{k}_3$ is on the first-axis, i.e.$\,$there exists $\tilde{k}_{3,1} \in \Z/N$ such that $\tilde{k}_3 = (\tilde{k}_{3,1},0)$. 
For fixed $k_{3,1} \in \Z/N$ which satisfies $|k_{3,1}| \lesssim |\tilde{k}_{3,1}|$, we define
\begin{equation*}
\tilde{K}_{k_1,k_2,L_2}^{k_{3,1}} = \{ k_{3,2} \in \Z/N \, | \, (k_{3,1}, k_{3,2}) \in \tilde{K}_{k_1,k_2,L_2} \}
\end{equation*}
and show 
\begin{equation}
\# \tilde{K}_{k_1,k_2,L_2}^{k_{3,1}} \lesssim \max (A  N L_2^2/ |\tilde{k}_3|, 1),\label{est01-conv-thm}
\end{equation}
which gives the desired estimate as follows.
\begin{align*}
\#  \tilde{K}_{k_1,k_2,L_2} & \sim  |\tilde{k}_3| N \cdot \# \tilde{K}_{k_1,k_2,L_2}^{k_{3,1}}\\
& \lesssim \max(A L_2^2 N^2,  |\tilde{k}_3| N) \lesssim A L_2 N.
\end{align*}
Here we used $L_2 \leq N^{-1}$ and $|\tilde{k}_3| \ll A L_2$. 
We prove \eqref{est01-conv-thm} by contradiction. 
Assume that there exist $k_{3,2}$, $k_{3,2}' \in \tilde{K}_{k_1,k_2,L_2}^{k_{3,1}}$ such that
$d := |k_{3,2}-k_{3,2}'| \gg \max (A L_2^2/ |\tilde{k}_3|, 1/N)$. 
We define $\sigma_1$, $\sigma_2 \in \mathbb{S}^2$ as
\begin{equation*}
\sigma_1 = \frac{\lambda(0)-\lambda(\tilde{k}_3)}{|\lambda(0)-\lambda(\tilde{k}_3)|}, \quad 
\sigma_2 = \frac{\lambda(k_{3,1},k_{3,2})-\lambda(k_{3,1},k_{3,2}')}
{|\lambda(k_{3,1},k_{3,2})-\lambda(k_{3,1},k_{3,2}')|}.
\end{equation*}
Note that since $\|\nabla \psi_3 \|_{L^{\infty}} \lesssim 1$ there exists a constant $0<c<1$ such that $|\sigma_1 \cdot \sigma_2| <1-c$. 
By the same observation as above, it follows from $0$, $\tilde{k}_3 \in \tilde{K}_{k_1,k_2,L_2}$ that there exist 
$\lambda_1' \in S_1'$, $\lambda_2' \in S_2'$, $\lambda_3' \in S_3$ such that $|\lambda_i' - \lambda_3(0)| \lesssim |\tilde{k}_3|$ for $i=1,2,3$ and
\begin{equation}
\{ {\mathfrak{n}_1'}  ({\lambda_1'}), {\mathfrak{n}_2'}  ({\lambda_2'}), {\mathfrak{n}_3}  ({\lambda_3'}) \} 
\subset U_{\sigma_1}^{L_2/{|\tilde{k}_3|}}:= 
\{ \sigma \in \mathbb{S}^2_+ \, | \, |\sigma \cdot \sigma_1| \leq L_2/|\tilde{k}_3|\},\label{est02-conv-thm}
\end{equation}
where $\mathfrak{n}_j' (\lambda_j)$ $(j=1,2)$ is a unit normal on $\lambda_j \in S_j'$ and $\mathbb{S}^2_+ = \{(x,y,z)\in \mathbb{S}^2 \, | \, z >0\}.$ 
Similarly, $k_{3,2}$, $k_{3,2}' \in \tilde{K}_{k_1,k_2,L_2}^{k_{3,1}}$ implies that there exist 
$\tilde{\lambda}_1 \in S_1'$, $\tilde{\lambda}_2 \in S_2'$, $\tilde{\lambda}_3 \in S_3$ such that 
 $|\tilde{\lambda}_i - \lambda_3(0)| \lesssim |\tilde{k}_3|$
\begin{equation}
\{ {\mathfrak{n}_1'}  ({\tilde{\lambda}_1}), {\mathfrak{n}_2'}  (\tilde{\lambda}_2), {\mathfrak{n}_3}  (\tilde{\lambda}_3) \} 
\subset U_{\sigma_2}^{L_2/d}:= 
\{ \sigma \in \mathbb{S}^2_+ \, | \, |\sigma \cdot \sigma_2| \leq L_2/d\}.\label{est03-conv-thm}
\end{equation}
Our aim is to get
\begin{equation}
|{\mathfrak{n}_1'}  ({\lambda_1'})- {\mathfrak{n}_2'}  ({\lambda_2'})| + | {\mathfrak{n}_2'}  ({\lambda_2'}) - {\mathfrak{n}_3}  ({\lambda_3'})| + | {\mathfrak{n}_3}  ({\lambda_3'}) - {\mathfrak{n}_1'}  ({\lambda_1'})| \lesssim \tilde{d} := 
|\tilde{k}_3|+L_2/d+L_2/|\tilde{k}_3|.\label{est04-conv-thm}
\end{equation}
Note that $\tilde{d} \ll 1$. 
Since $\tilde{d} L_2/|\tilde{k}_3| \ll 1/A$, we easily confirm that \eqref{est02-conv-thm}, \eqref{est04-conv-thm} contradict the transversality condition \eqref{eq:TransversalityAssumption}. 
We turn to show \eqref{est04-conv-thm}. For the sake of contradiction, suppose that 
$|{\mathfrak{n}_1'}  ({\lambda_1'})- {\mathfrak{n}_2'}  ({\lambda_2'})| \gg \tilde{d}$. 
Firstly, we note that since 
$|\lambda_i' - \lambda_3(0)| + |\tilde{\lambda}_i - \lambda_3(0)| \lesssim |\tilde{k}_3|$ and $S_1'$, $S_2'$ are $C^{1,1}$-hypersurfaces, we have
\begin{equation}
|{\mathfrak{n}_1'}  ({\lambda_1'})- {\mathfrak{n}_1'}  ({\tilde{\lambda}_1})|+
|{\mathfrak{n}_2'}  ({\lambda_2'})- {\mathfrak{n}_2'}  ({\tilde{\lambda}_2})| \lesssim |\tilde{k}_3|.
\label{est05-conv-thm}
\end{equation}
We deduce from \eqref{est02-conv-thm}, \eqref{est03-conv-thm}, \eqref{est05-conv-thm} and the assumption $|{\mathfrak{n}_1'}  ({\lambda_1'})- {\mathfrak{n}_2'}  ({\lambda_2'})| \gg \tilde{d}$ that 
there exist $s_1$, $s_1'$, $s_2$, $s_2' \in \mathbb{S}^2_+$ which satisy
\begin{align*}
& s_1 \cdot \sigma_1 = s_1' \cdot \sigma_2 = s_2 \cdot \sigma_1 = s_2' \cdot \sigma_2 = 0,\\
& |s_1 - s_1'| + |s_2-s_2'| \lesssim \tilde{d}, \quad |s_1 - s_2| \gg \tilde{d}, \ |s_1'-s_2'| \gg \tilde{d}.
\end{align*}
For $a$, $b \in \R^3$, $a \times b$ denotes the cross product of $a$ and $b$. 
We see that the above contradicts $|\sigma_1 \cdot \sigma_2| <1-c$ as follows. 
\begin{align*}
|\sigma_1 \cdot \sigma_2| & = \frac{|(s_1 \times s_2) \cdot (s_1' \times s_2')|}
{|s_1 \times s_2| |s_1' \times s_2'|}\\
& \geq \frac{|s_1 \times s_2|-|s_1-s_1'|-|s_2-s_2'|}{|s_1 \times s_2| + |s_1-s_1'|+|s_2-s_2'|}\\
& > 1- c/2.
\end{align*}
Here we used $|s_1 \times s_2| \gg \tilde{d}$ which follows from $|s_1 - s_2| \gg \tilde{d}$ and $\tilde{d} \ll 1$.

By using the estimate \eqref{orthogonality}, we complete the proof of \eqref{eq:InequalityOnLattices} as the proof of Theorem \ref{thm:GlobalNonlinearLoomisWhitneyInequality}.
%
\end{proof}

%

\subsection{Examples}
\label{subsection:ExamplesNLW}
At last, we consider an example to compare Loomis-Whitney inequalities in $\R^3$ to the $\R \times$lattice case.

 Let $\psi(\xi,\eta) = \xi^3 + \eta^3$ and consider the surface
\begin{equation*} 
S=\{(\psi(\xi,\eta),\xi,\eta) \;| \;(\xi,\eta) \in \R^2 \}.
\end{equation*}
 
  Let $( \mathcal{U}_i)_{i=1}^3$ denote neighborhoods of $(N,-N)$, $(N,2N)$ and $(2N,N)$. Let $f_i \in L^2(\R^3)$, $f_i \geq 0$ and suppose that $\text{supp}(f_i) \subseteq \{(\xi,\eta,\tau) \in \R^3 \; | \; (\xi,\eta) \in \mathcal{U}_i, \; | \tau - \psi(\xi,\eta)| \leq L_i \}$. After rescaling, $(\xi,\eta) \rightarrow (\xi,\eta)/N$, $\tau \rightarrow \tau/N^3$ we find $\tilde{f}_i \in L^2(\R^3)$ supported in a fixed compact set. Moreover, we have
\begin{equation*}
|\det(\mathfrak{n}(\xi_1,\eta_1),\mathfrak{n}(\xi_2,\eta_2),\mathfrak{n}(\xi_1+\xi_2,\eta_1+\eta_2))| \gtrsim 1.
\end{equation*}

An application of Fubini's theorem and \cite[Corollary~1.6.,~p.~713]{BejenaruHerrTataru2010}
\begin{equation*}
\int_{\R^3} (f_1 * f_2)(\tau,\xi,\eta) f_3(\tau,\xi,\eta) d\xi d\eta d\tau \lesssim N^{-2} (L_1 L_2 L_3)^{1/2} \prod_{i=1}^3 \Vert f_i \Vert_{L^2}.
\end{equation*}

In the periodic case we consider $f_i: \R \times \mathbb{Z}^2 \rightarrow \R$. It is easy to see choosing $f_1$ with $\Z^2$-support $(N,-N)$, $f_2$ with $\Z^2$-support $(N,2N)$ and $f_3$ with $\Z^2$-support $(2N,N)$ that
\begin{equation*}
\int_{\R \times \Z^2} (f_1 * f_2)(\tau,\xi,\eta) f_3(\tau,\xi,\eta) (d\xi)_1 (d\eta)_1 d\tau \sim L_{\min}^{1/2} \prod_{i=1}^3 \Vert f_i \Vert_{L^2(\R \times \Z^2)}.
\end{equation*}

\section{Energy estimates}
\label{section:ShorttimeEnergyEstimates}
The main result of this section is the following proposition:

\begin{proposition}
\label{prop:ShorttimeEnergyEstimate}
Let $\lambda \geq 1$, $T \in (0,1]$, $1<s \leq s^\prime$ and $u_1, u_2 \in C([0,T],H_\lambda^3)$ be $\lambda$-periodic classical solutions to \eqref{eq:PeriodicZakharovKuznetsovEquation}. Set $v=u_1-u_2$. Then, we find the following estimates to hold:
\begin{align}
\label{eq:ShorttimeEnergyEstimateSolutions}
\Vert u_1 \Vert^2_{E_\lambda^{s^\prime}(T)} &\lesssim \Vert u_1(0) \Vert_{H_\lambda^{s^\prime}}^2 + \Vert u_1 \Vert^2_{F_\lambda^{s^\prime}(T)} \Vert u_1 \Vert_{F_\lambda^{s}(T)}, \\
\label{eq:ShorttimeEnergyEstimateDifferencesI}
\Vert v \Vert^2_{E_\lambda^0(T)} &\lesssim \Vert v(0) \Vert_{L_\lambda^2}^2 + \Vert v \Vert_{F_\lambda^0(T)}^2 ( \Vert u_1 \Vert_{F_\lambda^s(T)} +\Vert u_2 \Vert_{F_\lambda^s(T)}), \\
\label{eq:ShorttimeEnergyEstimateDifferencesII}
\Vert v \Vert^2_{E_\lambda^{s^\prime}(T)} &\lesssim \Vert v(0) \Vert_{H_\lambda^{s^\prime}}^2 + \Vert v \Vert_{F_\lambda^0(T)} \Vert v \Vert_{F_\lambda^{s^\prime}(T)} \Vert u_2 \Vert_{F_\lambda^{2s^\prime}(T)} + \Vert v \Vert^2_{F_\lambda^{s^\prime}(T)} \Vert v \Vert_{F_\lambda^s(T)}.
\end{align}
\end{proposition}
At the end of the section, we provide an example indicating that the methods of this paper give estimates that are sharp up to endpoints in terms of Sobolev regularity.

For the proof of Proposition \ref{prop:ShorttimeEnergyEstimate}, we write by the fundamental theorem of calculus for a solution $u \in C([0,T],H_\lambda^3)$ to \eqref{eq:PeriodicZakharovKuznetsovEquation} on $\lambda \T^2$:
\begin{equation*}
\Vert P_N u(t) \Vert_{L^2(\lambda \T^2)}^2 = \Vert P_N u(0) \Vert_{L_\lambda^2}^2 + 2 \int_0^t \int_{\lambda \T^2} P_N u(s,x) \partial_x P_N (u^2)(s,x) dx ds.
\end{equation*}
 To exploit the form of the nonlinearity, we integrate by parts to put the derivative on the lowest frequency. We sketch the necessary standard arguments, for details we refer to previous works \cite{IonescuKenigTataru2008,Schippa2019HDBO}:
 
For $K \ll N$,
\begin{equation*}
\begin{split}
&\quad \int_0^T \int_{\lambda \T^2} P_N u(s,x) \partial_x P_N( u P_K u)(s,x) dx ds \\
&= \int_0^T \int_{\lambda \T^2} P_N u(s,x) \partial_x [(P_N u P_K u) + [P_N(u P_K u) - P_N u P_K u]](s,x) dx ds \\
&= \frac{1}{2} \int_0^T \int_{\lambda \T^2} (P_N u)^2(s,x) (\partial_x P_K u)(s,x) dx ds \\
&\quad + \int_0^T \int_{\lambda \T^2} (P_N u)(s,x) \partial_x [P_N(u P_K u) - P_N u P_K u](s,x) dx ds = A + B.
\end{split}
\end{equation*}
$A$ is already in suitable form. For $B$, we change to Fourier variables to write by the mean value theorem
\begin{equation*}
\begin{split}
&\quad \int_{\lambda \T^2} (P_N u)(s,x) \partial_x [P_N(u P_K u)-P_N u P_K u](s,x) dx \\
&= \frac{1}{\lambda^4} \sum_{\substack{k_1 + k_2 + k_3 = 0, \\ k_i \in \Z^2/\lambda} } \chi_N(k_1)  (-i k_{1,1}) [\chi_N(k_2+k_3) - \chi_N(k_2)] \chi_K(k_3) \prod_{i=1}^3 \hat{u}(s,k_i) \\
&= \frac{1}{\lambda^4} \sum_{\substack{ k_1 + k_2 + k_3 = 0, \\ k_i \in \Z^2 / \lambda} } \chi_N(k_1) \hat{u}(k_1) (-i k_{1,1}) (\nabla \chi_N(\zeta) \cdot k_3) \hat{u}(k_2) \chi_K(k_3) \hat{u}(s,k_3), \\
&\quad \quad \text{ where } |\zeta| \sim N.
\end{split}
\end{equation*}
In the following let $\lambda \geq 1$ denote the period length, $\xi, \eta \in \R$, and we denote $(d\sigma_i)_\lambda = d\tau_i (dk_i)_\lambda$ and
\begin{equation*}
\begin{split}
&\quad \int_* f_1(\tau_1,k_1) f_2(\tau_2,k_2) f_3(\tau_3,k_3) (d\sigma_1)_\lambda (d\sigma_2)_\lambda \\
&= \int_{(\R \times \Z^2 / \lambda)^2} f_1(\tau_1,k_1) f_2(\tau_2,k_2) f_3(\tau_1+\tau_2,k_1+k_2) (d\sigma_1)_\lambda (d\sigma_2)_\lambda.
\end{split}
\end{equation*}

To estimate the frequency localized functions in the short-time function spaces $F_{n,\lambda}$, time has to be localized reciprocally to the highest occuring frequency. The reductions are standard and can already be found in \cite[Section~5]{IonescuKenigTataru2008}. Taking absolute values, we find that the estimates from Proposition \ref{prop:ShorttimeEnergyEstimate} are implied by the following:
\begin{proposition}\label{prop8.1}
Let $\lambda \geq 1$. Assume that $1 \ll N_3 \lesssim N_2 \leq N_1$, $L_{\textnormal{med}} \leq N_1^2$, $f_i:\R \times \Z^2 / \lambda \rightarrow \R_{\geq 0}$ and $\supp f_i \subset G_{N_i,L_i}$. 
Then, we have
\begin{equation*}
\begin{split}
&\quad \left| \int_{*}  \bigl(|k_{3,1}| + |k_{1,1}| \frac{N_3}{N_1} \bigr)f_1 (\tau_1,k_1) 
f_2(\tau_2,k_2) 
f_3(\tau_3,k_3) (d \sigma_1)_\lambda (d \sigma_2)_\lambda \right| \\
& \lesssim 
N_3^{1+\ep} L_{\min}^{\frac{1}{2}} \LR{N_1^{-\frac{1}{2}} L_{\max}^{\frac{1}{2}}}
\|f_1 \|_{L_\tau^2 L^2_{(dk)_\lambda}} \| f_2 \|_{L_\tau^2 L^2_{(dk)_\lambda}} \| f_3 \|_{L_\tau^2 L^2_{(dk)_\lambda}}.
\end{split}
\end{equation*}
\end{proposition}
Note that Proposition \ref{prop:ShorttimeEnergyEstimate} in the case of large modulations $L_{\mathrm{med}} \gtrsim N_1^2$ follows from the Cauchy-Schwarz inequality (cf. Lemma \ref{lemma7.3}). For the same reason, we can suppose that $N_3 \gg 1$.
 
We record estimates, which will be used in the proof. Set $\psi(\xi,\eta)= \xi(\xi^2+\eta^2)$.
\begin{proposition}\label{nlw-ZK}
Let $K_1$, $K_2$, $K_3 \subset \R^2$ satisfy for $i=1$, $2$, $3$
\begin{equation}
\sup_{(\xi_i,\eta_i), (\xi_i',\eta_i') \in K_i} |\nabla \psi(\xi_i,\eta_i) - \nabla \psi(\xi_i', \eta_i')| \ll A^{-1} N_1^2,
\label{AssumptionSurfaceRegularity}
\end{equation}
and for all $(\xi_1, \eta_1) \in K_1$, $(\xi_2,\eta_2) \in K_2$
\begin{equation}
\bigl|
(\xi_1 \eta_2 - \xi_2 \eta_1) \, 
\bigl(3 (\xi_1^2+ \xi_1 \xi_2 + \xi_2^2) - (\eta_1^2 + \eta_1 \eta_2 + \eta_2^2)\bigr)
\bigr| \gtrsim A^{-1} N_1^4,
\label{AssumptionSurfaceTransversality}
\end{equation}
and $\tilde{K}_i = \R \times K_i$. 
Assume that $1 \leq A \leq N_1$, $1 \leq N_3 \lesssim N_1 \sim N_2$ and $f_i$ $(i=1,2,3)$ satisfy
$ \supp f_i \subset G_{N_i, L_i} $.
Then, we find the following estimate to hold:
\begin{equation*}
\begin{split}
&\quad \left| \int_{*}  f_1|_{\tilde{K}_1} (\tau_1,k_1) 
f_2|_{\tilde{K}_2}(\tau_2,k_2) 
f_3|_{\tilde{K}_3}(\tau_3,k_3) (d \sigma_1)_\lambda (d \sigma_2)_\lambda \right| \\
&\lesssim \tilde{C}(A, N, L_1, L_2, L_3) 
\|f_1 \|_{L_\tau^2 L^2_{(dk)_\lambda}} \| f_2 \|_{L_\tau^2 L^2_{(dk)_\lambda}} \| f_3 \|_{L_\tau^2 L^2_{(dk)_\lambda}},
\end{split}
\end{equation*}
where
\begin{equation*}
\tilde{C}(A,N_1,L_1,L_2,L_3) = L_{\min}^{1/2} \langle L_{\mathrm{med}} N_1^{-2} \rangle^{1/2} \langle A L_{\max} N_1^{-2} \rangle^{1/2}. 
\end{equation*}
\end{proposition}
\begin{proof}
We note that
\begin{equation*}
\tilde{C}(A, N_1, L_1, L_2, L_3)  = 
\begin{cases}
L_{\min}^{\frac{1}{2}} 
\LR{A N_1^{-2} L_{\max} }^{\frac{1}{2}}, \quad & L_{\textnormal{med}} \leq  N_1^2 ,\\
(A L_1 L_2 L_3)^{\frac{1}{2}} N_1^{-2}, \quad & N_1^2 \leq L_{\textnormal{med}}\\
\end{cases}
.
\end{equation*}

If we define $\tilde{f_i} (\tau,k) = f_i (N_1^3 \tau, N_1 k)$, these satisfy 
$\supp \tilde{f}_i \subset G_{N_i/N_1, L_i/N_1^3}$, and 
the claim can be rewritten as follows:
\begin{equation}
\begin{split}
& \left| \int_{**}  \tilde{f}_1|_{\tilde{K}_1^{N_1}} (\tau_1,k_1) 
\tilde{f}_2|_{\tilde{K}_2^{N_1}}(\tau_2,k_2) \tilde{f}_3|_{\tilde{K}_3^{N_1}}
(\tau_3,k_3) (d \sigma_1)_{\lambda N_1} (d \sigma_2)_{\lambda N_1} \right|\\
&  \lesssim N_1^{-\frac{5}{2}} \tilde{C}(A, N, L_1, L_2, L_3) 
\|\tilde{f}_1 \|_{L_\tau^2 L^2_{(dk)_{\lambda N_1}}} \| \tilde{f}_2 \|_{L_\tau^2 L^2_{(dk)_{\lambda N_1}}} \| \tilde{f}_3 \|_{L_\tau^2 L^2_{(dk)_{\lambda N_1}}},\label{est01-prop7.1}
\end{split}
\end{equation}
where $** =  (\R \times \mathbb{Z}^2/ \lambda N_1)^2 $, ${K}_i^{N_1} = \{(\xi,\eta) \in \R^2 \, | \, (N_1 \xi, N_1 \eta) \in K_i \}$ and $\tilde{K}_i^{N_1} = \R \times K_{i}^{N_1}$. 
Define $\tilde{L}_i = N_1^{-3} L_i$. Then, by using the notation $C(A, N, L_1, L_2, L_3)$ defined in  \eqref{ConstantC}, \eqref{est01-prop7.1} is implied by
\begin{equation}
\begin{split}
&\quad \left| \int_{**}  \tilde{f}_1|_{\tilde{K}_1^{N_1}} (\tau_1,k_1) 
\tilde{f}_2|_{\tilde{K}_2^{N_1}}(\tau_2,k_2) \tilde{f}_3|_{\tilde{K}_3^{N_1}}
(\tau_3,k_3) (d \sigma_1)_{ \lambda N_1} (d \sigma_2)_{\lambda N_1} \right|\\
& \lesssim  C(A, \lambda N_1, \tilde{L}_1, \tilde{L}_2, \tilde{L}_3) / ( \lambda N_1)
\|\tilde{f}_1 \|_{L_\tau^2 L^2_{(dk)_{\lambda N_1}}} \| \tilde{f}_2 \|_{L_\tau^2 L^2_{(dk)_{ \lambda N_1}}} \| \tilde{f}_3 \|_{L_\tau^2 L^2_{(dk)_{ \lambda N_1}}}.\label{est02-prop7.1}
\end{split}
\end{equation}
We define
\begin{equation*}
S_i = \{(\psi(\xi,\eta), \xi,\eta) \in \R^3 \, | \, (\xi,\eta) \in K_i^{N_1}, \ |(\xi,\eta)| \lesssim 1 \}.
\end{equation*} 
\eqref{est02-prop7.1} is immediately established by Proposition \ref{prop:PeriodicNLW} if the hypersurfaces $S_1$, $S_2$, $S_3$ satisfy Assumption 1. 
Since $\psi$ is a polynomial function, we only need to confirm that the hypersurfaces satisfy the 
necessary transversality condition. 
To show this, we describe the unit normals ${\mathfrak{n}_i}$ on 
$\lambda_i = (\psi(\xi_i,\eta_i),\xi_i,\eta_i) \in S_i$ explicitly:
\begin{equation*}
{\mathfrak{n}}_i(\lambda_i) = 
\frac{1}{\sqrt{1+ (3 \xi_i^2 + \eta_i^2)^2 + 4\xi_i^2 \eta_i^2}} 
\left(-1, \ 3 \xi_i^2 + \eta_i^2, \ 2 \xi_i \eta_i \right).
\end{equation*}
We can assume that there exist $\widehat{\lambda}_i = (\psi (\widehat{\xi}_i, \widehat{\eta}_i), (\widehat{\xi}_i, \widehat{\eta}_i)) \in S_i$ such that 
$\widehat{\lambda}_1 + \widehat{\lambda}_2 = \widehat{\lambda}_3$. 
It is easily observed that \eqref{AssumptionSurfaceRegularity} provides 
\begin{equation*}
\sup_{\lambda_i, \lambda_i' \in S_i} | \mathfrak{n}_i  
({\lambda_i})- \mathfrak{n}_i 
({\lambda_i'}) | \ll A^{-1}.
\end{equation*}
Therefore, it suffices to show 
\begin{equation*}
|\textnormal{det} ({\mathfrak{n}_1} 
({\widehat{\lambda}_1}), {\mathfrak{n}_2}({\widehat{\lambda}_2}), {\mathfrak{n}_3}({\widehat{\lambda}_3})) | 
\gtrsim A^{-1},
\end{equation*}
which follows from the condition \eqref{AssumptionSurfaceTransversality} as follows:
\begin{align*}
& |\textnormal{det} ({\mathfrak{n}_1} 
({\widehat{\lambda}_1}), {\mathfrak{n}_2}({\widehat{\lambda}_2}), {\mathfrak{n}_3}({\widehat{\lambda}_3})) |  \\
 \gtrsim & 
 \left|\textnormal{det}
\begin{pmatrix}
-1 & -1 & - 1 \\
3 {\widehat{\xi}_1}^2 + {\widehat{\eta}_1}^2  & 3 {\widehat{\xi}_2}^2 + {\widehat{\eta}_2}^2  
& 3 {\widehat{\xi}_3}^2 + {\widehat{\eta}_3}^2 \\
2 \widehat{\xi}_1 \widehat{\eta}_1   & 2 \widehat{\xi}_2 \widehat{\eta}_2  & 2 \widehat{\xi}_3 
\widehat{\eta}_3
\end{pmatrix} \right| \notag \\
\gtrsim &  \bigl| (\widehat{\xi}_1 \widehat{\eta}_2 - \widehat{\xi}_2 \widehat{\eta}_1)\bigl( 3 ({\widehat{\xi}_1}^2 + \widehat{\xi}_1 \widehat{\xi}_2 + {\widehat{\xi}_2}^2 ) - 
({\widehat{\eta}_1}^2 + \widehat{\eta}_1 \widehat{\eta}_2 + {\widehat{\eta}_2}^2) \bigr) \bigr|\\
\gtrsim & A^{-1}.
\end{align*}
\end{proof}
It is known that in $\R^2$ a linear transformation (cf. \cite{BenArtziKochSaut2003,GruenrockHerr2014}) allows for a symmetrization of the Zakharov-Kuznetsov equation to the following (up to irrelevant factors)
\begin{equation}
\label{eq:SymmetrizedZakharovKuznetsov}
\partial_t u + (\partial_{x_1}^3 + \partial_{x_2}^3) u = u (\partial_{x_1} + \partial_{x_2}) u.
\end{equation}
We digress for a moment to consider the effect of this transformation:
\begin{equation}
\label{eq:UnsymmetrizedZakharovKuznetsov}
\partial_t u + \partial_{x_1}^3 u + \partial_{x_1} \partial^2_{x_2} u = u \partial_{x_1} u, \quad (t,x) \in \R \times \lambda \T^2, \quad \lambda >0.
\end{equation}
In Fourier space we can morally still consider the symmetrized equation \eqref{eq:SymmetrizedZakharovKuznetsov}. However, the Fourier variables $(\xi,\eta)$ have to satisfy the following equation:
\begin{align}
\label{eq:FourierVariablesSymmetrizedEquation}
\xi &= \sqrt{2} ( \alpha + 3^{-1/2} \beta) \\
\eta &= \sqrt{2} (\alpha - 3^{-1/2} \beta)
\end{align}
for $(\alpha ,\beta) \in \mathbb{Z}^2/\lambda$. 
When we wants to use the orthogonal decompositions from \cite{Kinoshita2019}, we can do so after taking into account that the Fourier support of the ``symmetrized" equation \eqref{eq:SymmetrizedZakharovKuznetsov} is on $M(\mathbb{Z}^2/\lambda)$, where
\begin{equation*}
M = \sqrt{2}
\begin{pmatrix}
1 & 3^{-1/2} \\
1 & -3^{-1/2}
\end{pmatrix}
,\quad
M^{-1} = 2^{-3/2}
\begin{pmatrix}
1 & 1 \\
3^{1/2} & -3^{1/2}
\end{pmatrix}.
\end{equation*}
If we want to compute the measure of a set $S$ with respect to counting measure on $M(\mathbb{Z}^2/\lambda)$ it is more convenient to apply $M^{-1}$ and count the lattice points of $\mathbb{Z}^2/\lambda$ in $S^\prime := M^{-1} S$.

For one of the critical interactions, we have to estimate the number of points of $M(\mathbb{Z}^2/\lambda)$ in a rectangle parallel to the $\eta$-axis with height $N_2$ and width $\ll N_1^{-1}$ with $N_2 \ll N_1$. See Lemma \ref{liouvilleSymmetrized}.

A lattice point $(q,p) \in \N \times \Z$ is in the rotated rectangle parallel to the line $\eta =\sqrt{3} \xi$ with width $\ll N_1^{-1}$ if and only if $(q,p)$ satisfies
\begin{equation}
\label{eq:DiophantineApproximation}
|\sqrt{3} q -  p| \ll \frac{1}{N_1} \Leftrightarrow | \sqrt{3} - \frac{p}{q} | \ll \frac{1}{N_1 q}.
\end{equation}
Here we invoke Liouville's theorem on diophantine approximation:
\begin{theorem}
\label{thm:LiouvilleTheorem}
If $x$ is an irrational algebraic number of degree $n$ over the rational numbers, then there exists a constant $c(x)>0$ such that
\begin{equation*}
\left| x - \frac{p}{q} \right| > \frac{c(x)}{q^n}
\end{equation*}
holds for all integers $p$ and $q$ where $q>0$.
\end{theorem}
Hence, since $\sqrt{3}$ is an irrational algebraic number of degree $2$, for $0<q \ll N_1$, the inequality \eqref{eq:DiophantineApproximation} has no solution since
\begin{equation*}
\left|\sqrt{3} - \frac{p}{q} \right| \leq \frac{1}{N_1 q} \leq \frac{c(\sqrt{3})}{q^2}.
\end{equation*}

The following lemma will be needed:
\begin{lemma}\label{liouville}
Let $\lambda \geq 1$, $\ell$, $w >0$ such that $\ell w \geq 1$ and $\alpha \in \R^2$. Define the vectors $\vec{v_1}=(1,\sqrt{3})$, $\vec{v_2} = (-1, \sqrt{3})$ and
\begin{align*}
S_{\ell, w}^\alpha & = 
\{ (\xi, \eta) \in \R^2 \, | \, (\xi,\eta)= c_1 \vec{v}_1 + c_2 \vec{v}_2, 
\quad |c_1| \leq \ell, \ |c_2| \leq w \} - \alpha,\\
\tilde{S}_{\ell, w}^{\alpha} & = \{ k \in \Z^2/\lambda \cap S_{\ell,w}^{\alpha}\}.
\end{align*}
Then, we have $\displaystyle{\sup_{\alpha \in \R^2} \# \tilde{S}_{\ell, w}^{\alpha} \lesssim \ell w \lambda^2}$.
\end{lemma}
\begin{remark}
Observe how the argument hinges on the ratio of the period lengths. We can still apply Theorem \ref{thm:LiouvilleTheorem} if the ratio of the period lengths is rational. On the other hand, if $k \in \Z/\lambda \times \sqrt{3} \Z / \lambda$ this lemma does not hold true. Indeed, for $\ell \gg 1$ and $0 < w \ll 1$ we find that 
$\# \{ k \in \Z/\lambda \times \sqrt{3} \Z / \lambda \cap S_{\ell,w}^{0}\} \sim \ell \lambda^2 $.
\end{remark}
\begin{proof}
We can assume $\lambda = 1$ by rescaling. By performing a suitable decomposition, it suffices to show 
$\displaystyle{\sup_{\alpha \in \R^2} \# \tilde{S}_{\ell, w}^{\alpha} \lesssim 1}$ for $w \ll \ell^{-1}$. 
Assume that $\tilde{S}_{\ell, w}^{\alpha}$ is not empty. 
Then, after parallel translation, it suffices to show 
$\# \tilde{S}_{2 \ell, 2 w}^{0} = 1$ which is verified by Theorem \ref{thm:LiouvilleTheorem} as above observation.
\end{proof}
The following estimate follows from the Cauchy-Schwarz inequality.
\begin{lemma}\label{lemma7.3}
For $i=1,2,3$, assume that $f_i : \R \times \Z^2/\lambda \rightarrow \R_{\geq 0}$, $\supp f_i \subset G_{N_i, L_i}$ \\
and $\displaystyle{\min_{i=1,2,3} \# \supp_k f_i \lesssim P}$. 
Then we have
\begin{equation*}
\begin{split}
&\quad \left| \int_{*}  f_1 (\tau_1,k_1) 
f_2 (\tau_2,k_2) 
f_3 (\tau_3,k_3) (d \sigma_1)_\lambda (d \sigma_2)_\lambda \right| \\
&\lesssim (P L_{\min})^{\frac{1}{2}}/\lambda 
\|f_1 \|_{L_\tau^2 L^2_{(dk)_\lambda}} \| f_2 \|_{L_\tau^2 L^2_{(dk)_\lambda}} \| f_3 \|_{L_\tau^2 L^2_{(dk)_\lambda}}.
\end{split}
\end{equation*}
\end{lemma}

 
We begin the proof in earnest, for which we consider the two cases:\\
(I) $\max(|k_{1,1}|, |k_{2,1}|) \geq 2^{-5} N_1$, $\quad$ 
(II) $\max(|k_{1,1}|, |k_{2,1}|) \leq 2^{-5} N_1$.

First we consider the case (I). 
Since $|k_{3,1}| + |k_{1,1}| N_3 / N_1 \lesssim N_3$, it suffices to show
\begin{equation}
\begin{split}
&\quad \left| \int_{*}  f_1 (\tau_1,k_1) 
f_2(\tau_2,k_2) 
f_3(\tau_3,k_3) (d \sigma_1)_\lambda (d \sigma_2)_\lambda \right| \\
& \lesssim 
N_3^{\ep} L_{\min}^{\frac{1}{2}} \LR{N_1^{-\frac{1}{2}} L_{\max}^{\frac{1}{2}}}
\|f_1 \|_{L_\tau^2 L^2_{(dk)_\lambda}} \| f_2 \|_{L_\tau^2 L^2_{(dk)_\lambda}} \| f_3 \|_{L_\tau^2 L^2_{(dk)_\lambda}}.
\end{split}\label{est01-prop8.1}
\end{equation}
We perform the linear transformation 
\begin{equation*}
(x,y) \to M (x,y) = \sqrt{2}( x + y/\sqrt{3}, x - y/\sqrt{3})
\end{equation*}
and show the following estimate which is equivalent to \eqref{est01-prop8.1}:
\begin{equation}
\begin{split}
&\quad \left| \int_{*}  g_1 (\tau_1,\ell_1) 
g_2(\tau_2,\ell_2) 
g_3(\tau_3,\ell_3) (d \tilde{\sigma}_1)_\lambda (d \tilde{\sigma}_2)_\lambda \right| \\
&\lesssim 
N_3^{\ep} L_{\min}^{\frac{1}{2}} \LR{N_1^{-\frac{1}{2}} L_{\max}^{\frac{1}{2}}}
\|g_1 \|_{L_\tau^2 L^2_{(d \ell)_\lambda}} \| g_2 \|_{L_\tau^2 L^2_{(d \ell)_\lambda}} \| g_3 \|_{L_\tau^2 L^2_{(d \ell)_\lambda}},
\end{split}\label{est02-prop8.1}
\end{equation}
where, for $\ell \in \R^2$ letting $\tilde{\psi}(\ell)= \tilde{\psi}(\ell_1,\ell_2)=\ell_1^3+\ell_2^3$, 
\begin{equation}
 \supp g_i \subset \tilde{G}_{N_i, L_i}, \quad 
\tilde{G}_{N,L} = \{ (\tau, \ell) \in \R \times M(\Z^2 / \lambda) \, | \, |\ell| \sim N , \ 
|\tau -\tilde{\psi}(\ell)| \lesssim L \}.
\label{Assumptionfunctiong}
\end{equation}
In \eqref{est02-prop8.1} $(d \tilde{\sigma}_i)_\lambda$ denotes the image measure under the linear transformation of $(d\sigma_i)_\lambda$; similarly, for $(d \ell)_\lambda$ and $(d k)_\lambda$. 

As above, the advantage of considering \eqref{est02-prop8.1} over \eqref{est01-prop8.1} is that we can reuse the Whitney type decompositions from \cite{Kinoshita2019}.\\
We note that the assumption (I) $\max(|k_{1,1}|, |k_{2,1}|) \geq 2^{-5} N_1$ provides 
$\max(|\ell_{1,1}+\ell_{1,2}|, |\ell_{2,1}+\ell_{2,2}|) \geq 2^{-6} N_1$ in \eqref{est02-prop8.1}. 
For convenience, performing the linear transformation $M$, we state the estimates that correspond to Proposition \ref{nlw-ZK} and Lemmas \ref{liouville} and \ref{lemma7.3}.
\begin{proposition}\label{nlw-ZKSymmetrized}
Let $K_1$, $K_2$, $K_3 \subset \R^2$ satisfy for $i=1$, $2$, $3$
\begin{equation}
\sup_{(\xi_i,\eta_i), (\xi_i',\eta_i') \in K_i} |\nabla \tilde{\psi}(\xi_i,\eta_i) - \nabla \tilde{\psi}(\xi_i', \eta_i')| \ll A^{-1} N_1^2,
\label{AssumptionSurfaceRegularitySymmetrized}
\end{equation}
and for all $(\xi_1, \eta_1) \in K_1$, $(\xi_2,\eta_2) \in K_2$
\begin{equation}
\bigl|
(\xi_1 \eta_2 - \xi_2 \eta_1) \, 
\bigl(\xi_1 \eta_2 + \xi_2 \eta_1 + 2(\xi_1 \eta_1 + \xi_2 \eta_2)\bigr)
\bigr| \gtrsim A^{-1} N_1^4,
\label{AssumptionSurfaceTransversalitySymmetrized}
\end{equation}
and $\tilde{K}_i = \R \times K_i$. 
Assume that $1 \leq A \leq N_1$, $1 \leq N_3 \lesssim N_1 \sim N_2$ and 
$g_i$ $(i=1,2,3)$ satisfy \eqref{Assumptionfunctiong}. 
Then, we have
\begin{equation*}
\begin{split}
&\quad \left| \int_{*}  g_1|_{\tilde{K}_1} (\tau_1,\ell_1) 
g_2|_{\tilde{K}_2}(\tau_2,\ell_2) 
g_3|_{\tilde{K}_3}(\tau_3,\ell_3) (d \tilde{\sigma}_1)_\lambda (d \tilde{\sigma}_2)_\lambda \right| \\
&\lesssim \tilde{C}(A, N, L_1, L_2, L_3) 
\|g_1 \|_{L_\tau^2 L^2_{(d \ell)_\lambda}} \| g_2 \|_{L_\tau^2 L^2_{(d \ell)_\lambda}} \| g_3 \|_{L_\tau^2 L^2_{(d \ell)_\lambda}}.
\end{split}
\end{equation*}
\end{proposition}
Since $M \vec{v_1} = (\sqrt{2},0)$, $M\vec{v_2} = (0, - \sqrt{2})$, we find that Lemma \ref{liouville} is equivalent to the following:
\begin{lemma}\label{liouvilleSymmetrized}
Let $\lambda \geq 1$ and $c_1$, $c_2 >0$ such that $c_1 c_2 \geq 1$ and $\alpha \in \R^2$. 
Define 
\begin{align*}
R_{c_1,c_2}^\alpha & = 
\{ (\xi, \eta) \in \R^2 \, | \, |\xi|\leq c_1, \ |\eta| \leq c_2 \} - \alpha,\\
\tilde{R}_{c_1, c_2}^{\alpha} & = \{ k \in M(\Z^2/\lambda) \cap R_{\ell,w}^{\alpha}\}.
\end{align*}
Then, we have $\displaystyle{\sup_{\alpha \in \R^2} \# \tilde{R}_{c_1, c_2}^{\alpha} \lesssim \lambda^2 c_1 c_2}$.
\end{lemma}
\begin{lemma}\label{lemma8.4}
For $i=1,2,3$, assume \eqref{Assumptionfunctiong} and $\displaystyle{\min_{i=1,2,3} \# \supp_k g_i \lesssim P}$. 
Then, we have
\begin{equation*}
\begin{split}
&\quad \left| \int_{*}  g_1 (\tau_1,\ell_1) 
g_2 (\tau_2,\ell_2) 
g_3 (\tau_3,\ell_3) (d \tilde{\sigma}_1)_\lambda (d \tilde{\sigma}_2)_\lambda \right| \\
&\lesssim (P L_{\min})^{\frac{1}{2}}/\lambda 
\|g_1 \|_{L_\tau^2 L^2_{(d\ell)_\lambda}} \| g_2 \|_{L_\tau^2 L^2_{(d\ell)_\lambda}} \| g_3 \|_{L_\tau^2 L^2_{(d\ell)_\lambda}}.
\end{split}
\end{equation*}
\end{lemma}

We turn to \eqref{est02-prop8.1} in the case (I). 
We divide the proof into the two cases (Ia) $|\sin \angle(\ell_1,\ell_2)| \gtrsim 1$ and 
(Ib) $|\sin \angle(\ell_1,\ell_2)| \ll 1$.

Let us consider  the case (Ia) first. 
It should be noted that in this case we can assume $N_1 \sim N_2 \sim N_3$. 
We introduce the Whitney decomposition of $\R^2 \times \R^2$ into square tiles.
\begin{definition}[Whitney type decomposition]
 Let $A \geq 2^{10}$ be dyadic, $m \in \Z^2$ and set
\begin{align*}
\mathcal{T}^A_m &= \{ (\xi,\eta) \in \R^2 \; | \; (\xi,\eta) \in [m_1/A, (m_1+1)/A)) \\
&\quad \quad \times [m_2/A,(m_2+1)/A)) \},\\
\Phi (\xi_1, \eta_1, \xi_2 ,\eta_2) & = \xi_1 \xi_2(\xi_1 + \xi_2) + \eta_1 \eta_2 (\eta_1 + \eta_2), \\
F (\xi_1, \eta_1, \xi_2 ,\eta_2) & = \xi_1 \eta_2 +  \xi_2 \eta_1 + 2 (\xi_1 \eta_1 + \xi_2 \eta_2).
\end{align*}
We define
\begin{align*}
Z_A^1 & = \{ (k_1, k_2) \in \Z^2 \times \Z^2 \, | \, 
|\Phi(\xi_1, \eta_1, \xi_2, \eta_2)| \geq A^{-1} N_1^3  \ \ \textnormal{for any} \ (\xi_j, \eta_j) \in 
\mathcal{T}_{k_j}^A \},\\
Z_A^2  & = \{ (k_1, k_2) \in \Z^2 \times \Z^2 \, | \, 
| F (\xi_1, \eta_1, \xi_2 ,\eta_2)| \geq A^{-1} N_1^2  \ \ \textnormal{for any} \ (\xi_j, \eta_j) \in 
\mathcal{T}_{k_j}^A \},\\
Z_A & = Z_A^1 \cup Z_A^2 \subset \Z^2 \times \Z^2, 
\qquad R_A = \bigcup_{(k_1, k_2) \in Z_A} \mathcal{T}_{k_1}^A \times 
\mathcal{T}_{k_2}^A \subset \R^2 \times \R^2.
\end{align*}
It is clear that $A_1 \leq A_2 \Longrightarrow R_{A_1} \subset R_{A_2}$. 
Further, we define
\begin{equation*}
Q_A = 
\begin{cases}
R_A \setminus R_{A/2} \quad & \textnormal{for} \  A \geq 2^{11},\\
 \ R_{2^{10}}  \ \qquad & \textnormal{for} \  A = 2^{10}.
\end{cases}
\end{equation*}
and a set of pairs of integer coordinates $Z_A' \subset Z_A$ such that
\begin{equation*}
\bigcup_{(k_1, k_2) \in Z_A'} \mathcal{T}_{k_1}^A \times 
\mathcal{T}_{k_2}^A = Q_A.
\end{equation*}
We easily see that $Z_A'$ is uniquely defined and 
\begin{equation*}
A_1 \not= A_2 \Longrightarrow Q_{A_1} \cap Q_{A_2} = \emptyset, \quad 
\bigcup_{2^{10} \leq A \leq A_0} Q_{A} = R_{A_0},
\end{equation*}
where $A_0 \geq 2^{10}$ is dyadic. Thus, we can decompose $\R^2 \times \R^2$ as
\begin{equation*}
\R^2 \times \R^2 = \left( \bigcup_{2^{10} \leq A \leq A_0} Q_{A}\right) \cup (R_{A_0})^c.
\end{equation*}
Lastly, we define
\begin{align*}
\mathcal{A}  & = \{ (\tau_1, \xi_1, \eta_1) \times (\tau_2, \xi_2, \eta_2) \in \R^3 \times \R^3 \, | \, 
| \sin \angle \left( (\xi_1, \eta_1), (\xi_2, \eta_2) \right)| \gtrsim 1  \},\\
\tilde{Z}_{A} & = \{ (k_1, k_2) \in Z_A' \, | \, 
\left( \tilde{\mathcal{T}}_{k_1}^A \times \tilde{\mathcal{T}}_{k_2}^A\right) \cap \left( 
\tilde{G}_{N_1, L_1} \times \tilde{G}_{N_2, L_2} \right) \cap \mathcal{A} \not= 
\emptyset \}.
\end{align*}
\end{definition}
\begin{proposition}\label{prop8.6}
Let $\lambda \geq 1$ and $1 \leq A \leq N_1$. 
Assume that $1 \ll N_3 \lesssim N_2 \leq N_1$, $L_{\textnormal{med}} \leq N_1^2$, $(k_1,k_2) \in \tilde{Z}_{A}$ and 
\eqref{Assumptionfunctiong}. Then, we have
\begin{equation*}
\begin{split}
& \left| \int_{*}  g_1|_{\tilde{\mathcal{T}}_{k_1}^A} (\tau_1,\ell_1) 
g_2|_{\tilde{\mathcal{T}}_{k_2}^A} (\tau_2,\ell_2) 
g_3 (\tau_3,\ell_3) (d \tilde{\sigma}_1)_\lambda (d \tilde{\sigma}_2)_\lambda \right| \\
& \lesssim L_{\min}^{\frac{1}{2}} 
\bigl( A^{-\frac{1}{2}} N_1^{-\frac{1}{2}} L_{\max}^{\frac{1}{2}} + \LR{A^{\frac{1}{2}} N_1^{-1} L_{\max}^{\frac{1}{2}}} \bigr) \\
&\quad \quad \|g_1|_{\tilde{\mathcal{T}}_{k_1}^A} \|_{L_\tau^2 L^2_{(d\ell)_\lambda}} \| g_2|_{\tilde{\mathcal{T}}_{k_2}^A} \|_{L_\tau^2 L^2_{(d\ell)_\lambda}} \| g_3 \|_{L_\tau^2 L^2_{(d\ell)_\lambda}}.
\end{split}
\end{equation*}
\end{proposition}
\begin{proof}
For $(\xi_1,\eta_1) \times (\xi_2,\eta_2) \in \mathcal{T}_{k_1}^A \times \mathcal{T}_{k_2}^A$, it holds either 
$|\Phi(\xi_1, \eta_1, \xi_2, \eta_2)| \geq A^{-1} N_1^3$ or $|F(\xi_1, \eta_1, \xi_2, \eta_2)| \geq A^{-1} N_1^2$. 
If $|\Phi(\xi_1, \eta_1, \xi_2, \eta_2)| \geq A^{-1} N_1^3$, by using Lemma \ref{lemma8.4} with $P = \lambda^2 A^{-2}N_1^2$, we obtain
\begin{equation*}
\begin{split}
&\quad \left| \int_{*}  g_1|_{\tilde{\mathcal{T}}_{k_1}^A} (\tau_1,\ell_1) 
g_2|_{\tilde{\mathcal{T}}_{k_2}^A} (\tau_2,\ell_2) 
g_3 (\tau_3,\ell_3) (d \tilde{\sigma}_1)_\lambda (d \tilde{\sigma}_2)_\lambda \right| \\
& \lesssim  A^{-\frac{1}{2}} N_1^{-\frac{1}{2}} L_{\min}^{\frac{1}{2}} L_{\max}^{\frac{1}{2}} 
\|g_1|_{\tilde{\mathcal{T}}_{k_1}^A} \|_{L_\tau^2 L^2_{(d \ell)_\lambda}} \| g_2|_{\tilde{\mathcal{T}}_{k_2}^A} \|_{L_\tau^2 L^2_{(d \ell)_\lambda}} \| g_3 \|_{L_\tau^2 L^2_{(d \ell)_\lambda}}.
\end{split}
\end{equation*}
Next we assume $|F(\xi_1, \eta_1, \xi_2, \eta_2)| \geq A^{-1} N_1^2$. 
This case is handled by Proposition \ref{nlw-ZKSymmetrized}. 
Note that the assumption $| \sin \angle \left( (\xi_1, \eta_1), (\xi_2, \eta_2) \right)| \gtrsim 1$ implies 
$|\xi_1 \eta_2 -\xi_2 \eta_1| \gtrsim N_1^2$ which means 
\eqref{AssumptionSurfaceTransversalitySymmetrized} for 
$(\xi_1,\eta_1) \times (\xi_2,\eta_2) \in \mathcal{T}_{k_1}^A \times \mathcal{T}_{k_2}^A$. 
Since $\mathcal{T}_{k}^A$ is a square tile whose side length is $A^{-1}N_1$, after performing harmless decompositions, 
$\supp_k g_i$ is confined in a ball such that its radius is $r \ll A^{-1}N_1$, which provides  \eqref{AssumptionSurfaceRegularitySymmetrized}. 
Consequently, because $L_{\textnormal{med}} \leq N_1^2$, the claim follows from Proposition \ref{nlw-ZKSymmetrized}.
\end{proof}

In the following we recall the almost orthogonal decompositions from \cite{Kinoshita2019}.
\begin{definition}
Let $\mathcal{K}_0$, $\mathcal{K}_1$, $\mathcal{K}_2$, $\mathcal{K}_0'$, $\mathcal{K}_1'$, 
$\mathcal{K}_2' \subset \R^2$ and  $\tilde{\mathcal{K}}_0$, $\tilde{\mathcal{K}}_1$, 
$\tilde{\mathcal{K}}_2$, 
$\tilde{\mathcal{K}}_0'$, 
$\tilde{\mathcal{K}}_1'$, $\tilde{\mathcal{K}}_2' \subset \R^3$ be defined as follows:
\begin{align*}
\mathcal{K}_0 & = \left\{ (\xi, \eta) \in \R^2 \, | \, \left| 
\eta -(\sqrt{2} - 1)^{\frac{4}{3}} \xi \right| 
\leq 2^{-20} N_1 \right\},\\
\mathcal{K}_1 & = \left\{ (\xi, \eta) \in \R^2 \, | \, \left| 
\eta - ( \sqrt{2}+ 1 )^{\frac{2}{3}} (\sqrt{2} + \sqrt{3} ) \xi \right| 
\leq 2^{-20} N_1 \right\},\\
\mathcal{K}_2 & = \left\{ (\xi, \eta) \in \R^2 \, | \, \left| 
\eta + ( \sqrt{2}+ 1 )^{\frac{2}{3}} (\sqrt{3} - \sqrt{2} ) \xi \right| 
\leq 2^{-20} N_1 \right\},\\
\mathcal{K}_0' & =  \left\{ (\xi, \eta) \in \R^2 \ | \ 
 (\eta, \xi) \in \mathcal{K}_0  \right\},\\
\mathcal{K}_1' & = \left\{ (\xi, \eta) \in \R^2 \ | \ (\eta, \xi) \in \mathcal{K}^1 \right\},\\
\mathcal{K}_2' & = \left\{ (\xi, \eta) \in \R^2 \ | \ (\eta, \xi) \in \mathcal{K}^2 \right\},\\
\tilde{\mathcal{K}}_j & = \R \times \mathcal{K}_j, \quad \tilde{\mathcal{K}}_j' = \R \times \mathcal{K}_j'
 \ \ \textnormal{for} \ j = 0, 1,2.
\end{align*}
We define the subsets of $\R^2 \times \R^2$ and $\R^3 \times \R^3$ as
\begin{align*}
\mathcal{K} \, = & ( \mathcal{K}_0 \times ( \mathcal{K}_1\cup \mathcal{K}_2 ) ) \cup 
( ( \mathcal{K}_1\cup \mathcal{K}_2 ) \times  \mathcal{K}_0 ) \subset \R^2 \times \R^2,\\
\tilde{\mathcal{K}} \, = & ( \tilde{\mathcal{K}}_0 \times ( \tilde{\mathcal{K}}_1 
\cup \tilde{\mathcal{K}}_2 ) ) \cup 
( 
( \tilde{\mathcal{K}}_1\cup \tilde{\mathcal{K}}_2 ) \times  \tilde{\mathcal{K}}_0 ) \subset \R^3 \times \R^3,\\
\mathcal{K}' = & \left( \mathcal{K}_0' \times \left( \mathcal{K}_1' \cup \mathcal{K}_2' \right) \right) \cup 
\left( 
\left( \mathcal{K}_1' \cup \mathcal{K}_2' \right) \times  \mathcal{K}_0' \right) \subset \R^2 \times \R^2,\\
\tilde{\mathcal{K}}'  = & ( \tilde{\mathcal{K}}_0' \times ( \tilde{\mathcal{K}}_1' 
\cup \tilde{\mathcal{K}}_2' ) ) \cup 
( 
( \tilde{\mathcal{K}}_1'\cup \tilde{\mathcal{K}}_2' ) \times  \tilde{\mathcal{K}}_0' ) \subset \R^3 \times \R^3,
\end{align*}
and the complementary sets as 
\begin{align*}
(\mathcal{K})^c & =  (\R^2 \times \R^2) \setminus \mathcal{K} , \quad \ 
(\tilde{\mathcal{K}})^c =  (\R^3 \times \R^3) \setminus \tilde{\mathcal{K}}\\
(\mathcal{K}')^c & = (\R^2 \times \R^2) \setminus \mathcal{K}', \quad 
(\tilde{\mathcal{K}}')^c = (\R^3 \times \R^3) \setminus \tilde{\mathcal{K}}'.
\end{align*}
Lastly, we define
\begin{equation*}
\widehat{Z}_{A} = \{ (k_1, k_2) \in \tilde{Z}_{A} \, | \, 
\left( \mathcal{T}_{k_1}^A \times \mathcal{T}_{k_2}^A\right) \cap \left( 
(\mathcal{K})^c \cap (\mathcal{K}')^c \right) \not= 
\emptyset \},
\end{equation*}
and $\overline{Z}_{A}$ as the collection of $(k_1, k_2) \in \Z^2 \times \Z^2$ which satisfies 
\begin{align*}
\mathcal{T}_{k_1}^{A} & \times \mathcal{T}_{k_2}^{A} \not\subset \bigcup_{2^{100} \leq A' \leq A} 
\bigcup_{(k_1', k_2') \in \widehat{Z}_{A}} \left( \mathcal{T}_{k_1'}^{A'} \times \mathcal{T}_{k_2'}^{A'} \right),
\\ 
 \left( \tilde{\mathcal{T}}_{k_1}^{A} \times \tilde{\mathcal{T}}_{k_2}^{A} \right) & \cap \left( 
\tilde{G}_{N_1, L_1} \times \tilde{G}_{N_2, L_2} \right) \cap \mathcal{A}  \cap \left( 
(\tilde{\mathcal{K}})^c \cap (\tilde{\mathcal{K}}')^c \right) \not= 
\emptyset. 
\end{align*}
\end{definition}
\begin{lemma}[{\cite[Lemma~3.7,~p.~17]{Kinoshita2019}}]
\label{lemma8.8}
For fixed $k_1 \in \Z^2$, the number of $k_2 \in \Z^2$ such that 
$(k_1, k_2) \in \widehat{Z}_{A}$ is less than $2^{1000}$. Furthermore, the same claim holds true if we replace $\widehat{Z}_{A}$ by $\overline{Z}_{A}$.
\end{lemma}
We show \eqref{est02-prop8.1} under the assumption $(\ell_1,\ell_2) \in (\mathcal{K})^c \cap (\mathcal{K}')^c$.
\begin{proof}[Proof of \eqref{est02-prop8.1} for the case $(\ell_1,\ell_2) \in (\mathcal{K})^c \cap (\mathcal{K}')^c$.]
By the definitions of $\widehat{Z}_A$ and $ \overline{Z}_{A_0}$, we see that the set 
$\left( 
\tilde{G}_{N_1, L_1} \times \tilde{G}_{N_2, L_2} \right) \cap \mathcal{A}  \cap 
(\tilde{\mathcal{K}})^c \cap (\tilde{\mathcal{K}}')^c $ is contained in
\begin{equation*}
\bigcup_{2^{10} \leq A \leq N_1} 
\bigcup_{(k_1, k_2) \in \widehat{Z}_A} 
 \left( \tilde{\mathcal{T}}_{k_1}^{A} \times \tilde{\mathcal{T}}_{k_2}^{A} \right)
 \cup \bigcup_{(k_1, k_2) \in \overline{Z}_{N_1}} 
 \left( \tilde{\mathcal{T}}_{k_1}^{N_1} \times \tilde{\mathcal{T}}_{k_2}^{N_1} \right).
\end{equation*}
For short, we use
\begin{equation*}
I_A^{k_1,k_2} = \left| \int_{*}  g_1|_{\tilde{\mathcal{T}}_{k_1}^A} (\tau_1,\ell_1) 
g_2|_{\tilde{\mathcal{T}}_{k_2}^A} (\tau_2,\ell_2) 
g_3 (\tau_3,\ell_3) (d \tilde{\sigma}_1)_\lambda (d \tilde{\sigma}_2)_\lambda \right|.
\end{equation*}
It is observed that
\begin{align*}
&\quad \left| \int_{*}  g_1 (\tau_1,\ell_1) 
g_2(\tau_2,\ell_2) 
g_3(\tau_3,\ell_3) (d \sigma_1)_\lambda (d \sigma_2)_\lambda \right| \\
& \lesssim \sum_{2^{10} \leq A \leq N_1} 
\sum_{(k_1, k_2) \in \widehat{Z}_A} I_{A}^{k_1,k_2} + 
\sum_{(k_1, k_2) \in \overline{Z}_{N_1}} I_{N_1}^{k_1,k_2}.
\end{align*}
For the former term, since $N_1 \sim N_2 \sim N_3$, by employing Proposition \ref{prop8.6} and Lemma \ref{lemma8.8}, we get
\begin{align*}
& \sum_{2^{10} \leq A \leq N_1} 
\sum_{(k_1, k_2) \in \widehat{Z}_A} I_{A}^{k_1,k_2} \\
& \lesssim 
\sum_{2^{10} \leq A \leq N_1}  L_{\min}^{\frac{1}{2}} 
\bigl( A^{-\frac{1}{2}} N_1^{-\frac{1}{2}} L_{\max}^{\frac{1}{2}} + \LR{A^{\frac{1}{2}} N_1^{-1} L_{\max}^{\frac{1}{2}}} \bigr) \\
&\quad \times \sum_{(k_1, k_2) \in \widehat{Z}_A} 
\|g_1|_{\tilde{\mathcal{T}}_{k_1}^A} \|_{L_\tau^2 L^2_{(d\ell)_\lambda}} \| g_2|_{\tilde{\mathcal{T}}_{k_2}^A} \|_{L_\tau^2 L^2_{(d\ell)_\lambda}} \| g_3 \|_{L_\tau^2 L^2_{(d \ell)_\lambda}}\\
& \lesssim (\log N_3)  L_{\min}^{\frac{1}{2}} \LR{N_1^{-\frac{1}{2}} L_{\max}^{\frac{1}{2}}}
\|g_1 \|_{L_\tau^2 L^2_{(d \ell)_\lambda}} \| g_2 \|_{L_\tau^2 L^2_{(d \ell)_\lambda}} \| g_3 \|_{L_\tau^2 L^2_{(d \ell)_\lambda}}.
\end{align*}
For the latter term, it follows from Lemma \ref{lemma8.4} with $M \sim \lambda^2$ and Lemma \ref{lemma8.8} that
\begin{equation*}
\sum_{(k_1, k_2) \in \overline{Z}_{N_1}} I_{N_1}^{k_1,k_2} \lesssim 
L_{\min}^{\frac{1}{2}} \|g_1 \|_{L_\tau^2 L^2_{(d \ell)_\lambda}} \| g_2 \|_{L_\tau^2 L^2_{(d \ell)_\lambda}} \| g_3 \|_{L_\tau^2 L^2_{(d \ell)_\lambda}},
\end{equation*}
which completes the proof.
\end{proof}

Next we prove the estimate \eqref{est02-prop8.1} for $(\ell_1,\ell_2) \in \left( \mathcal{K} \cup \mathcal{K}' \right)$. In this case, the almost one-to-one correspondence of 
$(k_1, k_2) \in \tilde{Z}_{A}$ does not hold. Therefore, we need to introduce another decomposition. 
We note that, by exchanging the roles of $\ell_{i,1}$ and $\ell_{i,2}$ with $i=1,2$, once the estimate 
\eqref{est02-prop8.1} is verified for the case 
$ (\ell_1,\ell_2) \in \mathcal{K}$, one can obtain the same estimate for $(\ell_1,\ell_2) \in \mathcal{K}'$. For the same reason, it suffices to show the estimate \eqref{est02-prop8.1} for the case $(\ell_1,\ell_2) \in (\mathcal{K}_1 \cup \mathcal{K}_2 ) \times \mathcal{K}_0$. 
\begin{definition}
Let $m=(n, z) \in \N  \times \Z$. We define the monotone increasing sequence $\{ a_{A,n} \}_{n \in \N} $ as 
\begin{equation*}
a_{A,1} = 0, \qquad a_{A,n+1} = a_{A,n} + \frac{N_1}{\sqrt{(n+1)A}}.
\end{equation*}
and sets $\mathcal{R}_{A,m,1}$, 
$\mathcal{R}_{A,m,2}$ as follows:
\begin{align*}
\mathcal{R}_{A,m,1} = &
\left\{ (\xi, \eta)  \in \R^2  \, \left| \, 
\begin{aligned} & a_{A,n} \leq | \eta-   ( \sqrt{2}+ 1 )^{\frac{2}{3}} (\sqrt{2} + \sqrt{3} ) \xi | 
< a_{A,n+1}, \\ 
& z A^{-1}  N_1 \leq \eta-( \sqrt{2}+ 1 )^{\frac{2}{3}}\xi < (z+1) A^{-1}N_1
 \end{aligned} \right.
\right\},\\
\mathcal{R}_{A,m,2} = &
\left\{ (\xi, \eta)  \in \R^2  \, \left| \, 
\begin{aligned} & a_{A,n} \leq | \eta +  ( \sqrt{2}+ 1 )^{\frac{2}{3}} (\sqrt{3} - \sqrt{2} ) \xi | 
< a_{A,n+1}, \\ 
& z A^{-1}  N_1 \leq \eta-( \sqrt{2}+ 1 )^{\frac{2}{3}}\xi < (z+1) A^{-1}N_1
 \end{aligned} \right.
\right\},\\
\tilde{\mathcal{R}}_{A,m,1} = & \R \times \mathcal{R}_{A,m,1}, \quad 
\tilde{\mathcal{R}}_{A,m,2} =  \R \times \mathcal{R}_{A,m,2}.
\end{align*}
We will perform the Whitney type decomposition by using the above sets 
instead of simple square tiles. 
We define for $i=1,2$ that
\begin{align*}
M_{A,i}^1 & = \left\{ (m, k) \in (\N \times \Z) \times \Z^2 \ \left| \ 
\begin{aligned} & |\Phi(\xi_1, \eta_1, \xi_2, \eta_2)| \geq A^{-1} N_1^3 \ \\ 
&\textnormal{for any} \ (\xi_1, \eta_1) \in \mathcal{R}_{A,m,i} \  \textnormal{and} \
(\xi_2, \eta_2) \in \mathcal{T}_{k}^A 
 \end{aligned} \right.
\right\},\\
M_{A,i}^2 & = \left\{ (m, k) \in (\N \times \Z) \times \Z^2 \ \left| \ 
\begin{aligned} & |F(\xi_1, \eta_1, \xi_2, \eta_2)| \geq A^{-1} N_1^3 \ \\ 
&\textnormal{for any} \ (\xi_1, \eta_1) \in \mathcal{R}_{A,m,i} \  \textnormal{and} \
(\xi_2, \eta_2) \in \mathcal{T}_{k}^A 
 \end{aligned} \right.
\right\},\\
M_{A,i} & = M_{A,i}^1 \cup M_{A,i}^2 \subset  (\N \times \Z) \times \Z^2, \\
R_{A,i} & = \bigcup_{(m, k) \in M_{A,i}} \mathcal{R}_{A,m,i} \times 
\mathcal{T}_{k}^A \subset \R^2 \times \R^2.
\end{align*}
Furthermore, we define ${M}_{A,i}' \subset M_{A,i}$ as the collection of $(m,k) \in \N \times \Z$ such that 
\begin{equation*}
\mathcal{R}_{A,m,i} \times 
\mathcal{T}_{k}^A \subset \bigcup_{2^{10} \leq A' <A} R_{A', i}.
\end{equation*}
By using ${M}_{A,i}'$, we define
\begin{align*}
Q_{A,i} = 
\begin{cases}
R_{A,i} \setminus {\displaystyle \bigcup_{(m,k) \in {M}_{A,i}'}} (\mathcal{R}_{A,m,i} \times 
\mathcal{T}_{k}^A) \ & \textnormal{for} \   A \geq 2^{11},\\
 \quad \  R_{2^{10},i}  \, \quad \qquad \qquad \qquad \qquad & \textnormal{for} \  A = 2^{10},
\end{cases}
\end{align*}
and $\tilde{M}_{A,i} = M_{A,i} \setminus {M}_{A,i}'$. Clearly, the followings hold.
\begin{equation*}
\bigcup_{(m,k) \in \tilde{M}_{A,i}} \mathcal{R}_{A,m,i} \times 
\mathcal{T}_{k}^A = Q_{A,i}, \quad 
\bigcup_{2^{10} \leq A \leq A_0} Q_{A,i} = R_{A_0,i},
\end{equation*}
where $A_0 \geq 2^{10}$ is dyadic. 
Lastly, we define
\begin{align*}
\widehat{Z}_{A,i} & = \{ (m, k) \in \tilde{M}_{A,i} \, | \, 
(  \tilde{\mathcal{R}}_{A,m,i} \times 
\tilde{\mathcal{T}}_{k}^A ) \cap \left( 
\tilde{G}_{N_1, L_1} \times \tilde{G}_{N_2, L_2} \right) \cap ( \tilde{\mathcal{K}}_i \times 
\tilde{\mathcal{K}}_0 ) \not= 
\emptyset \},\\
\overline{Z}_{A,i} & = \{ (m, k) \in  M_{A,i}^c \, | \, 
(  \tilde{\mathcal{R}}_{A,m,i} \times 
\tilde{\mathcal{T}}_{k}^A ) \cap \left( 
\tilde{G}_{N_1, L_1} \times \tilde{G}_{N_2, L_2}  \right) \cap ( \tilde{\mathcal{K}}_i \times 
\tilde{\mathcal{K}}_0 )  \not= 
\emptyset \},
\end{align*}
where $M_{A,i}^c = (\N \times \Z) \setminus M_{A,i}$. We easily see that
\begin{equation*}
 \left( 
\tilde{G}_{N_1, L_1} \times \tilde{G}_{N_2, L_2}  \right) \cap ( \tilde{\mathcal{K}}_i \times 
\tilde{\mathcal{K}}_0  ) \subset \bigcup_{(m,k) \in \widehat{Z}_{A,i}} (  \tilde{\mathcal{R}}_{A,m,i} \times 
\tilde{\mathcal{T}}_{k}^A ) \cup \bigcup_{(m,k) \in \overline{Z}_{A,i}} (  \tilde{\mathcal{R}}_{A,m,i} \times 
\tilde{\mathcal{T}}_{k}^A ).
\end{equation*}
\end{definition}
\begin{lemma}[{\cite[Lemma~3.9,~p.~26]{Kinoshita2019}}]
\label{lemma8.10}
Let $i=1,2$. For fixed $m \in \N \times \Z$, the number of $k \in \Z^2$ such that 
$(m, k) \in \widehat{Z}_{A,i}$ is less than $2^{1000}$. On the other hand, for 
fixed $k \in \Z^2$, the number of $m \in \N \times \Z$ such that 
$(m, k) \in \widehat{Z}_{A,i}$ is less than $2^{1000}$. 
Furthermore, the claim holds true if we replace $\widehat{Z}_{A,i}$ by $\overline{Z}_{A,i}$ in 
the above statements.
\end{lemma}
We establish \eqref{est02-prop8.1} under the case $(\ell_1,\ell_2) \in (\mathcal{K}_1 \cup \mathcal{K}_2 ) \times \mathcal{K}_0$. 
To avoid redundancy, here we treat only the case 
$(\ell_1,\ell_2) \in \mathcal{K}_1 \times \mathcal{K}_0$. 
\begin{proof}[Proof of \eqref{est02-prop8.1} for the case $(\ell_1,\ell_2) \in\mathcal{K}_1 \times \mathcal{K}_0$.]
The strategy of the proof is the same as that for the case $(\ell_1,\ell_2) \in (\mathcal{K})^c \cap (\mathcal{K}')^c$. 
Let us write
\begin{equation*}
I_A^{m,k} = \left| \int_{*}  g_1|_{\tilde{\mathcal{R}}_{A,m,1}} (\tau_1,\ell_1) 
g_2|_{\tilde{\mathcal{T}}_{k}^A} (\tau_2,\ell_2) 
g_3 (\tau_3,\ell_3) (d \tilde{\sigma}_1)_\lambda (d \tilde{\sigma}_2)_\lambda \right|.
\end{equation*}
By the definitions of $\widehat{Z}_{A,1}$ and $\overline{Z}_{A,1}$, we observe that 
\begin{align*}
& \left| \int_{*}  g_1 (\tau_1,\ell_1) 
g_2(\tau_2,\ell_2) 
g_3(\tau_3,\ell_3) (d \tilde{\sigma}_1)_\lambda (d \tilde{\sigma}_2)_\lambda \right| \\
& \lesssim \sum_{2^{10} \leq A \leq N_1} 
\sum_{(m,k) \in \widehat{Z}_A} I_{A}^{m,k} + 
\sum_{(m,k) \in \overline{Z}_{N_1}} I_{N_1}^{m,k}.
\end{align*}
As in the proof for the case $(\ell_1,\ell_2) \in (\mathcal{K})^c \cap (\mathcal{K}')^c$, 
the first term is estimated by Proposition \ref{prop8.6} and Lemma \ref{lemma8.10}, and 
the second is estimated by Lemmas \ref{lemma8.4} and \ref{lemma8.10}. 
We omit the details.
\end{proof}

Next we show \eqref{est02-prop8.1} for the case (Ib) $|\sin \angle(\ell_1,\ell_2)| \ll 1$. This case requires an angular decomposition. 
We cover the unit circle with the sets
\begin{equation*}
\Theta_j^A = [\frac{\pi}{A}(j-2), \frac{\pi}{A}(j+2)] \cup [ - \pi + \frac{\pi}{A}(j-2), - \pi + \frac{\pi}{A}(j+2)].
\end{equation*} 
Angles from these sets give rise to the following covering of the plane:
\begin{equation*}
\mathfrak{D}_j^A = \{ r(\cos \theta, \sin \theta) \in \R^2 \, | \, \theta \in \Theta_j^A \text{ and } r \in [0,\infty) \}.
\end{equation*}
We set $\tilde{\mathfrak{D}}_j^A = \R \times \mathfrak{D}_j^A$.
 
Recall that it is assumed $\max(|\ell_{1,1}+\ell_{1,2}|, |\ell_{2,1}+\ell_{2,2}|) \geq 2^{-6} N_1$, which means 
$(\ell_1,\ell_2) \notin  {\mathfrak{D}}_{2^9 \times 3}^{2^{11}} \times {\mathfrak{D}}_{2^9 \times 3}^{2^{11}}$. 
The proof is divided into two cases:\\
Case 1. $(\ell_1,\ell_2) \in \displaystyle{\bigcup_{{\substack{0 \leq j \leq 2^{11}-1\\{j \not= 0, 2^9 \times 3,2^{10}}}}}  \left(
{\mathfrak{D}}_{j}^{2^{11}} \times {\mathfrak{D}}_{j}^{2^{11}}\right)}$,\\
Case 2. $(\ell_1,\ell_2) \in \left( {\mathfrak{D}}_{0}^{2^{11}} \times {\mathfrak{D}}_{0}^{2^{11}} \right) 
\cup \left( {\mathfrak{D}}_{2^{10}}^{2^{11}} \times {\mathfrak{D}}_{2^{10}}^{2^{11}} \right)$.

We begin with Case 1. 
It suffices to show \eqref{est02-prop8.1} under the assumption 
$(\ell_1,\ell_2) \in {\mathfrak{D}}_{j}^{2^{11}} \times {\mathfrak{D}}_{j}^{2^{11}}$ 
with fixed $j \not= 0, 2^9 \times 3,2^{10}$. 
Further, since $|\sin \angle(\ell_1,\ell_2)| \ll 1$, we may assume 
\begin{equation}
(\ell_1,\ell_2) \in \bigcup_{2^{100} \leq A \leq N_1} \ \bigcup_{\tiny{\substack{(j_1,j_2) \in J_{A}^{j}\\ 16 \leq |j_1 - j_2|\leq 32}}} 
{\mathfrak{D}}_{j_1}^A \times {\mathfrak{D}}_{j_2}^A 
\cup \bigcup_{\tiny{\substack{(j_1,j_2) \in J_{N_1}^{j}\\|j_1 - j_2|\leq 16}}} 
{\mathfrak{D}}_{j_1}^{N_1} \times {\mathfrak{D}}_{j_2}^{N_1},\label{AngularDecompositionj}
\end{equation}
where 
\begin{equation*}
J_{A}^{j} = \{ (j_1, j_2) \, | \, 0 \leq j_1,j_2 \leq A -1, \ \left( {\mathfrak{D}}_{j_1}^A \times {\mathfrak{D}}_{j_2}^A \right) \subset {\mathfrak{D}}_{j}^{2^{11}} \times {\mathfrak{D}}_{j}^{2^{11}} \}.
\end{equation*}
\begin{proposition}\label{prop8.11}
Let $2^{100} \leq A \leq N_1$. 
Assume that $1 \ll N_3 \lesssim N_2 \leq N_1$, $L_{\textnormal{med}} \leq N_1^2$, 
$(j_1,j_2) \in J_{A}^{j}$ such that $16 \leq |j_1 - j_2|\leq 32$ and 
\eqref{Assumptionfunctiong}. Then we have
\begin{equation*}
\begin{split}
&\quad \left| \int_{*}  g_1|_{\tilde{{\mathfrak{D}}}_{j_1}^A} (\tau_1,\ell_1) 
g_2|_{\tilde{\mathfrak{D}}_{j_2}^A} (\tau_2,\ell_2) 
g_3 (\tau_3,\ell_3) (d \tilde{\sigma}_1)_\lambda (d \tilde{\sigma}_2)_\lambda \right| \\
& \lesssim L_{\min}^{\frac{1}{2}} 
\overline{C}(A, N_1, N_3, L_{\max})
\|g_1|_{\tilde{{\mathfrak{D}}}_{j_1}^A} \|_{L_\tau^2 L^2_{(d \ell)_\lambda}} \| g_2|_{\tilde{\mathfrak{D}}_{j_2}^A} \|_{L_\tau^2 L^2_{(d \ell)_\lambda}} \| g_3 \|_{L_\tau^2 L^2_{(d \ell)_\lambda}},
\end{split}
\end{equation*}
where
\begin{equation*}
\overline{C}(A, N_1, N_3, L_{\max}) = 
\begin{cases}
N_1^{-1} N_3^{\frac{1}{2}} L_{\max}^{\frac{1}{2}} \qquad & \textnormal{for} \ \, N_3 \geq 2^{30} A^{-1} N_1,\\
\LR{A^{\frac{1}{2}} N_1^{-1} L_{\max}^{\frac{1}{2}}} \qquad &  \textnormal{for} \ \, N_3 \leq 2^{30} A^{-1} N_1.
\end{cases}
\end{equation*}
\end{proposition}
\begin{proof}
First we assume $N_3 \geq 2^{30}A^{-1} N_1$. In this case, for 
$(\xi_1, \eta_1) \times  (\xi_2 ,\eta_2) \in {\mathfrak{D}}_{j_1}^A \times {\mathfrak{D}}_{j_2}^A$ a simple calculation yields
\begin{equation*}
|\Phi (\xi_1, \eta_1, \xi_2 ,\eta_2)| \gtrsim A^{-1} N_1^3.
\end{equation*}
To see this, we put $r_1=|(\xi_1,\eta_1)|$, $r_2 =|(\xi_2,\eta_2)|$. $\theta_1$, $\theta_2 \in [0,2\pi)$ denote angular variables defined by 
\begin{equation*}
(\xi_1,\eta_1) = r_1 (\cos \theta_1, \sin \theta_1), \quad 
(\xi_2,\eta_2) = r_2 (\cos \theta_2, \sin \theta_2).
\end{equation*}
Recall that $(\xi_1,\eta_1) \times (\xi_2,\eta_2) \notin {\mathfrak{D}}_{2^9 \times 3}^{2^{11}} \times {\mathfrak{D}}_{2^9 \times 3}^{2^{11}}$ is assumed. 
Thus without loss of generality, we may 
assume that $(\xi_1,\eta_1) \notin  {\mathfrak{D}}_{2^{9} \times 3}^{2^{11}}$ which provides $|\cos \theta_1 + \sin \theta_1| = \sqrt{2} |\sin (\theta_1 + \pi/4)| > 2^{-11}\pi$. 
We deduce from the assumption $|j_1-j_2| \leq 32$ that 
$ |(\cos \theta_1,\sin \theta_1 ) - (\cos \theta_2, \sin \theta_2)| \leq 2^{7} A^{-1}$ or $ |(\cos \theta_1,\sin \theta_1 ) + (\cos \theta_2, \sin \theta_2)| \leq 2^{7} A^{-1}$. 
If $ |(\cos \theta_1,\sin \theta_1 ) - (\cos \theta_2, \sin \theta_2)| \leq 2^{7} A^{-1}$, it is observed that 
\begin{align*}
|\Phi(\xi_1,\eta_1,\xi_2,\eta_2) |= & |\xi_1\xi_2 (\xi_1 +\xi_2) + \eta_1 \eta_2(\eta_1+\eta_2)|\\
\geq  & r_1 r_2 (r_1 + r_2)|\cos^3 \theta_1+\sin^3 \theta_1 |- 2^{9} A^{-1}r_1 r_2 (r_1 + r_2)\\
= &  r_1 r_2 (r_1 + r_2) (1- 2^{-1}\sin 2\theta_1) | \cos \theta_1 + \sin \theta_1 | \\
&\qquad  -2^{9} A^{-1}r_1 r_2 (r_1 + r_2).
\end{align*}
Clearly, this implies 
$|\Phi(\xi_1,\eta_1,\xi_2,\eta_2) | \gtrsim N_1^3$. 
Similarly, for the case $ |(\cos \theta_1,\sin \theta_1 ) + (\cos \theta_2, \sin \theta_2)| \leq 2^{7} A^{-1}$, we calculate
\begin{align*}
|\Phi(\xi_1,\eta_1,\xi_2,\eta_2)|= & |\xi_1\xi_2 (\xi_1 +\xi_2) + \eta_1 \eta_2(\eta_1+\eta_2)|\\
\geq  & r_1 r_2 (r_1 - r_2)|\cos^3 \theta_1+\sin^3 \theta_1 |- 2^{10} A^{-1}N_1^3\\ 
\geq & 2^{-13} r_1 r_2 (r_1 - r_2) -  2^{10} A^{-1}N_1^3.
\end{align*}
Then it suffices to show $|r_1-r_2| \geq 2^{27} A^{-1} N_1$. Since $N_3 \geq 2^{30} A^{-1}N_1$, without loss of generality, we can assume 
$|\xi_1 + \xi_2| = |r_1 \cos \theta_1 + r_2 \cos \theta_2| \geq 2^{28} A^{-1} N_1$. We see
\begin{align*}
|r_1-r_2| & \geq |r_1 \cos \theta_1 - r_2 \cos \theta_1|\\
& \geq  |r_1 \cos \theta_1 + r_2 \cos \theta_2| - 2^{10} A^{-1} N_1\\
& \geq 2^{27} A^{-1} N_1.
\end{align*}
This completes the proof of $|\Phi (\xi_1, \eta_1, \xi_2 ,\eta_2)| \gtrsim A^{-1} N_1^3$ which yields 
$L_{\max} \gtrsim A^{-1} N_1^3$. 
Consequently, it follows from Lemma \ref{lemma8.4} with $P \sim \lambda^2 A^{-1} N_1 N_3$ that 
\begin{equation*}
\begin{split}
&\quad \left| \int_{*}  g_1|_{\tilde{{\mathfrak{D}}}_{j_1}^A} (\tau_1,\ell_1) 
g_2|_{\tilde{\mathfrak{D}}_{j_2}^A} (\tau_2,\ell_2) 
g_3 (\tau_3,\ell_3) (d \tilde{\sigma}_1)_\lambda (d \tilde{\sigma}_2)_\lambda \right| \\
& \lesssim  N_1^{-1} N_3^{\frac{1}{2}} L_{\min}^{\frac{1}{2}}  L_{\max}^{\frac{1}{2}}
\|g_1|_{\tilde{{\mathfrak{D}}}_{j_1}^A} \|_{L_\tau^2 L^2_{(d \ell)_\lambda}} \| g_2|_{\tilde{\mathfrak{D}}_{j_2}^A} \|_{L_\tau^2 L^2_{(d \ell)_\lambda}} \| g_3 \|_{L_\tau^2 L^2_{(d \ell)_\lambda}}.
\end{split}
\end{equation*}

Next we assume $N_3 \leq 2^{30}A^{-1} N_1$. 
This case is treated by Proposition \ref{nlw-ZKSymmetrized}. 
To utilize Proposition \ref{nlw-ZKSymmetrized}, we only need to show
\begin{equation*}
\bigl|
(\xi_1 \eta_2 - \xi_2 \eta_1) \, 
\bigl(\xi_1 \eta_2 + \xi_2 \eta_1 + 2(\xi_1 \eta_1 + \xi_2 \eta_2)\bigr)
\bigr| \gtrsim A^{-1} N_1^4,
\end{equation*}
for 
$(\xi_1, \eta_1) \times  (\xi_2 ,\eta_2) \in {\mathfrak{D}}_{j_1}^A \times {\mathfrak{D}}_{j_2}^A$. 
Since $|\xi_1 \eta_2 - \xi_2 \eta_1| \gtrsim A^{-1}N_1^2$ is clear, it suffices to show 
$|\xi_1 \eta_2 + \xi_2 \eta_1 + 2(\xi_1 \eta_1 + \xi_2 \eta_2)| \gtrsim N_1^2$. 
Let us recall
\begin{equation*}
(\ell_1,\ell_2) \notin \left( {\mathfrak{D}}_{0}^{2^{11}} \times {\mathfrak{D}}_{0}^{2^{11}} \right) 
\cup \left( {\mathfrak{D}}_{2^{10}}^{2^{11}} \times {\mathfrak{D}}_{2^{10}}^{2^{11}} \right),
\end{equation*}
which suggests that we may assume $|\xi_1 \eta_1| \geq 2^{-15} N_1^2$. 
Thus it is observed that
\begin{align*}
|\xi_1 \eta_2 + \xi_2 \eta_1 + 2(\xi_1 \eta_1 + \xi_2 \eta_2)| & \geq 
2|\xi_1 \eta_1 + \xi_2 (\eta_1 + \eta_2)| - |\xi_1 \eta_2 - \xi_2 \eta_1|\\
& \geq 2^{-14} N_1^2 - 2^3 N_1 N_3 - 2^{10} A^{-1} N_1^2 \geq 2^{-15} N_1^2.
\end{align*}
This completes the proof. 
\end{proof}
\begin{proof}[Proof of \eqref{est02-prop8.1} for the case $(\ell_1,\ell_2) \in {\mathfrak{D}}_{j}^{2^{11}} \times {\mathfrak{D}}_{j}^{2^{11}}$ with fixed $j \not= 0, 2^9 \times 3,2^{10}$.]
By using
\begin{equation*}
I_A^{j_1,j_2} = \left| \int_{*}  g_1|_{\tilde{{\mathfrak{D}}}_{j_1}^A} (\tau_1,\ell_1) 
g_2|_{\tilde{\mathfrak{D}}_{j_2}^A} (\tau_2,\ell_2) 
g_3 (\tau_3,\ell_3) (d \tilde{\sigma}_1)_\lambda (d \tilde{\sigma}_2)_\lambda \right|,
\end{equation*}
since \eqref{AngularDecompositionj}, we can write
\begin{align*}
&\quad \left| \int_{*}  g_1|_{\tilde{{\mathfrak{D}}}_{j}^{2^{11}}} (\tau_1,\ell_1) 
g_2|_{\tilde{{\mathfrak{D}}}_{j}^{2^{11}}}(\tau_2,\ell_2) 
g_3(\tau_3,\ell_3) (d \tilde{\sigma}_1)_\lambda (d \tilde{\sigma}_2)_\lambda \right| \\
& \lesssim \bigl(\sum_{2^{100} \leq A \leq 2^{30} N_1/N_3} + \sum_{2^{30}N_1/N_3 \leq A \leq N_1} \bigr)
\sum_{\substack{(j_1,j_2) \in J_{A}^{j}\\ 16 \leq |j_1 - j_2|\leq 32}} I_A^{j_1,j_2} + 
\sum_{\substack{(j_1,j_2) \in J_{N_1}^{j}\\|j_1 - j_2|\leq 16}} I_{N_1}^{j_1,j_2}.
\end{align*}
For the first term, it should be noted that we may assume $A \geq N_1/N_3$, otherwise 
$I_A^{j_1,j_2}$ with $16 \leq |j_1-j_2|$ vanishes. 
By using Proposition \ref{prop8.11}, we obtain
\begin{align*}
& \bigl(\sum_{N_1/N_3 \leq A \leq 2^{30} N_1/N_3} + \sum_{2^{30}N_1/N_3 \leq A \leq N_1} \bigr)
\sum_{\substack{(j_1,j_2) \in J_{A}^{j}\\ 16 \leq |j_1 - j_2|\leq 32}} I_A^{j_1,j_2}\\
& \lesssim L_{\min}^{\frac{1}{2}}\bigl( \sum_{A \sim N_1/N_3} 
\LR{A^{\frac{1}{2}} N_1^{-1} L_{\max}^{\frac{1}{2}}} 
+ \sum_{2^{30}N_1/N_3 \leq A \leq N_1} N_1^{-1} N_3^{\frac{1}{2}} L_{\max}^{\frac{1}{2}}  \bigr)
\prod_{i=1}^3 \|g_i \|_{L_\tau^2 L^2_{(d \ell)_\lambda}} \\
& \lesssim L_{\min}^{\frac{1}{2}} (
\LR{N_1^{-\frac{1}{2}} N_3^{-\frac{1}{2}} L_{\max}^{\frac{1}{2}}} 
+ (\log N_1)N_1^{-1} N_3^{\frac{1}{2}} L_{\max}^{\frac{1}{2}}  \bigr)
\prod_{i=1}^3 \|g_i \|_{L_\tau^2 L^2_{(d \ell)_\lambda}}.
\end{align*}

For the second term, we find $|\Phi (\xi_1, \eta_1, \xi_2 ,\eta_2)| \gtrsim  N_1^2 N_3$ as in the proof of 
Proposition \ref{prop8.11}. 
Then, by using Lemma \ref{lemma8.4} with $P \sim \lambda^2 N_3$, we can verify \eqref{est02-prop8.1}.
\end{proof}
Let us turn to Case 2: 
$(\ell_1,\ell_2) \in \left( {\mathfrak{D}}_{0}^{2^{11}} \times {\mathfrak{D}}_{0}^{2^{11}} \right) 
\cup \left( {\mathfrak{D}}_{2^{10}}^{2^{11}} \times {\mathfrak{D}}_{2^{10}}^{2^{11}} \right)$. 
We only consider the case $(\ell_1,\ell_2) \in {\mathfrak{D}}_{0}^{2^{11}} \times {\mathfrak{D}}_{0}^{2^{11}}$ to reduce redundancy. 
\begin{proposition}\label{prop8.12}
Let $\max(2^{100}, \, N_1/N_3) \leq A \leq N_1$. 
Assume that $1 \ll N_3 \lesssim N_2 \leq N_1$, $L_{\textnormal{med}} \leq N_1^2$, 
$(j_1,j_2) \in J_{A}^{0}$ such that $16 \leq |j_1 - j_2|\leq 32$ and 
\eqref{Assumptionfunctiong}. Then we have
\begin{equation*}
\begin{split}
& \left| \int_{*}  g_1|_{\tilde{{\mathfrak{D}}}_{j_1}^A} (\tau_1,\ell_1) 
g_2|_{\tilde{\mathfrak{D}}_{j_2}^A} (\tau_2,\ell_2) 
g_3 (\tau_3,\ell_3) (d \tilde{\sigma}_1)_\lambda (d \tilde{\sigma}_2)_\lambda \right| \\
& \lesssim L_{\min}^{\frac{1}{2}} 
\widehat{C}(A, N_1, N_3, L_{\max})
\|g_1|_{\tilde{{\mathfrak{D}}}_{j_1}^A} \|_{L_\tau^2 L^2_{(d \ell)_\lambda}} \| g_2|_{\tilde{\mathfrak{D}}_{j_2}^A} \|_{L_\tau^2 L^2_{(d \ell)_\lambda}} \| g_3 \|_{L_\tau^2 L^2_{(d \ell)_\lambda}},
\end{split}
\end{equation*}
where
\begin{equation*}
\widehat{C}(A, N_1, N_3, L_{\max}) = 
\begin{cases}
N_1^{-1} N_3^{\frac{1}{2}} L_{\max}^{\frac{1}{2}} \qquad & \textnormal{for} \ \, N_3 \geq 2^{30} A^{-1} N_1,\\
N_3^{\ep} \LR{ N_1^{-\frac{1}{2}} L_{\max}^{\frac{1}{2}}} \qquad &  \textnormal{for} \ \, N_3 \leq 2^{30} A^{-1} N_1.
\end{cases}
\end{equation*}
\end{proposition}
\begin{proof}
The case $N_3 \geq 2^{30} A^{-1} N_1$ can be handled in the same manner as in the proof of Proposition \ref{prop8.11}. 
We focus on the case $N_3 \leq 2^{30} A^{-1} N_1$, which means $A \leq 2^{30} N_1/N_3$. 
Put $A_0=2^{30}N_1/N_3$ and for a dyadic 
$K$ such that $2^{10} \leq K \leq 2^{-10}A_0$, we define
\begin{align*}
\mathfrak{J}_{A_0}^K & = \left\{ j \in \N \ | \ \frac{A_0}{K} \leq j \leq 2 \frac{A_0}{K}, \quad 
A_0- 2 \frac{A_0}{K} \leq j \leq A_0-  \frac{A_0}{K} \right\},\\
\mathfrak{J}_{A_0} & = \left\{ j \in \N \ | \ 0 \leq j \leq 2^{10}, \quad A_0- 2^{10}  \leq j \leq A_0- 1 \right\}.
\end{align*}
Since $N_1/N_3 \leq A \leq 2^{30}N_1/N_3 =A_0$, it suffices to show that for 
$j_1 \in \mathfrak{J}_{A_0}^K \cup \mathfrak{J}_{A_0}$ and $|j_1-j_2| \sim 1$, it holds
\begin{equation}
\begin{split}
&\quad \left| \int_{*}  g_1|_{\tilde{{\mathfrak{D}}}_{j_1}^{A_0}} (\tau_1,\ell_1) 
g_2|_{\tilde{\mathfrak{D}}_{j_2}^{A_0}} (\tau_2,\ell_2) 
g_3 (\tau_3,\ell_3) (d \tilde{\sigma}_1)_\lambda (d \tilde{\sigma}_2)_\lambda \right| \\
& \lesssim N_3^{\ep} L_{\min}^{\frac{1}{2}}  \LR{N_1^{-\frac{1}{2}}  L_{\max}^{\frac{1}{2}}} 
\|g_1|_{\tilde{{\mathfrak{D}}}_{j_1}^{A_0}} \|_{L_\tau^2 L^2_{(d \ell)_\lambda}} 
\| g_2|_{\tilde{\mathfrak{D}}_{j_2}^{A_0}} \|_{L_\tau^2 L^2_{(d \ell)_\lambda}} \| g_3 \|_{L_\tau^2 L^2_{(d\ell)_\lambda}}.
\end{split}\label{est01-prop8.12}
\end{equation}
We divide the proof of \eqref{est01-prop8.12} into four cases:\\
(1) $1 \ll N_3 \lesssim N_1^{\frac{1}{2}}$, $j_1 \in \mathfrak{J}_{A_0}^K$ with $1 \ll K \ll N_3$,\\
(2) $1 \ll N_3 \lesssim N_1^{\frac{1}{2}}$, $j_1 \in \mathfrak{J}_{A_0}^K \cup \mathfrak{J}_{A_0} $
 with $N_3 \lesssim K \lesssim A_0$,\\
(3) $N_1^{\frac{1}{2}} \lesssim N_3 \ll N_1$, 
$j_1 \in \mathfrak{J}_{A_0}^K$ with $1 \ll K \ll A_0$,\\
(4) $N_1^{\frac{1}{2}} \lesssim N_3 \ll N_1$, 
$j_1 \in \mathfrak{J}_{A_0}^K\cup \mathfrak{J}_{A_0} $
 with $ K \sim A_0$.\\
\textbf{\underline{Case (1)}} 
Note that $j_1 \in \mathfrak{J}_{A_0}^K$ with $|j_1-j_2| \sim 1$ implies $|\eta_1|+|\eta_2| \lesssim K^{-1}N_1$ for $(\xi_1,\eta_1) \times (\xi_2,\eta_2) \in \mathfrak{D}_{j_1}^{A_0} \times \mathfrak{D}_{j_2}^{A_0}$. 
We define the sets 
\begin{align*}
S_{a} & =\{(\tau, \ell_{(1)}, \ell_{(2)}) \in \R \times M(\Z^2 / \lambda) \, |\, |\ell_{(1)}| \leq a^{-1}N_3\},\\
\tilde{S}_{a} & =\{(\tau, \ell_{(1)}, \ell_{(2)}) \in \R \times M(\Z^2 / \lambda) \, |\, a^{-1} N_3 \leq |\ell_{(1)}| \leq 2 a^{-1}N_3\}.
\end{align*} 
First we assume $\supp g_3 \subset S_{2^{-30}K}$ and prove \eqref{est01-prop8.12} by Proposition \ref{nlw-ZKSymmetrized} with $A = A_0 K \sim  K N_1/N_3$. 
We deduce from $|j_1-j_2| \sim 1$ and $\supp g_3 \subset S_{2^{-10}K}$ that, 
after harmless decompositions, we can assume that for $i=1,2,3$, 
$\supp_k g_i$ are confined to the rectangle set 
\begin{equation*}
K_i= 
\{ (\xi, \eta) \in \R^2 \, | \, |\xi-\alpha_i|\ll K^{-1}N_3 \sim (A_0K)^{-1} N_1, \ |\eta - \beta_i| \ll 
A_0^{-1} N_1\},
\end{equation*}
with some fixed $(\alpha_i,\beta_i) \in \R^2$ such that $|\beta_i| \lesssim K^{-1}N_1$, respectively.\\
Since $\partial_{\eta} \tilde{\psi} (\xi,\eta) = 3 \eta^2$, this implies \eqref{AssumptionSurfaceRegularitySymmetrized} with $A=A_0 K$. 
Next we show the transversality condition \eqref{AssumptionSurfaceTransversalitySymmetrized}. 
It is clear that $(\xi_1,\eta_1) \times (\xi_2,\eta_2) \in \mathfrak{D}_{j_1}^{A_0} \times \mathfrak{D}_{j_2}^{A_0}$ gives $|\xi_1 \eta_2-\xi_2 \eta_1| \gtrsim A_0^{-1}N_1^2$. Furthermore, we see
\begin{align*}
|\xi_1 \eta_2 + \xi_2 \eta_1 + 2(\xi_1 \eta_1 + \xi_2 \eta_2)| & \geq 
2|\xi_1 \eta_1 + (\xi_1+\xi_2) \eta_2| - |\xi_1 \eta_2 - \xi_2 \eta_1|\\
& \geq 2^{-3} K^{-1} N_1^2 - 2^{40} K^{-2} N_1 N_3 - 2^{10} A_0^{-1} N_1^2\\
& \gtrsim K^{-1}N_1^2.
\end{align*}
Here we used $K \ll A_0$. Hence we can utilize Proposition \ref{nlw-ZKSymmetrized} with $A = A_0 K \sim  K N_1/N_3$ and obtain
\begin{align*}
&\quad \left| \int_{*}  g_1|_{\tilde{\mathfrak{D}}_{j_1}^{A_0}}(\tau_1,\ell_1)  
g_2|_{\tilde{\mathfrak{D}}_{j_2}^{A_0}}(\tau_2,\ell_2)  g_3 (\tau_3,\ell_3) (d \tilde{\sigma}_1)_\lambda (d \tilde{\sigma}_2)_\lambda \right|\\
 & \lesssim L_{\min}^{\frac{1}{2}} 
\LR{K^{\frac{1}{2}} N_1^{-\frac{1}{2}}N_3^{-\frac{1}{2}}L_{\max}^{\frac{1}{2}}} 
\|g_1 \|_{L_\tau^2 L^2_{(d \ell)_\lambda}} \| g_2 \|_{L_\tau^2 L^2_{(d \ell)_\lambda}} \| g_3 \|_{L_\tau^2 L^2_{(d \ell)_\lambda}}.
\end{align*}
Next suppose $\supp g_3 \subset (S_{2^{-30}K})^c$. Then we easily observe that $|\Phi| \geq N_1^2$ 
which, combined with Lemma \ref{lemma8.4}, provides 
the desired estimate since $N_3 \lesssim N_1^{1/2}$.\\
\textbf{\underline{Case (2)}} To avoid redundancy, we only treat the case $j_1 \in \mathfrak{J}_{A_0}^K$. 
Assume $\supp g_3 \subset S_{2^{-10}K^{2}}$. Since $N_3 \lesssim K$, we have $K^{-2} N_3^2 \lesssim 1$. 
Therefore, by Lemmas \ref{liouvilleSymmetrized} and \ref{lemma8.4}, we get
\begin{equation*}
\begin{split}
&\quad \left| \int_{*}  g_1|_{\tilde{\mathfrak{D}}_{j_1}^{A_0}}(\tau_1,\ell_1)  
g_2|_{\tilde{\mathfrak{D}}_{j_2}^{A_0}}(\tau_2,\ell_2)  g_3 (\tau_3,\ell_3) (d \tilde{\sigma}_1)_\lambda (d \tilde{\sigma}_2)_\lambda \right| \\
&\lesssim  L_{\min}^{\frac{1}{2}} 
\|g_1 \|_{L_\tau^2 L^2_{(d \ell)_\lambda}} \| g_2 \|_{L_\tau^2 L^2_{(d \ell)_\lambda}} \| g_3 \|_{L_\tau^2 L^2_{(d \ell)_\lambda}}.
\end{split}
\end{equation*}
In the case $\supp g_3 \subset \tilde{S}_{\alpha^{-1} K^2}$ with $2^{10} \leq \alpha \lesssim K^2$, we can observe that $|\Phi| \gtrsim \alpha K^{-2} N_1^2 N_3$. 
In addition, Lemma \ref{liouvilleSymmetrized} provides 
$\# \supp_{k} g_3 \lesssim \lambda^2 \alpha K^{-2} N_3^2$. Hence, by employing Lemma \ref{lemma8.4}, we have
\begin{equation*}
\begin{split}
&\quad \left| \int_{*}  g_1|_{\tilde{\mathfrak{D}}_{j_1}^{A_0}}(\tau_1,\ell_1)  
g_2|_{\tilde{\mathfrak{D}}_{j_2}^{A_0}}(\tau_2,\ell_2)  g_3 (\tau_3,\ell_3) (d \sigma_1)_\lambda (d \sigma_2)_\lambda \right| \\
&\lesssim N_3^{\frac{1}{2}} N_1^{-1} L_{\min}^{\frac{1}{2}} L_{\max}^{\frac{1}{2}}
\|g_1 \|_{L_\tau^2 L^2_{(d \ell)_\lambda}} \| g_2 \|_{L_\tau^2 L^2_{(d \ell)_\lambda}} \| g_3 \|_{L_\tau^2 L^2_{(d \ell)_\lambda}}.
\end{split}
\end{equation*}
Consequently, for $\supp g_3 \subset (S_{2^{-10}K^{2}})^c$, by summing up the above, we get
\begin{equation*}
\begin{split}
&\quad \left| \int_{*}  g_1|_{\tilde{\mathfrak{D}}_{j_1}^{A_0}}(\tau_1,\ell_1)  
g_2|_{\tilde{\mathfrak{D}}_{j_2}^{A_0}}(\tau_2,\ell_2)  g_3 (\tau_3,\ell_3) (d \tilde{\sigma}_1)_\lambda (d \tilde{\sigma}_2)_\lambda \right| \\
&\lesssim N_3^{\frac{1}{2}} N_1^{-1+\ep} L_{\min}^{\frac{1}{2}} L_{\max}^{\frac{1}{2}}
\|g_1 \|_{L_\tau^2 L^2_{(d \ell)_\lambda}} \| g_2 \|_{L_\tau^2 L^2_{(d \ell)_\lambda}} \| g_3 \|_{L_\tau^2 L^2_{(d \ell)_\lambda}}.
\end{split}
\end{equation*}
\textbf{\underline{Case (3)}} The case $\supp g_3 \subset S_{2^{-30}K}$ can be handled in the same way as in Case (1) and the case $\supp g_3 \subset (S_{2^{-30}K})^c$ is treated as in the proof of Case (2). We omit the details.\\
\textbf{\underline{Case (4)}} Since $K \sim A_0$, we treat only $j_1 \in \mathfrak{J}_{A_0}$ here. 
We will see that Case (4) is the most difficult part in the proof of Proposition \ref{prop8.12} and we need to perform an additional Whitney-type decomposition as in \cite{Kinoshita2019}.

First, we assume $\supp g_3 \subset S_{2^{-10}N_1^2/N_3^2}$. 
We introduce the Whitney-type decomposition of $\R^2 \times \R^2$ into rectangle tiles. 
\begin{definition}
Let $1 \lesssim d \lesssim N_3^2/ N_1$ 
be dyadic and $m = ( m_{(1)}, m_{(2)} ) \in \Z^2$. We define rectangle-tiles 
$\{ \mathcal{R}_m^{d}\}_{m \in \Z^2}$ whose short side is parallel to $\xi$-axis and its length is $ d^{-1} N_1^{-2}N_3^3$, long side length is $ d^{-1} N_3 $ and prisms $\{ \tilde{\mathcal{R}}_m^{d}\}_{m \in \Z^2}$ as follows:
\begin{align*}
& \mathcal{R}_m^{d} : = \{ (\xi,\eta) \in \R^2 \, | \, \xi \in d^{-1} N_1^{-2}N_3^3 
[ m_{(1)}, m_{(1)} + 1), \ \eta \in d^{-1} N_3 [ m_{(2)}, m_{(2)} + 1)  \},\\
& \tilde{\mathcal{R}}_m^{d}  : = \R \times \mathcal{R}_m^{d}.
\end{align*}
\end{definition}
\begin{definition}[Whitney type decomposition]
Let $1 \lesssim d \lesssim N_3^2/ N_1$ be dyadic and $j_1 \in \mathfrak{J}_{A_0}$. Recall that
\begin{align*}
\Phi (\xi_1, \eta_1, \xi_2 ,\eta_2) & = \xi_1 \xi_2(\xi_1 + \xi_2) + \eta_1 \eta_2 (\eta_1 + \eta_2), \\
F (\xi_1, \eta_1, \xi_2 ,\eta_2) & = \xi_1 \eta_2 +  \xi_2 \eta_1 + 2 (\xi_1 \eta_1 + \xi_2 \eta_2).
\end{align*}
We define $Z_{d,j_1,j_2}^1 $ as the set of $(m_1, m_2) \in \Z^2 \times \Z^2 $ such that
\begin{equation*}
\begin{cases}
 |\Phi(\xi_1, \eta_1, \xi_2, \eta_2)| \geq  d^{-1} N_3^3  \ \ \textnormal{for any} \ 
(\xi_1, \eta_1) \times (\xi_2,\eta_2) \in 
\mathcal{R}_{m_1}^{d} \times \mathcal{R}_{m_2}^{d}, \\
 \left( \mathcal{R}_{m_1}^{d} \times \mathcal{R}_{m_2}^{d} \right) 
\cap \left( \mathfrak{D}_{j_1}^{A_0} \times \mathfrak{D}_{j_2}^{A_0} \right) 
\not= \emptyset,\\
 |\xi_1+ \xi_2| \lesssim   N_1^{-2}N_3^3  \ \ \textnormal{for any} \ 
(\xi_1, \eta_1) \times (\xi_2,\eta_2) \in 
\mathcal{R}_{m_1}^{d} \times \mathcal{R}_{m_2}^{d}.
\end{cases}
\end{equation*}
Similarly, we define $Z_{d,j_1,j_2}^2$ 
as the set of $(k_1, k_2) \in \Z^2 \times \Z^2$ such that
\begin{equation*}
\begin{cases}
 |F(\xi_1, \eta_1, \xi_2, \eta_2)| \geq d^{-1} N_1 N_3  \ \ \textnormal{for any} \ 
(\xi_1, \eta_1) \times (\xi,\eta) \in 
\mathcal{R}_{m_1}^{d} \times \mathcal{R}_{m_2}^{d}, \\
 \left( \mathcal{R}_{m_1}^{d} \times \mathcal{R}_{m_2}^{d} \right) 
\cap \left( \mathfrak{D}_{j_1}^{A_0} \times \mathfrak{D}_{j_2}^{A_0} \right) 
\not= \emptyset,\\
 |\xi_1+ \xi_2| \lesssim   N_1^{-2}N_3^3  \ \ \textnormal{for any} \ 
(\xi_1, \eta_1) \times (\xi_2,\eta_2) \in 
\mathcal{R}_{m_1}^{d} \times \mathcal{R}_{m_2}^{d},
\end{cases}
\end{equation*}
and
\begin{equation*}
Z_{d}^{j_1,j_2} = Z_{d,j_1,j_2}^1 \, \cup \, Z_{d,j_1,j_2}^2 , 
\quad R_{d}^{j_1,j_2} = \bigcup_{(m_1, m_2) \in Z_{d}^{j_1,j_2}} 
\mathcal{R}_{m_1}^{d} \times 
\mathcal{R}_{m_2}^{d} \subset \R^2 \times \R^2.
\end{equation*}
It is clear that $d_1 \leq d_2 \Longrightarrow  R_{d_1}^{j_1,j_2} \subset  R_{d_2}^{j_1,j_2}$. 
Further, we define
\begin{equation*}
Q_{d}^{j_1,j_2} = 
\begin{cases}
R_{d}^{j_1,j_2} \setminus R_{d/2}^{j_1,j_2} \quad & \textnormal{for} \  d \geq 2^{21},\\
 \ R_{2^{20}}^{j_1,j_2}  \quad \, \qquad & \textnormal{for} \  d = 2^{20},
\end{cases}
\end{equation*}
and a set of pairs of integer coordinates $\widehat{Z}_{d}^{j_1,j_2} \subset Z_{d}^{j_1,j_2}$ as
\begin{equation*}
\bigcup_{(m_1, m_2) \in \widehat{Z}_{d}^{j_1,j_2}} \mathcal{R}_{m_1}^{d} \times 
\mathcal{R}_{m_2}^{d} = Q_{d}^{j_1,j_2}.
\end{equation*}
Clearly, $\widehat{Z}_{d}^{j_1,j_2}$ is uniquely defined and 
\begin{equation*}
d_1 \not= d_2 \Longrightarrow Q_{d_1}^{j_1,j_2} \cap Q_{d_2}^{j_1,j_2} = \emptyset, \quad 
\bigcup_{2^{20} \leq d \leq d_0} Q_{d}^{j_1,j_2} = R_{d_0}^{j_1,j_2},
\end{equation*}
where $d_0 \gtrsim 1$ is dyadic. 
Lastly, we define $\overline{Z}_{d}^{j_1,j_2}$ as the collection of $(m_1, m_2) \in \Z^2 \times \Z^2$ which satisfies
\begin{equation*}
\begin{cases}
\mathcal{R}_{m_1}^{d} \times \mathcal{R}_{m_2}^{d} \not\subset 
\displaystyle{\bigcup_{2^{20} \leq d' \leq d} 
\bigcup_{(m_1', m_2') \in \widehat{Z}_{A,d'}^{j_1,j_2}}} 
\left( \mathcal{R}_{m_1'}^{A,d'} \times \mathcal{R}_{m_2'}^{A,d'} \right),\\
 \left( \mathcal{R}_{m_1}^{d} \times \mathcal{R}_{m_2}^{d} \right) 
\cap \left( \mathfrak{D}_{j_1}^A \times \mathfrak{D}_{j_2}^A \right) 
\not= \emptyset,\\
 |\xi_1+ \xi_2| \lesssim   N_1^{-2}N_3^3  \ \ \textnormal{for any} \ 
(\xi_1, \eta_1) \times (\xi_2,\eta_2) \in 
\mathcal{R}_{m_1}^{d} \times \mathcal{R}_{m_2}^{d}.
\end{cases}
\end{equation*}
\end{definition}
The following lemma ensures the almost orthogonality of rectangle sets $\mathcal{R}_{m_1}^{d}$ and 
$\mathcal{R}_{m_2}^{d}$ such that $(m_1, m_2) \in \widehat{Z}_{d}^{j_1,j_2}$ or $(m_1, m_2) \in \overline{Z}_{d}^{j_1,j_2}$. We note that the proof is almost the same as that for Lemmas 3.7 and 3.24 in \cite{Kinoshita2019}.
\begin{lemma}\label{lemma8.15}
Let $1 \lesssim d \lesssim N_3^2/ N_1$ be dyadic and $j_1 \in \mathfrak{J}_{A_0}$. 
For fixed $m_1 \in \Z^2$, the number of $m_2 \in \Z^2$ such that 
$(m_1, m_2) \in \widehat{Z}_{d}^{j_1,j_2}$ is less than $2^{1000}$. Furthermore, the same claim holds true if we replace $\widehat{Z}_{d}^{j_1,j_2}$ by $\overline{Z}_{d}^{j_1,j_2}$.
\end{lemma}
\begin{proof}
Clearly, we can assume $d \geq 2^{100}$. 
For $(m_1, m_2) \in \widehat{Z}_{d}^{j_1,j_2}$, we can find 
$m_1'=m_1'(m_1) \in \Z^2$ and $m_2'=m_2'(m_2) \in \Z^2$ which satisfy 
$\mathcal{R}_{m_1}^{d} \subset \mathcal{R}_{m_1'}^{d/2}$ and 
$\mathcal{R}_{m_2}^{d} \subset \mathcal{R}_{m_2'}^{d/2}$, respectively. 
In view of the definitions, $(m_1, m_2) \in \widehat{Z}_{d}^{j_1,j_2}$ implies that there exist 
$(\bar{\xi}_1, \bar{\eta}_1)$, $(\tilde{\xi}_1, \tilde{\eta}_1) \in \mathcal{R}_{m_1'}^{d/2}$, 
$(\bar{\xi}_2, \bar{\eta}_2)$, $(\tilde{\xi}_2, \tilde{\eta}_2) \in \mathcal{R}_{m_2'}^{d/2}$ which satisfy
\begin{equation}
|\Phi(\bar{\xi}_1, \bar{\eta}_1, \bar{\xi}_2, \bar{\eta}_2)| \leq 2 d^{-1} N_3^3  \quad \textnormal{and} \quad 
 |F (\tilde{\xi}_1, \tilde{\eta}_1, \tilde{\xi}_2 ,\tilde{\eta}_2)| \leq 2 d^{-1} N_1 N_3.\label{lemma3.6-est001}
\end{equation}
Define $(\xi_1', \eta_1')$ as the center of $\mathcal{R}_{m_1}^{d}$. Since $j_1 \in \mathfrak{J}_{A_0}$, 
$|j_1-j_2| \sim 1$ and 
$( \mathcal{R}_{m_1}^{d} \times \mathcal{R}_{m_2}^{d} ) 
\cap ( \mathfrak{D}_{j_1}^A \times \mathfrak{D}_{j_2}^A ) 
\not= \emptyset$, we have $|\eta_1| + |\eta_2| \leq N_3$ for any 
$(\xi_1,\eta_1) \times  (\xi_2,\eta_2) \in \mathcal{R}_{m_1}^{d} \times \mathcal{R}_{m_2}^{d} $. 
Therefore, for $(\xi_2,\eta_2) \in \mathcal{R}_{m_2}^{d}$, 
\eqref{lemma3.6-est001} implies 
\begin{equation*}
|\Phi(\xi_1', \eta_1', \xi_2,\eta_2)| \leq 2^5 d^{-1} N_3^3  \quad \textnormal{and} \quad 
 |F (\xi_1', \eta_1', \xi_2,\eta_2)| \leq 2^5 d^{-1} N_1 N_3.
\end{equation*}
Consequently, it suffices to see that there exist 
$\tilde{m}_1, \tilde{m}_2 \in \Z^2$ such that 
$\mathcal{R}_{\tilde{m}_1}^{2^{-20}d} \cup \mathcal{R}_{\tilde{m}_2}^{2^{-20}d}$ contains
\begin{equation*}
\left\{ (\xi_2, \eta_2)  \in \mathfrak{D}_{j_2}^{A_0} \, \left| \, 
\begin{aligned} & |\Phi(\xi_1', \eta_1', \xi_2,\eta_2)| \leq 2^5 d^{-1} N_3^3, \\ 
 &  |F (\xi_1', \eta_1', \xi_2,\eta_2)| \leq 2^5  d^{-1} N_1 N_3,
 \end{aligned} \ |\xi_1'+ \xi_2| \lesssim N_1^{-2}N_3^3 \right.
\right\}.
\end{equation*}
Now let us perform the transformation $\xi_2' = \xi_2 + \xi_1'/2$, $\eta_2' = \eta_2 + \eta_1'/2$ and see that
\begin{align*}
|\Phi(\xi_1', \eta_1', \xi_2, \eta_2) |& = |\xi_1' \xi_2(\xi_1' + \xi_2) + \eta_1' \eta_2 (\eta_1' + \eta_2)|  
\leq  2^5 d^{-1} N_3^3, \\
 |F (\xi_1', \eta_1', \xi_2,\eta_2)| & = |\xi_1' \eta_2 +  \xi_2 \eta_1' + 2 (\xi_1' \eta_1' + \xi_2 \eta_2)| 
\leq 2^5 d^{-1} N_1 N_3,
\end{align*}
are rewritten as
\begin{align}
\tilde{\Phi} (\xi_2' ,\eta_2') & := 
\left| \xi_1' {\xi_2'}^2 + \eta_1' {\eta_2'}^2 - \frac{{\xi_1'}^3 + {\eta_1'}^3}{4} \right| 
\leq 2^5 d^{-1} N_3^3 , \label{info03-lemma3.23} \\
\tilde{F} (\xi_2' ,\eta_2') & := \left| \frac{3}{2} \, \xi_1' \, \eta_1' + 2 \,  \xi_2'  \, \eta_2' \right| 
\leq 2^5 d^{-1} N_1 N_3,\label{info04-lemma3.23}
\end{align}
respectively. It should be noted that $|\xi_1'+ \xi_2| \lesssim N_1^{-2}N_3^3$ provides 
$|\xi_2'| \geq |\xi_2|/2 - |\xi_1' + \xi_2|/2 \geq 2^{-2} N_1$. 
We compute that
\begin{equation}
\eqref{info04-lemma3.23} \Longrightarrow  \left| \eta_2' + \frac{3  \xi_1' \, \eta_1'}{4 \xi_2'} \right|  \leq \frac{2^4 d^{-1} N_1 N_3}{|\xi_2'|} 
\leq 2^{6} d^{-1} N_3.\label{est003-lemma3.23}
\end{equation}
Since $|\eta_1'| \leq  N_3$, \eqref{info03-lemma3.23} and \eqref{est003-lemma3.23} yield
\begin{align}
& \Bigl| \xi_1' {\xi_2'}^2 + \eta_1' {\eta_2'}^2 - \frac{{\xi_1'}^3 + {\eta_1'}^3}{4} \Bigr| 
\leq 2^5 d^{-1} N_3^3 \notag\\
 \xLongrightarrow[ \ ]{\eqref{est003-lemma3.23}} & 
\Bigl| \xi_1' {\xi_2'}^2 + \eta_1' \frac{9 {\xi_1'}^2 {\eta_1'}^2}{16 {\xi_2'}^2} - \frac{{\xi_1'}^3 + {\eta_1'}^3}{4} \Bigr| 
\leq 2^{8} d^{-1} N_3^3.\label{est004-lemma3.23}
\end{align}
We define 
\begin{equation*}
G(\xi_2') := \xi_1' {\xi_2'}^2 +  \frac{9 {\xi_1'}^2 {\eta_1'}^3}{16 {\xi_2'}^2} - 
\frac{{\xi_1'}^3 + {\eta_1'}^3}{4}.
\end{equation*}
It follows from $2^{-2} N_1 \leq |\xi_1'|$, $|\xi_2'| \leq 2 N_1$ and $|\eta_1'| \leq N_3$ that
\begin{equation*}
\left| \left( \frac{dG}{d\xi_2'} \right) 
 (\xi_2') \right| = \left| \frac{2 \xi_1'}{{\xi_2'}^3} \left( {\xi_2'}^4 -  \frac{9 {\xi_1'} {\eta_1'}^3}{16} 
\right) \right| \geq 2^{-5}N_1^2.
\end{equation*}
This and \eqref{est004-lemma3.23} verify that there exist constants 
$c_1(\xi_1',\eta_1')$, $c_2(\xi_1', \eta_1') \in \R$ such that
\begin{equation*}
\min_{i=1,2} |\xi_2' - c_i (\xi_1', \eta_1')| \leq 2^{15} d^{-1} N_1^{-2} N_3^3.
\end{equation*}
This and \eqref{est003-lemma3.23} imply that there exist constants 
$c_1'(\xi_1',\eta_1')$, $c_2'(\xi_1', \eta_1') \in \R$  such that
\begin{equation*}
\min_{i=1,2} |\eta_2' - c_i' (\xi_1', \eta_1')| \leq 2^{7} d^{-1} N_3,
\end{equation*}
which completes the proof.
\end{proof}
\begin{lemma}\label{lemma8.16}
Let $1 \lesssim d \lesssim N_3^2/ N_1$ be dyadic, $j_1 \in \mathfrak{J}_{A_0}$. 
Assume that $|j_1-j_2| \sim 1$ and $(m_1, m_2) \in \widehat{Z}_{d}^{j_1,j_2}$. Then we have
\begin{align*}
& \left| \int_{*}  g_1|_{\tilde{\mathcal{R}}_{m_1}^{d}} (\tau_1,\ell_1) 
g_2|_{\tilde{\mathcal{R}}_{m_2}^{d}} (\tau_2,\ell_2) g_3 (\tau_3,\ell_3) (d \tilde{\sigma}_1)_\lambda (d \tilde{\sigma}_2)_\lambda \right|\\
& \lesssim  L_{\min}^{\frac{1}{2}} 
\LR{(d^{\frac{1}{2}} N_3^{-1}+d^{-\frac{1}{2}} N_1^{-1} N_3^{\frac{1}{2}}) L_{\max}^{\frac{1}{2}}}
\|g_1 \|_{L_\tau^2 L^2_{(d \ell)_\lambda}} \| g_2 \|_{L_\tau^2 L^2_{(d \ell)_\lambda}} \| g_3 \|_{L_\tau^2 L^2_{(d \ell)_\lambda}}.
\end{align*}
\end{lemma}
\begin{proof}
First we assume $|\Phi(\xi_1, \eta_1, \xi_2, \eta_2)| \geq  d^{-1} N_3^3$ for any $(\xi_1, \eta_1) \times (\xi_2,\eta_2) \in 
\mathcal{R}_{m_1}^{d} \times \mathcal{R}_{m_2}^{d}$. 
By Lemma \ref{liouvilleSymmetrized}, we have $\displaystyle{\sup_{\alpha \in \R^2} \# \tilde{R}_{d^{-1}N_3, d^{-1}N_1^{-2}N_3^3}^{\alpha} \lesssim \lambda^2 d^{-2}N_1^{-2}N_3^4}$. Therefore, it follows from Lemma \ref{lemma8.4} that
\begin{equation*}
\begin{split}
& \left| \int_{*}   g_1|_{\tilde{\mathcal{R}}_{m_1}^{d}} (\tau_1,\ell_1) 
g_2|_{\tilde{\mathcal{R}}_{m_2}^{d}} (\tau_2,\ell_2) g_3 (\tau_3,\ell_3) (d \tilde{\sigma}_1)_\lambda (d \tilde{\sigma}_2)_\lambda \right|\\
& \lesssim 
d^{-\frac{1}{2}} N_1^{-1} N_3^{\frac{1}{2}} L_{\min}^{\frac{1}{2}} L_{\max}^{\frac{1}{2}} \prod_{i=1}^3
\|g_i \|_{L_\tau^2 L^2_{(d \ell)_\lambda}}.
\end{split}
\end{equation*}
For the case $ |F(\xi_1, \eta_1, \xi_2, \eta_2)| \geq d^{-1} N_1 N_3$, 
it follows from Proposition \ref{nlw-ZKSymmetrized} with $A = d A_0 N_1/N_3  \sim  d N_1^2/N_3^2$ that 
\begin{equation*}
\begin{split}
&\quad \left| \int_{*}  g_1|_{\tilde{\mathcal{R}}_{m_1}^{d}} (\tau_1,\ell_1) 
g_2|_{\tilde{\mathcal{R}}_{m_2}^{d}} (\tau_2,\ell_2) g_3 (\tau_3,\ell_3) (d \tilde{\sigma}_1)_\lambda (d \tilde{\sigma}_2)_\lambda \right| \\
 &\lesssim  L_{\min}^{\frac{1}{2}} 
\LR{d^{\frac{1}{2}} N_3^{-1} L_{\max}^{\frac{1}{2}}} \prod_{i=1}^3 \|g_i \|_{L_\tau^2 L^2_{(d \ell)_\lambda}}.
\end{split}
\end{equation*}
This completes the proof.
\end{proof}
We turn to show \eqref{est01-prop8.12} for Case (4) under the assumption $\supp g_3 \subset S_{2^{-10}N_1^2/N_3^2}$. 
By the properties of 
$\widehat{Z}_{d}^{j_1,j_2}$ and $\overline{Z}_{d}^{j_1,j_2}$, we observe that
\begin{align*}
&\quad \textnormal{(LHS) of \eqref{est01-prop8.12}} \\
& \leq  \sum_{1 \lesssim d \lesssim  N_3^2/ N_1} \sum_{(m_1, m_2) \in \widehat{Z}_{d}^{j_1,j_2}} 
\left| \int_{*}  g_1|_{\tilde{\mathcal{R}}_{m_1}^{d}} (\tau_1,\ell_1) 
g_2|_{\tilde{\mathcal{R}}_{m_2}^{d}} (\tau_2,\ell_2) g_3 (\tau_3,\ell_3) (d \tilde{\sigma}_1)_\lambda (d \tilde{\sigma}_2)_\lambda \right|\\
&\; + \sum_{(m_1, m_2) \in \overline{Z}_{N_3^2/N_1}^{j_1,j_2}} 
\left| \int_{*}  g_1|_{\tilde{\mathcal{R}}_{m_1}^{N_3^2/N_1}}(\tau_1,\ell_1)  
g_2|_{\tilde{\mathcal{R}}_{m_2}^{N_3^2/N_1}} (\tau_2,\ell_2)  g_3 (\tau_3,\ell_3) (d \tilde{\sigma}_1)_\lambda (d\tilde{\sigma}_2)_\lambda \right|\\
& =: \sum_{1 \lesssim d \lesssim  N_3^2/ N_1} \sum_{(m_1, m_2) \in \widehat{Z}_{d}^{j_1,j_2}}  I_1 + \sum_{(m_1, m_2) \in \overline{Z}_{N_3^2/N_1}^{j_1,j_2}}   I_2.
\end{align*}
The first is estimated by Lemmas \ref{lemma8.15}, \ref{lemma8.16} as
\begin{equation*}
\sum_{1 \lesssim d \lesssim  N_3^2/ N_1} \sum_{(m_1, m_2) \in \widehat{Z}_{d}^{j_1,j_2}}  I_1
 \lesssim (\log N_3)L_{\min}^{\frac{1}{2}} 
\LR{N_1^{-\frac{1}{2}} L_{\max}^{\frac{1}{2}}} 
\prod_{i=1}^3 \|g_i \|_{L_\tau^2 L^2_{(d \ell)_\lambda}}.
\end{equation*}
For the second term, since Lemma \ref{liouvilleSymmetrized} provides 
$\displaystyle{\sup_{\alpha \in \R^2} \# \tilde{R}_{N_1 N_3^{-1}, N_1^{-1}N_3}^{\alpha} \lesssim 1}$, 
Lemmas \ref{lemma8.4} and \ref{lemma8.15} establish
\begin{equation*}
\sum_{(m_1, m_2) \in \overline{Z}_{N_3^2/N_1}^{j_1,j_2}}   I_2 
\lesssim  L_{\min}^{\frac{1}{2}} 
\|g_1 \|_{L_\tau^2 L^2_{(d \ell)_\lambda}} \| g_2 \|_{L_\tau^2 L^2_{(d \ell)_\lambda}} \| g_3 \|_{L_\tau^2 L^2_{(d \ell)_\lambda}}.
\end{equation*}

Lastly, we treat the case $\supp g_3 \subset (S_{2^{-10}N_1^2/N_3^2})^c$. 
Let us assume $\supp g_3 \subset \tilde{S}_{\alpha^{-1} N_1^2/N_3^2}$ with $2^{10} \leq \alpha \lesssim N_1^2/N_3^2$.\\ 
This condition gives $|\Phi| \gtrsim \alpha N_3^3$ and $\# \supp_{k} g_3 \lesssim \lambda^2 \alpha N_3^4/N_1^2$. 
Thus, by Lemma \ref{liouvilleSymmetrized}, we obtain
\begin{equation*}
\begin{split}
&\quad \left| \int_{*}  g_1|_{\tilde{\mathfrak{D}}_{j_1}^{A_0}}(\tau_1,\ell_1)  
g_2|_{\tilde{\mathfrak{D}}_{j_2}^{A_0}}(\tau_2,\ell_2)  g_3 (\tau_3,\ell_3) (d \tilde{\sigma}_1)_\lambda (d \tilde{\sigma}_2)_\lambda \right| \\
&\lesssim N_3^{\frac{1}{2}} N_1^{-1} L_{\min}^{\frac{1}{2}} L_{\max}^{\frac{1}{2}}
\prod_{i=1}^3 \|g_i \|_{L_\tau^2 L^2_{(d \ell)_\lambda}}.
\end{split}
\end{equation*}
Consequently, if $\supp g_3 \subset (S_{2^{-10}N_1^2/N_3^2})^c$ by summing up the above, we get
\begin{equation*}
\begin{split}
&\quad \left| \int_{*}  g_1|_{\tilde{\mathfrak{D}}_{j_1}^{A_0}}(\tau_1,\ell_1)  
g_2|_{\tilde{\mathfrak{D}}_{j_2}^{A_0}}(\tau_2,\ell_2)  g_3 (\tau_3,\ell_3) (d \tilde{\sigma}_1)_\lambda (d \tilde{\sigma}_2)_\lambda \right| \\
&\lesssim 
N_3^{\frac{1}{2}} N_1^{-1+\ep} L_{\min}^{\frac{1}{2}} L_{\max}^{\frac{1}{2}} \prod_{i=1}^3 \|g_i \|_{L_\tau^2 L^2_{(d \ell)_\lambda}}.
\end{split}
\end{equation*}
\end{proof}
\begin{proof}[Proof of \eqref{est02-prop8.1} for the case $(\ell_1,\ell_2) \in {\mathfrak{D}}_{0}^{2^{11}} \times {\mathfrak{D}}_{0}^{2^{11}}$.]
We can see
\begin{align*}
&\quad \left| \int_{*}  g_1|_{\tilde{{\mathfrak{D}}}_{0}^{2^{11}}} (\tau_1,\ell_1) 
g_2|_{\tilde{{\mathfrak{D}}}_{0}^{2^{11}}}(\tau_2,\ell_2) 
g_3(\tau_3,\ell_3) (d \tilde{\sigma}_1)_\lambda (d \tilde{\sigma}_2)_\lambda \right| \\
& \lesssim \bigl(\sum_{N_1/N_3 \leq A \leq 2^{30} N_1/N_3} + \sum_{2^{30}N_1/N_3 \leq A \leq N_1} \bigr)
\sum_{\substack{(j_1,j_2) \in J_{A}^{0}\\ 16 \leq |j_1 - j_2|\leq 32}} I_A^{j_1,j_2} + 
\sum_{\substack{(j_1,j_2) \in J_{N_1}^{0}\\|j_1 - j_2|\leq 16}} I_{N_1}^{j_1,j_2}.
\end{align*}
It follows from Proposition \ref{prop8.12} that
\begin{align*}
& \bigl(\sum_{N_1/N_3 \leq A \leq 2^{30} N_1/N_3} + \sum_{2^{30}N_1/N_3 \leq A \leq N_1} \bigr)
\sum_{\substack{(j_1,j_2) \in J_{0}^{j}\\ 16 \leq |j_1 - j_2|\leq 32}} I_A^{j_1,j_2}\\
& \lesssim L_{\min}^{\frac{1}{2}}\bigl( \sum_{A \sim N_1/N_3} 
N_3^{\ep} \LR{ N_1^{-\frac{1}{2}} L_{\max}^{\frac{1}{2}}}
+ \sum_{2^{30}N_1/N_3 \leq A \leq N_1} N_1^{-1} N_3^{\frac{1}{2}} L_{\max}^{\frac{1}{2}}  \bigr)
\prod_{i=1}^3 \|g_i \|_{L_\tau^2 L^2_{(d \ell)_\lambda}} \\
& \lesssim L_{\min}^{\frac{1}{2}} (
N_3^{\ep} \LR{ N_1^{-\frac{1}{2}} L_{\max}^{\frac{1}{2}}}
+ (\log N_1)N_1^{-1} N_3^{\frac{1}{2}} L_{\max}^{\frac{1}{2}}  \bigr)
\prod_{i=1}^3 \|g_i \|_{L_\tau^2 L^2_{(d \ell)_\lambda}}.
\end{align*}

The second term is handled in the same manner as in the proof for the case $(\ell_1,\ell_2) \in {\mathfrak{D}}_{j}^{2^{11}} \times {\mathfrak{D}}_{j}^{2^{11}}$ with fixed $j \not= 0, 2^9 \times 3,2^{10}$.
\end{proof}
It remains to prove Proposition \ref{prop8.1} under the assumption $\max(|k_{1,1}|, |k_{2,1}|) \leq 2^{-5} N_1$. 
\begin{proof}[Proof of Proposition \ref{prop8.1} for the case \textnormal{(II)}]
First we treat non-parallel interactions. 
Let $N_1/N_3 \leq A \leq N_1$ and $16 \leq |j_1-j_2| \leq 32$. We prove
\begin{equation}
\begin{split}
&\quad \left| \int_{*}  \bigl(|k_{3,1}| + |k_{1,1}| \frac{N_3}{N_1} \bigr)f_1|_{\tilde{\mathfrak{D}}_{j_1}^{A}} (\tau_1,k_1) 
f_2|_{\tilde{\mathfrak{D}}_{j_2}^{A}}(\tau_2,k_2) 
f_3(\tau_3,k_3) (d \sigma_1)_\lambda (d \sigma_2)_\lambda \right| \\
& \lesssim 
 A^{-\frac{1}{2}} (N_1 N_3)^{\frac{1}{2}} L_{\min}^{\frac{1}{2}}  \LR{N_1^{-\frac{1}{2}} L_{\max}^{\frac{1}{2}}}
\prod_{i=1}^3 \|f_i \|_{L_\tau^2 L^2_{(dk)_\lambda}}.
\end{split}\label{est01-prop0.8}
\end{equation}
By symmetry, we can always assume $|k_{1,1}| \geq |k_{2,1}|$ and then there exists a dyadic number 
$2^5 \leq  \alpha \leq A$ such that $|k_{1,1}| \sim {\alpha}^{-1}N_1$. 
We divide the proof of \eqref{est01-prop0.8} into the two cases 
$|k_{3,2}| \lesssim \alpha A^{-1}N_1$ and $|k_{3,2}| \gg \alpha A^{-1} N_1$. \\
In the first case, we shall see that the condition $|k_{3,2}| \lesssim \alpha A^{-1}N_1$ gives 
$|k_{3,1}| \lesssim A^{-1} N_1$. 
Let $(r_k \cos \theta_k, r_k \sin \theta_k) \in \mathfrak{D}_{j_k}^{A}$, where $k=1,2$, satisfy 
$(r_1 \cos \theta_1, r_1 \sin \theta_1)+ (r_2 \cos \theta_2, r_2 \sin \theta_2) \in \supp_k f_3$. 
Clearly, $|\cos \theta_1+ \cos \theta_2|\lesssim A^{-1}$ and 
$|\cos \theta_1| \lesssim \alpha^{-1}$. 
Further, since $|k_{3,2}| \lesssim \alpha A^{-1}N_1$, it holds $|r_1-r_2| \lesssim \alpha A^{-1} N_1$. 
Therefore, we get
\begin{align}
 |r_1 \cos \theta_1 + r_2 \cos \theta_2|& 
\leq |(r_1-r_2)\cos \theta_1| + r_2|\cos \theta_1 + \cos \theta_2|\notag \\
& \lesssim A^{-1} N_1.\label{est-k-3,1}
\end{align}
Hence, \eqref{est01-prop0.8} is proved by
\begin{equation}
\begin{split}
&\quad  \left| \int_{*}  f_1|_{\tilde{\mathfrak{D}}_{j_1}^{A}} (\tau_1,k_1) 
f_2|_{\tilde{\mathfrak{D}}_{j_2}^{A}}(\tau_2,k_2) 
f_3(\tau_3,k_3) (d \sigma_1)_\lambda (d \sigma_2)_\lambda \right| \\ 
& \lesssim 
A^{\frac{1}{2}} N_1^{-\frac{1}{2}} N_3^{\frac{1}{2}} L_{\min}^{\frac{1}{2}} 
\LR{N_1^{-\frac{1}{2}} L_{\max}^{\frac{1}{2}}} \prod_{i=1}^3 \|f_i \|_{L_\tau^2 L^2_{(dk)_\lambda}}.
\end{split}\label{est02-prop0.8}
\end{equation}
To see this, we decompose $k_{3,2}$ by employing 
\begin{equation*}
\mathbb{S}_A^m = \{ \eta \in \R \, | \, m A^{-1} N_1 \leq |\eta| \leq (m+1) A^{-1} N_1\},
\end{equation*}
where $m \in \N_{0}$. 
Since $|k_{3,2}| \lesssim N_3$, we have
$\displaystyle{ \{k_{3,2}\} \subset \bigcup_{m \lesssim A N_3/N_1 } 
\mathbb{S}_{A}^m}$. Then, for fixed $m$, it suffices to show
\begin{equation*}
\begin{split}
&  \left| \int_{*} \chi_{\mathbb{S}_A^m}(k_{3,2}) f_1|_{\tilde{\mathfrak{D}}_{j_1}^{A}} (\tau_1,k_1) 
f_2|_{\tilde{\mathfrak{D}}_{j_2}^{A}}(\tau_2,k_2) 
f_3(\tau_3,k_3) (d \sigma_1)_\lambda (d \sigma_2)_\lambda \right| \\ 
& \lesssim 
L_{\min}^{\frac{1}{2}}  \LR{A^{\frac{1}{2}} N_1^{-1} L_{\max}^{\frac{1}{2}} }
\prod_{i=1}^3 \|f_i \|_{L_\tau^2 L^2_{(dk)_\lambda}}.
\end{split}
\end{equation*}
This can be obtained by Proposition \ref{nlw-ZK}. We omit the details.

Next we assume $|k_{3,2}| \gg \alpha A^{-1} N_1$. 
Since $|k_{3,2}| \lesssim N_3$ we can assume $A \gg \alpha N_1/N_3$.
The above observation \eqref{est-k-3,1} implies $|k_{3,1}|\lesssim \alpha^{-1} N_3$ and $|k_{3,2}| \sim N_3$. 
Let 
\begin{equation*}
\widehat{\Phi}=\widehat{\Phi}(\xi_1,\eta_1,\xi_2,\eta_2) = 3 \xi_1 \xi_2 (\xi_1+\xi_2) + \xi_1 \eta_2 (2 \eta_1+\eta_2) + \xi_2 \eta_1 (\eta_1 + 2 \eta_2).
\end{equation*} 
For all $(\xi_k, \eta_k) \in \mathfrak{D}_{j_k}^{A}$ such that $(\xi_1+\xi_2, \eta_1+\eta_2) \in \supp_{k} f_3$, we will show $|\widehat{\Phi}| \gtrsim \alpha^{-1} N_1^2 N_3$ which implies $L_{\max} \gtrsim \alpha^{-1} N_1^2 N_3$. To show this, 
we first observe that
\begin{align*}
&\quad |\xi_1 \eta_2 (2 \eta_1+\eta_2) + \xi_2 \eta_1 (\eta_1 + 2 \eta_2)| \\
& = 
\biggl| \frac{3}{2} (\xi_1 \eta_2 + \xi_2 \eta_1)(\eta_1+\eta_2) +\frac{\xi_1 \eta_2-\xi_2 \eta_1}{2} 
( \eta_1 - \eta_2)\biggr|\\
& \geq |(\xi_1 \eta_2 + \xi_2 \eta_1)(\eta_1+\eta_2) |-|(\eta_1-\eta_2) (\xi_1 \eta_2-\xi_2 \eta_1)|\\
& \gtrsim \alpha^{-1} N_1^2 N_3 .
\end{align*}
Here we used $A^{-1} N_1 \ll |\eta_1+\eta_2| \sim N_3$ and 
$|\xi_1 \eta_2-\xi_2 \eta_1| \lesssim A^{-1} N_1^2$. We calculate
\begin{align*}
|\widehat{\Phi}| 
& = |3 \xi_1 \xi_2 (\xi_1+\xi_2) + \xi_1 \eta_2 (2 \eta_1+\eta_2) + \xi_2 \eta_1 (\eta_1 + 2 \eta_2)|\\
& \geq  |\xi_1 \eta_2 (2 \eta_1+\eta_2) + \xi_2 \eta_1 (\eta_1 + 2 \eta_2)| - 3|\xi_1 \xi_2 (\xi_1+\xi_2) | \\ 
& \gtrsim \alpha^{-1}  N_1^2 N_3.
\end{align*}
Note that $\# \supp_k f_3 \lesssim \lambda^2 A^{-1} N_1 N_3$. Consequently, by Lemma \ref{lemma7.3}, we see that 
$|k_{3,1}|\lesssim \alpha^{-1} N_3$, 
$L_{\max} \gtrsim \alpha^{-1} N_1^2 N_3$ and $\# \supp_k f_3 \lesssim \lambda^2 A^{-1} N_1 N_3$ yield \eqref{est01-prop0.8}.

Next we treat parallel interactions. We show the following equation with $|j_1-j_2| \leq 16$.
\begin{equation}
\begin{split}
&\quad \left| \int_{*}  \bigl(|k_{3,1}| + |k_{1,1}| \frac{N_3}{N_1} \bigr)f_1|_{\tilde{\mathfrak{D}}_{j_1}^{N_1}} (\tau_1,k_1) 
f_2|_{\tilde{\mathfrak{D}}_{j_2}^{N_1}}(\tau_2,k_2) 
f_3(\tau_3,k_3) (d \sigma_1)_\lambda (d \sigma_2)_\lambda \right| \\
& \lesssim N_3^{\frac{1}{2}} L_{\min}^{\frac{1}{2}} 
\LR{N_1^{-\frac{1}{2}} L_{\max}^{\frac{1}{2}}} \prod_{i=1}^3 \|f_i \|_{L_\tau^2 L^2_{(dk)_\lambda}}.
\end{split}\label{est03-prop0.8}
\end{equation}
The proof is almost the same as that for \eqref{est01-prop0.8}. 
If $|k_{1,1}| \lesssim 1$, we easily confirm \eqref{est03-prop0.8} since $|k_{3,1}| \sim 1$ and 
$\# \supp_k f_3 \lesssim \lambda^2 N_3$. 
Let $2^5 \leq \alpha \leq N_1$ and suppose 
$|k_{1,1}| \sim \alpha^{-1} N_1$. As for the non-parallel case, the proof is divided into the cases 
$|k_{3,2}| \lesssim \alpha$ and $|k_{3,2}| \gg \alpha$. 
The first is dealt with the observation \eqref{est-k-3,1}, which provides $|k_{3,1}| \lesssim 1$, and $\# \supp_k f_3 \lesssim \lambda^2 N_3$. The second can be handled by the same argument as for the proof of \eqref{est01-prop0.8} in the case $|k_{3,2}| \gg \alpha A^{-1} N_1$. 

Now we complete the proof of Proposition \ref{prop8.1} by using \eqref{est01-prop0.8} and 
\eqref{est03-prop0.8}. 
The assumption $\max(|k_{1,1}|, |k_{2,1}|) \leq 2^{-5} N_1$ suggests 
$(k_1,k_2) \in {\mathfrak{D}}_{2^{4}}^{2^{5}} \times {\mathfrak{D}}_{2^{4}}^{2^{5}} $. 
Let us recall the Whitney type decomposition of angular variables. 
Define 
\begin{equation*}
J_A = \{ (j_1, j_2) \, | \, 0 \leq j_1, j_2 \leq A-1, \quad ( {\mathfrak{D}}_{j_1}^A \times {\mathfrak{D}}_{j_2}^A ) \subset ( {\mathfrak{D}}_{2^{4}}^{2^{5}} \times {\mathfrak{D}}_{2^{4}}^{2^{5}} ).\}
\end{equation*}
It is observed that
\begin{align*}
 {\mathfrak{D}}_{2^{4}}^{2^{5}} \times {\mathfrak{D}}_{2^{4}}^{2^{5}}  
=   \bigcup_{2^{8} \leq A \leq N_1} \ 
\bigcup_{\tiny{\substack{(j_1,j_2) \in J_{A}\\ 16 \leq |j_1 - j_2|\leq 32}}} 
{\mathfrak{D}}_{j_1}^A \times {\mathfrak{D}}_{j_2}^A \cup 
\bigcup_{\tiny{\substack{(j_1,j_2) \in J_{N_1}\\|j_1 - j_2|\leq 16}}} 
{\mathfrak{D}}_{j_1}^{N_1} \times {\mathfrak{D}}_{j_2}^{N_1}.
\end{align*}
Thus, if we write
\begin{equation*}
\tilde{I}_{A}^{j_1,j_2} := \left| \int_{*}  \bigl(|k_{3,1}| + |k_{1,1}| \frac{N_3}{N_1} \bigr)f_1|_{\tilde{\mathfrak{D}}_{j_1}^{A}} (\tau_1,k_1) 
f_2|_{\tilde{\mathfrak{D}}_{j_2}^{A}}(\tau_2,k_2) 
f_3(\tau_3,k_3) (d \sigma_1)_\lambda (d \sigma_2)_\lambda \right|,
\end{equation*}
we have
\begin{align*}
&\quad \left| \int_{*}  \bigl(|k_{3,1}| + |k_{1,1}| \frac{N_3}{N_1} \bigr)f_1|_{\tilde{\mathfrak{D}}_{2^4}^{2^5}} (\tau_1,k_1) 
f_2|_{\tilde{\mathfrak{D}}_{2^4}^{2^5}}(\tau_2,k_2) 
f_3(\tau_3,k_3) (d \sigma_1)_\lambda (d \sigma_2)_\lambda \right|\\
 & \lesssim 
\sum_{N_1/N_3 \leq A \leq N_1} \ 
\sum_{\tiny{\substack{(j_1,j_2) \in J_{A}\\ 16 \leq |j_1 - j_2|\leq 32}}} I_{A}^{j_1,j_2} 
+ 
\sum_{\tiny{\substack{(j_1,j_2) \in J_{N_1}\\|j_1 - j_2|\leq 16}}} I_{N_1}^{j_1,j_2}.
\end{align*}
The first term is handled by \eqref{est01-prop0.8} and the second term is estimated by \eqref{est03-prop0.8}, respectively.
\end{proof}
\begin{remark}
\textit{(i).} We end this section with an example indicating sharpness of Proposition \ref{prop8.1} up to endpoints.
Firstly, consider the symmetrized equation
\begin{equation*}
\partial_t u + (\partial_x^3 + \partial_y^3) u = u (\partial_x + \partial_y) u, \quad (t,x,y) \in \R \times \T^2.
\end{equation*}
In Subsection \ref{subsection:ExamplesNLW} we have seen that the frequencies $(N,-N), \; (N,2N), \; (2N,N)$ yield a fully transverse interaction, i.e., $A \sim 1$ in \eqref{AssumptionSurfaceTransversality}, with vanishing resonance $\Phi = 0$.
We find with $f_i$ supported on the above modes
\begin{equation*}
\begin{split}
&\quad \left| \int_{\R \times \Z^2} f_1(\tau_1,k_1) f_2(\tau_2,k_2) (k_{3,1} + k_{3,2}) f_3(\tau_3,k_3) (d\sigma_1)_1 (d\sigma_2)_1 \right| \\
 &\sim N L_{\min}^{1/2} \prod_{i=1}^3 \Vert f_i \Vert_{L_\tau^2 L^2_{(dk)_1}}.
 \end{split}
\end{equation*}
\textit{(ii).} For $L_{\mathrm{med}} = L_{\max}=T(N)^{-1}$, which is the minimal modulation for the corresponding time localization due to \eqref{eq:XkEstimateII},
\begin{equation*}
\begin{split}
&\quad \left| \int_{\R \times \Z^2} f_1(\tau_1,k_1) f_2(\tau_2,k_2) (k_{3,1} + k_{3,2}) f_3(\tau_3,k_3) (d\sigma_1)_1 (d\sigma_2)_1 \right| \\
&\lesssim N (L_1 L_2 L_3)^{1/2} T(N) \prod_{i=1}^3 \Vert f_i \Vert_{L_\tau^2 L^2_{(d k)_1}}.
\end{split}
\end{equation*}
This extends to the unsymmetrized equation by rational approximation. Let $\varepsilon > 0$. Consider $p_n, q_n \in \N$ with
\begin{equation*}
\left| \sqrt{3} - \frac{p_n}{q_n} \right| \leq \varepsilon.
\end{equation*}
As frequency modes for the unsymmetrized Zakharov-Kuznetsov equation choose $(0,2p_n)$, $(3q_n,-p_n)$, $(3q_n,p_n)$. Clearly, $\Phi = 0$. We find for the frequencies after symmetrization
\begin{equation*}
\begin{split}
k_1 &= ( \frac{2 p_n}{\sqrt{3}}, \frac{-2p_n}{\sqrt{3}}), \; k_2 = (3q_n - \frac{p_n}{\sqrt{3}}, 3q_n + \frac{p_n}{\sqrt{3}}), \\
k_3 &= (3q_n + \frac{p_n}{\sqrt{3}}, 3q_n - \frac{p_n}{\sqrt{3}}).
\end{split}
\end{equation*}
And we compute with $F$ from Section \ref{section:ShorttimeEnergyEstimates}, quantifying transversality, $F=18 q_n^2 - p_n^2 = O(N^2)$ full transversality of the frequency modes. Taking $p_n, q_n \to \infty$ yields the claim.
\end{remark}
\section{Norm inflation for complex-valued initial data}
\label{section:NormInflation}
In the following we give two examples of complex-valued initital data exhibiting norm inflation. We have already mentioned that this in sharp contrast to the $\R^2$ case as also for complex-valued initial data local well-posedness was proved in \cite{Kinoshita2019} for $s>-1/4$. The below considerations are inspired by \cite{Christ2004}, where related examples were considered for a quadratic Schr\"odinger equation with derivative nonlinearity.

The following initial data will give rise to norm inflation in any Sobolev space. However, it seems to be highly pathological:
\begin{equation}
u_0(t,x,y) = iA + B e^{i Nx}.
\end{equation}
From the form of the interaction it is easy to infer that the excited modes are precisely of the form $(Nk,0)$ for $k \in \mathbb{N}_0$. And moreover, the Fourier coefficient for $(N,0)$ satisfies the differential equation (due to vanishing resonance):
\begin{equation*}
\partial_t \hat{u}(t,N,0) = NA \hat{u}(t,N,0),
\end{equation*}
which yields the exponential growth $\hat{u}(t,N,0) = e^{t NA} iB$.
Setting $A = \varepsilon$ and $B = \varepsilon N^{-s}$, we easily see that $\Vert u_0 \Vert_{H^s} \lesssim \varepsilon$, however already for time-scales $\varepsilon$ the $H^s$-norm is bounded from below by $\gtrsim \varepsilon^{-1}$.

One can avoid the zero Fourier coefficient for another family of frequency modes with vanishing resonance, namely $(0,2), \; (N,-1), \; (N,1)$ (cf. \cite{LinaresPantheeRobertTzvetkov2019}) to find norm inflation for the initial data
\begin{equation*}
u_0(x,y) = A e^{2iy} + B e^{i( Nx -y)}.
\end{equation*}
By the above means, we infer that
\begin{equation*}
\hat{u}(t,N,-1) = (-iN) t A B,
\end{equation*}
and $\Vert u \Vert_{H^s} \gtrsim \varepsilon^2 N |t|$ again giving norm inflation, e.g., for $t = \varepsilon$, $N=\varepsilon^{-4}$.

\section*{Acknowledgements}

The first author acknowledges support by the German Research Foundation (DFG) through the CRC 1283, and the second author is supported through postdoctoral start-up funding by the DFG following his graduation within the IRTG 2235. We would like to thank Professor Sebastian Herr for helpful discussions on Loomis-Whitney-type inequalities and Dr. Tristan Robert for useful remarks on diophantine approximation.

\bibliographystyle{plain}

\end{document}